\documentclass[11pt,leqno]{article}
\usepackage{amssymb,amsfonts}
\usepackage{amsmath,amsthm,amsxtra}
\usepackage[all]{xy}

\usepackage{color}
\usepackage{verbatim}

\newcommand\bigcheck[1]{#1 \raise1ex\hbox{$\hspace{-1ex}{}^\vee$}}
\newcommand\sucheck[1]{#1 \raise0.5ex\hbox{$\hspace{-1ex}{}^\vee$}}

\newcommand{\alphaparenlist}{
  \renewcommand{\theenumi}{\alph{enumi}}%
  \renewcommand{\labelenumi}{(\theenumi)}%
}

\newcommand{\Der}{{\rm Der}}

\newcommand{\End}{\mathop{\rm End }}

\newcommand{\im}{\mathop{\rm Im  \, }}
\newcommand{\Ker}{\mathop{\rm Ker \, }}

\newcommand{\Hom}{\mathop{\rm Hom }}

\newcommand{\chom}{\mathop{\rm Chom }}
\newcommand{\cend}{\mathop{\rm Cend }}
\newcommand{\cder}{\mathop{\rm Cder }}
\newcommand{\Tor}{\mathop{\rm Tor }}

\renewcommand{\hat}{\widehat}

\newcommand{\Vir}{\mathop{\rm Vir }}

\newcommand{\incl}[1][r]{\ar@<-0.24pc>@{^(-}[#1] \ar@<+0.22pc>@{-}[#1]}
\newcommand{\equal}[1][r]{\ar@<-0.2pc>@{^-}[#1] \ar@<+0.2pc>@{-}[#1]}

\makeatletter
\renewcommand\section{\@startsection {section}{1}{\z@}%
                                   {-3.5ex \@plus -1ex \@minus -.2ex}%
                                   {2.3ex \@plus.2ex}%
                                   {\normalfont\large\bfseries}}
\renewcommand\subsection{\@startsection{subsection}{2}{\z@}%
                                     {-3.25ex\@plus -1ex \@minus -.2ex}%
                                     {0ex \@plus .0ex}%
                                     {\normalfont\normalsize\bfseries}}

\@addtoreset{equation}{section}
\makeatother

\newtheorem{theorem}{Theorem}[section]
\newtheorem{lemma}[theorem]{Lemma}
\newtheorem{corollary}[theorem]{Corollary}
\newtheorem{proposition}[theorem]{Proposition}
\newtheorem*{lemma*}{Lemma}

\theoremstyle{definition}
\newtheorem{definition}[theorem]{Definition}

\theoremstyle{remark}
\newtheorem{remark}[theorem]{Remark}
\newtheorem{example}[theorem]{Example}

\makeatletter
\def\@maketitle{\newpage
 \null
 \vskip 2em
 \begin{center}%
  \vskip 3em
  {\Large\bf \@title \par}%
  \vskip 1.5em
  {\normalsize
   \lineskip .5em
   \begin{tabular}[t]{c}\@author
   \end{tabular}\par}%
  \vskip 2em

 \end{center}%
 \par
 \vskip 2.5em}
\makeatother

\newcommand{\mc}[1]{{\mathcal #1}}
\newcommand{\mf}[1]{{\mathfrak #1}}
\newcommand{\mb}[1]{{\mathbb #1}}

\newcommand\tint{{\textstyle\int}}

\newcommand{\id}{{1 \mskip -5mu {\rm I}}}

\renewcommand{\tilde}{\widetilde}

\definecolor{light}{gray}{.9}

\input xy
\xyoption{all} \CompileMatrices

\begin{document}


\begin{center}
{\Large {\bf
Lie conformal algebra cohomology
and the variational complex
}
}

\vspace{20pt}

{\large
\begin{tabular}[t]{c}
$\mbox{Alberto De Sole}^{1}\phantom{m} and
\phantom{mm}
\mbox{Victor G. Kac}^{2}$
\\
\end{tabular}
\par
}

\bigskip

{\small
\begin{tabular}[t]{ll}
{1} &  {\it Dipartimento di Matematica, Universit\'a di Roma ``La Sapienza"} \\
&  Citt\'{a} Universitaria, 00185 Roma, Italy \\
& E-mail: {\tt desole@mat.uniroma1.it} \\
{2} & Department of Mathematics, MIT \\
& 77 Massachusetts Avenue, Cambridge, MA 02139, USA \\
& E-mail: {\tt kac@math.mit.edu}
\end{tabular}
}
\end{center}

\vspace{4pt}

\begin{center}
\emph{Dedicated to Corrado De Concini on his 60-th birthday.}
\end{center}

\vspace{2pt}

\begin{abstract}
\noindent 
We find an interpretation of the complex of variational calculus 
in terms of the Lie conformal algebra cohomology theory.
This leads to a better understanding of both theories.
In particular, we give 
an explicit construction of the Lie conformal algebra cohomology complex,
and endow it with a structure of a $\mf g$-complex.
On the other hand,
we give an explicit construction of the complex of variational
calculus in terms of skew-symmetric poly-differential operators.
\end{abstract}

\section*{Introduction.}
\label{sec:intro}

Lie conformal algebras encode the properties of operator product
expansions in conformal field theory, and, at the same time, of
local Poisson brackets in the theory of integrable evolution
equations.

Recall \cite{K} that a \emph{Lie conformal algebra} over a field
$\mb F$ is an $\mb F
[\partial]$-module $A$, endowed with a $\lambda$-\emph{bracket},
that is an $\mb F$-linear map $A \otimes A \to \mb F [\lambda]
\otimes A$ denoted by $a \otimes b \mapsto [a_\lambda b]$,
satisfying the two \emph{sesquilinearity} properties
\begin{equation}
  \label{eq:0.1}
  [\partial a_\lambda b] = - \lambda [a_\lambda b]\, , \qquad
  [a_\lambda \partial b] = (\partial + \lambda) [a_\lambda b]\, ,
\end{equation}
such that the \emph{skew-symmetry}
\begin{equation}
  \label{eq:0.2}
  [a_\lambda b] = - [b_{-\partial -\lambda}a]
\end{equation}
and the \emph{Jacobi identity}
\begin{equation}
  \label{eq:0.3}
  [a_\lambda [b_\mu c]] - [b_\mu [a_\lambda c]] 
     = [[a_\lambda b]_{\lambda +\mu}c]
\end{equation}
hold for any $a,b,c \in A$.
It is assumed in (\ref{eq:0.2}) that $\partial$ is moved to the left.

A module over a Lie conformal algebra~$A$ is an $\mb F
[\partial]$-module~$M$, endowed with a $\lambda$-\emph{action},
that is an $\mb F$-linear map $A \otimes M \to \mb F [\lambda]\otimes
M$, denoted by $a \otimes b \to a_\lambda b$, such that
sesquilinearity (\ref{eq:0.1}) holds for $a \in A$, $b \in M$ and
Jacobi identity (\ref{eq:0.3}) holds for $a,b \in A$, $c \in M$.

A cohomology theory for Lie conformal algebras was developed in
\cite{BKV}.  Given a Lie conformal algebra~$A$ and an $A$-module~$M$,
one first defines the \emph{basic cohomology complex}
$\tilde{\Gamma}^\bullet (A,M) = \sum_{k \in \mb Z_+} \tilde{\Gamma}^k$,
where $\tilde{\Gamma}^k$ consists of $\mb F$-linear maps
$\tilde{\gamma}:A^{\otimes k} \to \mb F [\lambda_1,\ldots
,\lambda_{k}]\otimes M$, satisfying certain sesquilinearity and
skew-symmetry properties, and endows this complex with a
differential $\delta:\,\tilde\Gamma^k\to\tilde\Gamma^{k+1}$, 
such that $\delta^2 =0$.  This complex is isomorphic to the
Lie algebra cohomology complex for the annihilation Lie algebra $\mf g_-$
of~$A$ with coefficients in the $\mf g_-$-module~$M$ \cite[Theorem~6.1]{BKV}.

Next, one endows $\tilde{\Gamma}^\bullet (A,M)$ with a structure of a
$\mb F [\partial]$-module, such that $\partial$ commutes
with~$\delta$, which  allows one to define the reduced
cohomology complex 
$\Gamma^\bullet (A,M) = \tilde{\Gamma}^\bullet
(A,M) / \partial \tilde{\Gamma}^\bullet (A,M)$, 
and this is the Lie
conformal algebra cohomology complex, introduced in 
\cite{BKV}.

Our first contribution to this theory is 
a more explicit construction of the reduced cohomology
complex.  Namely, we introduce a new cohomology complex
$C^\bullet (A,M)=\oplus_{k \in \mb Z_+} C^k$, where $C^0 =
M/\partial M$, $C^1 = \Hom_{\mb F [\partial]} (A,M)$, and for $k
\geq 2$, $C^k$ consists of poly $\lambda$-brackets,
namely
of $\mb F$-linear maps $c: A^{\otimes k}
\to \mb F [\lambda_1,\cdots,\lambda_{k-1}]\otimes M$, satisfying
certain sesquilinearity and skew-symmetry conditions, and we endow
$C^\bullet (A,M)$ with a square zero differential~$d$.  We
construct embeddings of complexes:
\begin{equation}
\label{eq:0.3a} 
 \Gamma^\bullet (A,M) \subset \bar{C}^\bullet (A,M)\subset C^\bullet (A,M)\, ,
\end{equation}
where $\bar{C}^\bullet (A,M)$ consists of cocycles which vanish
if one of the arguments is a torsion element of~$A$.  In
fact, $\bar{C}^k = C^k$, unless $k=1$.

We show that $\Gamma^\bullet (A,M) = \bar{C}^\bullet (A,M)$,
provided that, as an $\mb F [\partial]$-module, $A$~is isomorphic
to a direct sum of its  torsion and a free $\mb F [\partial]$-module
(which is always the case if~$A$ is a finitely generated $\mb F
[\partial]$-module).  Our opinion is that the slightly larger
complex $C^\bullet (A,M)$ is a more correct Lie conformal algebra
cohomology complex than the complex $\Gamma^\bullet (A,M)$ of
\cite{BKV}.  This is illustrated by our Theorem~3.1(c), which says
that the $\mb F [\partial]$-split abelian extensions of~$A$ by~$M$
are parameterized by $H^2 (A,M)$ for the complex $C^\bullet
(A,M)$.  This holds for the cohomology theory of
\cite{BKV} only if~$A$ is a free $\mb F [\partial]$-module.

Following \cite{BKV}, we also consider the superspace of basic chains
$\tilde{\Gamma}_\bullet (A,M)$ and its subspace of reduced chains
$\Gamma_\bullet (A,M)$ (they are not complexes in general).
Corresponding to the embeddings of complexes (\ref{eq:0.3a}),
we introduce the vector superspaces of chains $C_\bullet (A,M)$ and $\bar{C}
_\bullet (A,M)$, and the maps:
\begin{equation}
\label{eq:0.3b}  
C_\bullet (A,M) \twoheadrightarrow \bar{C}_\bullet (A,M)\to \Gamma_\bullet (A,M)\, .
\end{equation}

We develop the theory further in the important for the calculus
of variations case, when the $A$-module $M$ is endowed with a
commutative associative product, such that $\partial$ and
$a_\lambda$ for all $a \in A$ are derivations of this product.
In this case one can endow the superspace $\tilde{\Gamma}^\bullet (A,M)$ 
with a commutative associative product \cite{BKV}.
Furthermore, we introduce a Lie algebra bracket on the space 
$\mf g :=\Pi \tilde{\Gamma}_1 (A,M)$
($\Pi$, as usual, stands for reversing of the parity).
Let $\hat{\mf g} =\eta \mf g \oplus \mf g \oplus \mb F \partial_\eta$ be a
$\mb Z$-graded Lie superalgebra extension of~$\mf g$, where $\eta$ is an
odd indeterminate, $\eta^2 =0$.  We endow $\tilde{\Gamma}^\bullet(A,M)$ 
with a structure of a $\mf g$-\emph{complex}, which is
a $\mb Z$-grading preserving Lie superalgebra
homomorphism $\varphi : \hat{\mf g} \to \End_{\mb F}
\tilde{\Gamma}^\bullet (A,M)$, such that $\varphi (\partial_\eta)
= \delta$.  We also show that $\varphi (\hat{\mf g})$ lies in the
subalgebra of derivations of the superalgebra $\tilde{\Gamma}^\bullet
(A,M)$.  For each $X \in \mf g$ we thus have the Lie derivative
$L_X = \varphi (X)$ and the contraction operator $\iota_X
=\varphi (\eta X)$, satisfying all usual relations, in particular,
the Cartan formula $L_X = \iota_X \delta + \delta \iota_X$.

Denoting by $\mf g^\partial$ the centralizer of $\partial$ in
$\mf g$, we obtain the induced structure of a $\mf g^\partial$-complex
for $\Gamma^\bullet (A,M)$, which we, furthermore, extend to the
larger complex $C^\bullet (A,M)$.  Namely, we introduce a
canonical Lie algebra bracket on all spaces of $1$-chains with
reversed parity (see (\ref{eq:0.3b})), so that all the maps 
$\Pi C_1  \twoheadrightarrow \Pi
\bar{C}_1  \to \Pi \Gamma_1 \hookrightarrow \Pi \tilde{\Gamma}_1$\
are Lie algebra homomorphisms, and the embeddings (\ref{eq:0.3a}) are
morphisms of complexes, endowed with a corresponding Lie algebra structure.

What does it all have to do with the calculus of variations?
In order to explain this, introduce the notion of an
\emph{algebra of differentiable functions} (in
$\ell$~variables).  This is a differential algebra, i.e.,~a
unital commutative associative algebra~$\mc V$ with a derivation
$\partial$, endowed with commuting derivations
$\frac{\partial}{\partial u_i^{(n)}}$, $i \in I = \{ 1,\ldots ,\ell
\}$, $n \in \mb Z_+$, such that only a finite number of
$\frac{\partial f}{\partial u_i^{(n)}}$ are non-zero for each $f
\in \mc V$, and the following commutation rules with~$\partial$ hold:
\begin{equation}
  \label{eq:0.4}
  \left[  \frac{\partial}{\partial u_i^{(n)}}, \partial \right] =
       \frac{\partial}{\partial u_i^{(n-1)}}
\,\,\,\,
\text{ (the RHS is $0$ if $n=0$) }\,.
\end{equation}

The most important example is the algebra of differential
polynomials $\mb F [u_i^{(n)} | i \in I$, $n \in \mb Z_+]$
with $\partial (u_i^{(n)}) = u^{(n+1)}_i$,
$n \in \mb Z_+$, $i \in I$.  Other examples include any
localization by a multiplicative subset or any algebraic
extension of this algebra.  

The basic de~Rham complex
$\tilde{\Omega}^\bullet = \tilde{\Omega}^\bullet (\mc V)$ over $\mc V$ is
defined as an exterior superalgebra over the free $\mc V$-module
$\tilde{\Omega}^1 = \sum_{i \in I \, , \, n \in \mb Z_+} \mc V \delta
u^{(n)}_i$ on generators $\delta u^{(n)}_i$ with odd parity.  We
have:  $\tilde{\Omega}^\bullet = \bigoplus_{k \in \mb Z_+}
\tilde{\Omega}^k$, where $\tilde{\Omega}^0 =\mc V$,
$\tilde{\Omega}^k = \Lambda^k_{\mc V} \tilde{\Omega}^1$.  This
$\mb Z$-graded superalgebra is endowed by an odd derivation $\delta$ of
degree~$1$, such that $\delta f = \sum_{i \in I \, , \, n \in
  \mb Z_+} \frac{\partial f}{\partial u^{(n)}_i} \delta u^{(n)}_i$ for 
$f \in \tilde{\Omega}^0$ and
$\delta (\delta u^{(n)}_i) =0$.  One easily checks that $\delta^2
=0$, so that $\tilde{\Omega}^\bullet$ is a cohomology complex.

Let $\mf g$ be the Lie algebra of derivations of the algebra $\mc V$ of
the form
\begin{equation}
  \label{eq:0.5}
  X=\sum_{i \in I \, , \, n \in \mb Z_+} P_{i,n} 
     \frac{\partial}{\partial u^{(n)}_i}\, , \quad
     \mbox{\,\, where \,\,} P_{i,n} \in \mc V\, .
\end{equation}
To any such derivation $X$ we associate 
an even derivation $L_X$ (Lie derivative) and an odd derivation $\iota_X$ (contraction) of the
superalgebra $\tilde{\Omega}^\bullet$ by letting 
${L_X|}_{\mc V}=X,\,
L_X (\delta u^{(n)}_i) = \delta P_{i,n},\,
{\iota_X|}_{\mc V}=0,\,
\iota_X (\delta u^{(n)}_i) = P_{i,n}$.  
This provides $\tilde{\Omega}^\bullet$
with a structure of a $\mf g$-complex,
by letting $\varphi(X)=L_X$ and $\varphi(\eta X)=\iota_X$.  
Also, the
derivation~$\partial$ extends to an (even) derivation of
$\tilde{\Omega}^\bullet$ by letting $\partial (\delta u^{(n)}_i)
= \delta u^{(n+1)}_i$.

It is easy to check, using (\ref{eq:0.4}), that $\partial$ and $\delta$
commute, hence we can consider the reduced complex
\begin{displaymath}
\Omega^\bullet (\mc V) = \tilde{\Omega}^\bullet(\mc V) / 
     \partial \tilde{\Omega}^\bullet (\mc V) \, ,   
\end{displaymath}
which is called the \emph{variational complex}.  This is, of
course, a $\mf g^\partial$-complex.

Our main observation  is the interpretation of the variational
complex $\Omega^\bullet (\mc V)$ in terms of Lie conformal algebra
cohomology, given by Theorem~0.1 below.

Let $R=\bigoplus_{i \in I} \mb F [\partial]u_i$ be a free $\mb F
[\partial]$-module of rank~$\ell$, endowed with the trivial
$\lambda$-bracket $[a_\lambda b]=0$ for all $a,b \in R$.
Let~$\mc V$ be an algebra of differentiable functions.  
We endow $\mc V$ with
the structure of an $R$-module by letting
\begin{displaymath}
  u_{i_\lambda} f = 
  \sum_{n\in\mb Z_+} \lambda^n \frac{\partial f}{\partial u_i^{(n)}}\, , \qquad
  i \in I \, , 
\end{displaymath}
and extending to $R$ by sesquilinearity.  Let $\mf g$ be the Lie
algebra of derivations of $\mc V$ of the form (\ref{eq:0.5}), and let
$\mf g^\partial$ be the subalgebra of $\mf g$, consisting of
derivations commuting with $\partial$.
\begin{theorem}
  \label{th:0.1}
The $\mf g^\partial$-complexes $C^\bullet (R,\mc V)$ and
$\Omega^\bullet (\mc V)$ are isomorphic.
\end{theorem}

As a result, we obtain the following interpretation of the
complex $\Omega^\bullet (\mc V)$, which explains the name ``calculus
of variations''.

We have:  $\Omega^0 = \mc V /\partial \mc V$, $\Omega^1 = \Hom_{\mb F [\partial]} (R,\mc V) 
= \mc V^{\oplus\ell}$.  Elements of
$\Omega^0$ are called local functionals and the image of $f
\in \mc V$ in $\Omega^0$ is denoted by $\int f$.  Elements of
$\Omega^1$ are called local $1$-forms.  The differential $\delta:\,\Omega^0 \to \Omega^1$ 
is identified with the variational derivative:
$\delta \int f = \bigg( \frac{\delta \int f}{\delta  u_i}\bigg)_{i \in I} = \frac{\delta f}{\delta u}$,
where
\begin{equation}\label{eq:dic15_3}
\frac{\delta f}{\delta u_i} = \sum_{n \in \mb Z_+}(-\partial)^n
        \frac{\partial f}{\partial u^{(n)}_i}\, .
\end{equation}

Furthermore, the space of $2$-cochains $C^2$ is identified with
the space of skew-adjoint differential operators by associating
to the $\lambda$-bracket 
$\{\cdot\,_\lambda\,\cdot\}:\, R^{\otimes 2} \to \mb F [\lambda] \otimes \mc V$ 
the $\ell \times \ell$ matrix $S_{ij} (\partial) = 
\{{u_j}_\partial u_i\}_\to$, where the arrow means that $\partial$ is moved
to the right.  The differential $\delta : \Omega^1 \to \Omega^2$
is expressed in terms of the Frechet derivative
\begin{equation}\label{eq:dic15_2}
  D_F (\partial)_{ij} = \sum_{n \in \mb Z_+}
     \frac{\partial F_i}{\partial u^{(n)}_j} \partial^n\, , 
     \qquad   i,j \in I \, ,
\end{equation}
which defines an $\mb F$-linear map: $\mc V^{\ell}\to\mc V^{\oplus\ell}$.  Namely:
$\delta F = D_F (\partial) - D_F (\partial)^*$.  
The subspace of closed $2$-cochains in $C^2$ is identified with the space
of symplectic differential operators.

A $2$-cochain, which is a skew-adjoint differential operator $S_{ij} (\partial)$, 
can be identified with the  corresponding $\mb F$-linear map $(\mc V^\ell)^2\to\mc V/\partial\mc V$,
of  ``differential type'',  given by
$$
S(P,Q)
\,=\, \int \sum_{i,j\in I} Q_iS_{ij}(\partial)P_j\,.
$$
Skew-adjointness of $S$ translates to the skew-symmetry condition $S(P,Q)=-S(Q,P)$.

More generally, the space of $k$-cochains $C^k$ 
for $k \geq 2$ is identified with the space of
all skew-symmetric $\mb F$-linear maps 
$S:\,(\mc V^\ell)^k\to\mc V/\partial\mc V$, 
of ``differential type'':
$$
S(P^1,\cdots,P^k)
\,=\, \int \sum_{\substack{i_1,\cdots,i_{k}\in I \\ n_1,\cdots,n_{k}\in\mb Z_+}}
f^{n_1,\cdots,n_k}_{i_1,\cdots,i_k}
(\partial^{n_1}P^1_{i_1})\cdots(\partial^{n_k}P^k_{i_k})
\,\,,\,\,\,\,
\text{ where } f^{n_1,\cdots,n_k}_{i_1,\cdots,i_k}\in\mc V\,.
$$
The skew-symmetry condition is simply
$S(P^1,\cdots,P^k)=
\text{sign}(\sigma)S(P^{\sigma(1)},\cdots,P^{\sigma(k)})$, for every $\sigma \in S_k$.
The subspace of closed $k$-cochains for $k \geq 2$ is the subspace 
of ``symplectic'' $k-1$-differential operators.

We prove in \cite{BDK} that the cohomology $H^j$ of the complex
$\Omega^\bullet (\mc V)$ is zero for $j \geq 1$ and 
$H^0 =\mc C/(\mc C\cap\partial\mc V)$, 
where $\mc C:=\{f\in\mc V\,|\,\frac{\partial f}{\partial u_i^{(n)}}\,\forall i\in I,n\in\mb Z_+\}$, 
provided that $\mc V$ is normal, as defined in Section \ref{sec:final}.
(Any algebra of differentiable functions can be included in a normal one.)
As a corollary, we obtain (cf.~\cite{D}) 
that $\Ker \frac{\delta}{\delta u} = \partial \mc V + \mc C$, 
and $F \in \im \frac{\delta}{\delta u}$ iff $D_F(\partial)$ is a self-adjoint
differential operator, provided that
$\mc V$ is normal. The first result can be found in \cite{D}
(see also \cite{Di} and \cite{Vi}, where it is proved under stronger conditions
on $\mc V$), but it is certainly much older. The second result, at least under 
stronger conditions on $\mc V$, goes back to \cite{H}, \cite {V}. 
We also obtain the classification of symplectic differential operators
(cf. \cite{D}) and of
symplectic poly-differential operators for normal~$\mc V$, which
seems to be a new result.

Thus, the interaction between the Lie conformal algebra
cohomology and the variational calculus has led to  progress in
both theories.  On the one hand, the variational calculus
motivated some of our constructions 
in the Lie conformal algebra cohomology.  On the other hand, the Lie
conformal algebra cohomology interpretation of the variational
complex has led to a better understanding of this complex and to
a classification of symplectic and poly-symplectic differential
operators.  

The ground field is an arbitrary field $\mb F$ of
characteristic $0$.

\section{Lie conformal algebra cohomology complexes.}\label{sec:1}

\subsection{The basic cohomology complex $\tilde\Gamma^\bullet$ 
and the reduced cohomology complex $\Gamma^\bullet$.}~~
\label{sec:1.1}
Let us review, following \cite{BKV}, the definition of the basic 
and reduced cohomology complexes
associated to a Lie conformal algebra $A$ and an $A$-module $M$.
A $k$-\emph{cochain} of $A$ with coefficients in $M$ is an $\mb F$-linear map
$$
\tilde\gamma:\, A^{\otimes k}\to \mb F[\lambda_1,\dots,\lambda_{k}]\otimes M\,\,,\,\,\,\,
a_1\otimes\cdots\otimes a_{k}\mapsto \tilde\gamma_{\lambda_1,\cdots,\lambda_{k}}
(a_1,\cdots,a_{k})\,,
$$
satisfying the following two conditions:
\begin{enumerate}
\item[A1.] $\tilde\gamma_{\lambda_1,\cdots,\lambda_{k}}(a_1,\cdots,\partial a_i,\cdots,a_{k})=
-\lambda_i\tilde\gamma_{\lambda_1,\cdots,\lambda_{k}}(a_1,\cdots,a_{k})$ for all $i$,
\item[A2.] $\tilde\gamma$ is skew-symmetric with respect to simultaneous permutations
of the $a_i$'s and the $\lambda_i$'s.
\end{enumerate}
\begin{remark}\label{rem:anniv}
Note that condition A1. implies that $\tilde\gamma_{\lambda_1,\cdots,\lambda_k}(a_1,\cdots,a_{k})$
is zero if one of the elements $a_i$ is a torsion element of the $\mb F[\partial]$-module $A$.
\end{remark}

We let $\tilde\Gamma^k=\tilde\Gamma^k(A,M)$ be the space of all $k$-cochains, 
and $\tilde\Gamma^\bullet=\tilde\Gamma^\bullet(A,M)=\bigoplus_{k\geq0}\tilde\Gamma^k$.
The differential $\delta$ of a $k$-cochain $\tilde\gamma$ is defined by the following formula:
\begin{eqnarray}\label{eq:july24_7}
(\delta\tilde\gamma)_{\lambda_1,\cdots,\lambda_{k+1}}(a_1,\cdots,a_{k+1})
=\sum_{i=1}^{k+1} (-1)^{i+1} {a_i}_{\lambda_i}
\Big(\tilde\gamma_{\lambda_1,\stackrel{i}{\check{\cdots}},\lambda_{k+1}}
(a_1,\stackrel{i}{\check{\cdots}},a_{k+1}) \Big) \\
+ \sum_{\substack{i,j=1\\i<j}}^{k+1} (-1)^{k+i+j+1} 
\tilde\gamma_{\lambda_1,\stackrel{i}{\check{\cdots}}\stackrel{j}{\check{\cdots}},
\lambda_{k+1},\lambda_i+\lambda_j}
(a_1,\stackrel{i}{\check{\cdots}}\stackrel{j}{\check{\cdots}},a_{k+1},[{a_i}_{\lambda_i} a_j])\,. \nonumber
\end{eqnarray}
One checks that $\delta$ maps $\tilde\Gamma^k$ to $\tilde\Gamma^{k+1}$,
and that $\delta^2=0$.
The $\mb Z$-graded space $\tilde\Gamma^\bullet(A,M)$ with the differential $\delta$ 
is called the \emph{basic cohomology complex} associated to $A$ and $M$.

Define the structure of an $\mb F[\partial]$-module on $\tilde\Gamma^\bullet$ by letting
\begin{equation}\label{eq:july24_8}
(\partial\tilde\gamma)_{\lambda_1,\cdots,\lambda_{k}}(a_1,\cdots,a_{k})
=(\partial^M+\lambda_1+\cdots+\lambda_{k})
\Big(\tilde\gamma_{\lambda_1,\cdots,\lambda_{k}}
(a_1,\cdots,a_{k})\Big)\,,
\end{equation}
where $\partial^M$ denotes the action of $\partial$ on $M$.
One checks that $\delta$ and $\partial$ commute, and therefore 
$\partial\tilde\Gamma^\bullet\subset\tilde\Gamma^\bullet$ is a subcomplex.
We can consider the \emph{reduced cohomology complex}
$\Gamma^\bullet(A,M)=\tilde\Gamma^\bullet(A,M)/\partial\tilde\Gamma^\bullet(A,M)
=\bigoplus_{k\in\mb Z_+}\Gamma^k(A,M)$.
For example, $\Gamma^0=M/\partial^MM$, and we denote,
as in the calculus of variations, by $\tint m$ the image of $m\in M$ in $M/\partial^MM$.
As before we let, for brevity, $\Gamma^\bullet=\Gamma^\bullet(A,M)$ 
and $\Gamma^k=\Gamma^k(A,M),\,k\in\mb Z_+$.

In the following sections we will find a simpler construction 
of the reduced cohomology complex $\Gamma^\bullet$,
in terms of poly $\lambda$-brackets.

\vspace{3ex}
\subsection{Poly $\lambda$-brackets.}~~
\label{sec:1.3_a}
Let $A$ and $M$ be $\mb F[\partial]$-modules, and, as before, denote by 
$\partial^M$ the action of $\partial$ on $M$.
For $k\geq1$, a \emph{$k$-$\lambda$-bracket} on $A$ with coefficients in $M$ 
is, by definition, an $\mb F$-linear map
$c:\,A^{\otimes k}\to\mb F[\lambda_1,\dots,\lambda_{k-1}]\otimes M$,
denoted by
$$
a_1\otimes\cdots\otimes a_k\,\mapsto\,\{{a_1}_{\lambda_1}\cdots {a_{k-1}}_{\lambda_{k-1}} a_k\}_c\,,
$$
satisfying the following conditions:
\begin{enumerate}
\item[B1.] $\{{a_1}_{\lambda_1}\cdots (\partial a_i)_{\lambda_i}\cdots {a_{k-1}}_{\lambda_{k-1}} a_k\}_c
=-\lambda_i\{{a_1}_{\lambda_1}\cdots {a_{k-1}}_{\lambda_{k-1}} a_k\}_c$,
for $1\leq i\leq k-1$;
\item[B2.] $\{{a_1}_{\lambda_1}\cdots {a_{k-1}}_{\lambda_{k-1}} (\partial a_k)\}_c
=(\lambda_1+\cdots+\lambda_k+\partial^M)
\{{a_1}_{\lambda_1}\cdots {a_{k-1}}_{\lambda_{k-1}} a_k\}_c$;
\item[B3.] $c$ is skew-symmetric with respect to simultaneous permutations
of the $a_i$'s and the $\lambda_i$'s in the sense that,
for every permutation $\sigma$ of the indices $\{1,\dots,k\}$, we have:
$$
\{{a_1}_{\lambda_1}\cdots {a_{k-1}}_{\lambda_{k-1}} a_k\}_c
=\text{sign}(\sigma)
\{{a_{\sigma(1)}}_{\lambda_{\sigma(1)}}\cdots {a_{\sigma(k-1)}}_{\lambda_{\sigma(k-1)}} 
a_{\sigma(k)}\}_c\,\Big|_{\lambda_k\mapsto\lambda_k^\dagger}\,.
$$
The notation in the RHS means that $\lambda_k$ is replaced 
by $\lambda_k^\dagger=-\sum_{j=1}^{k-1}\lambda_j-\partial^M$,
if it occurs, and $\partial^M$ is moved to the left.
\end{enumerate}
\begin{remark}\label{rem:sep23}
A structure of a Lie conformal algebra on $A$ is a 2-$\lambda$-bracket on $A$
with coefficients in $A$, satisfying the Jacobi identity \eqref{eq:0.3}.
\end{remark}

We let $C^0=M/\partial^MM$ and, for $k\geq1$,
we denote by $C^k=C^k(A,M)$ the space of all $k$-$\lambda$-brackets
on $A$ with coefficients in $M$.
For example, $C^1$ is the space of all $\mb F[\partial]$-module homomorphisms 
$c:\,A\to M$.
We let $C^\bullet=\bigoplus_{k\in\mb Z_+}C^k$, 
the space of all \emph{poly} $\lambda$-\emph{brackets}.

We also define $\bar C^\bullet=\bigoplus_{k\in\mb Z_+}\bar C^k$,
where $\bar C^0=C^0=M/\partial^MM$,
and $\bar C^k\subset C^k$ is the subspace 
of  $k$-$\lambda$-brackets $c$ with the following additional property:
$\{{a_1}_{\lambda_1}\cdots {a_{k-1}}_{\lambda_{k-1}} a_k\}_c$ is zero
if one of the elements $a_i$ is a torsion element in $A$.
Clearly, $\bar C^1$ needs not be equal to $C^1$.
On the other hand, 
it is easy to check, using the sesquilinearity conditions B1. and B2.,
that $\bar C^k=C^k$ for $k\geq2$.

\vspace{3ex}
\subsection{The complex of poly $\lambda$-brackets.}~~
\label{sec:1.3_b}
We next define a differential $d$ on the space $C^\bullet$ of poly $\lambda$-brackets
such that $d(C^k)\subset C^{k+1}$ and $d^2=0$,
thus making $C^\bullet$ a cohomology complex.

For $\tint m\in C^0=M/\partial^MM$, we let $d\tint m\in C^1$ be the following
$\mb F[\partial]$-module homomorphism:
\begin{equation}\label{eq:d0}
\big(d\tint m\big)(a)\,
\Big(=\{a\}_{d\tint m}\Big)
\,:=\,
a_{-\partial^M}m\,.
\end{equation}
This is well defined since, if $m\in\partial^MM$, the RHS is zero due to sesquilinearity.
For $c\in C^k$, with $k\geq1$, we let
$dc\in C^{k+1}$ be the following poly $\lambda$-bracket:
\begin{eqnarray}\label{eq:d>}
&\displaystyle{
\{{a_1}_{\lambda_1}\cdots {a_{k}}_{\lambda_{k}} a_{k+1}\}_{dc} 
\,:=\,
\sum_{i=1}^k (-1)^{i+1} {a_i}_{\lambda_i}
\big\{{a_1}_{\lambda_1}\stackrel{i}{\check{\cdots}}{a_{k}}_{\lambda_{k}} a_{k+1}\big\}_{c} 
}\nonumber\\
&\displaystyle{
+ \sum_{\substack{i,j=1\\i<j}}^k (-1)^{k+i+j+1} 
\big\{{a_1}_{\lambda_1}\stackrel{i}{\check{\cdots}}\stackrel{j}{\check{\cdots}}
{a_{k}}_{\lambda_{k}} {a_{k+1}}_{\lambda_{k+1}^\dagger}[{a_i}_{\lambda_i} a_j]\big\}_{c} 
}\\
&\displaystyle{
+(-1)^k {a_{k+1}}_{\lambda_{k+1}^\dagger}
\big\{{a_1}_{\lambda_1}\cdots{a_{k-1}}_{\lambda_{k-1}} a_{k}\big\}_{c} 
+ \sum_{i=1}^k (-1)^{i}
\big\{{a_1}_{\lambda_1}\stackrel{i}{\check{\cdots}}{a_{k}}_{\lambda_{k}} 
[{a_i}_{\lambda_i} a_{k+1}]\big\}_{c}\,,
}\nonumber
\end{eqnarray}
where, as before,
$\lambda_{k+1}^\dagger=-\sum_{j=1}^k\lambda_j-\partial^M$, and $\partial^M$ is moved to the left.

For example, for an $\mb F[\partial]$-module homomorphism $c:\,A\to M$, we have
\begin{equation}\label{eq:july29_1}
\{a_\lambda b\}_{dc}
=a_\lambda c(b)-b_{-\lambda-\partial}c(a)-c([a_\lambda b])\,.
\end{equation}
\begin{proposition}\label{prop:anniv}
\begin{enumerate}
\alphaparenlist
\item For $c\in C^k$, we have $d(c)\in C^{k+1}$ and $d^2(c)=0$.
This makes $(C^\bullet,d)$ a cohomology complex.
\item $d(\bar C^k)\subset\bar C^{k+1}$ for all $k\geq0$. 
Hence $(\bar C^\bullet,d)$ is a cohomology subcomplex of $(C^\bullet,d)$.
\end{enumerate}
\end{proposition}
\begin{proof}
We prove part (b) first. For $k\geq1$ there is nothing to prove. For $k=0$
just notice that, if $\tint m\in M/\partial^MM$ and $a\in A$ is a torsion element, 
then, by \eqref{eq:d0}, we have $\big(d\tint m\big)(a)=0$,
since torsion elements of $A$ act trivially in any module \cite{K}.
Hence $d\tint m\in\bar C^1$.
In order to prove part (a) we have to check that, if $c\in C^k$,
then $dc$, defined by \eqref{eq:d0} and \eqref{eq:d>},
satisfies conditions B1., B2., B3., and $d(dc)=0$.
To simplify the arguments, we rewrite equation \eqref{eq:d>} in a more concise form:
\begin{eqnarray}\label{eq:d>>}
&\displaystyle{
\{{a_1}_{\lambda_1}\cdots {a_{k}}_{\lambda_{k}} a_{k+1}\}_{dc} 
\,:=\,
\bigg(
\sum_{i=1}^{k+1} (-1)^{i+1} {a_i}_{\lambda_i}
\big\{{a_1}_{\lambda_1}\stackrel{i}{\check{\cdots}}{a_{k}}_{\lambda_{k}} a_{k+1}\big\}_{c} 
}\nonumber\\
&\displaystyle{
+ \sum_{\substack{i,j=1\\i<j}}^{k+1} (-1)^{k+i+j+1} 
\big\{{a_1}_{\lambda_1}\stackrel{i}{\check{\cdots}}\stackrel{j}{\check{\cdots}}
{a_{k+1}}_{\lambda_{k+1}}[{a_i}_{\lambda_i} a_j]\big\}_{c} 
\bigg)\bigg|_{\lambda_{k+1}=\lambda_{k+1}^\dagger}\,,
}
\end{eqnarray}
where the RHS is evaluated at 
$\lambda_{k+1}=\lambda_{k+1}^\dagger=-\sum_{j=1}^k\lambda_j-\partial^M$,
with $\partial^M$ acting from the left.
The above equation should be interpreted by saying that, 
in the first term in the RHS, for $i=k+1$, 
the last index $\lambda_k$ does not appear in the poly $\lambda$-bracket.
Let us replace $a_h$ by $\partial a_h$ in equation \eqref{eq:d>>}.
It is not hard to check, using conditions B1. and B2. for $c$
and the sesquilinearity of the $\lambda$-action of $A$ on $M$,
that, for $1\leq h\leq k$,
each term in the RHS of \eqref{eq:d>>} gets multiplied by $-\lambda_h$,
while, for $h=k+1$, 
each term in the RHS of \eqref{eq:d>>} gets multiplied by 
$-\lambda_{k+1}^\dagger=\sum_{j=1}^k\lambda_j+\partial^M$.
Hence $dc$ satisfies conditions B1. and B2.
In order to prove condition B3., let $\sigma$ be a permutation of the set $\{1,\cdots,k+1\}$.
A basic observation is that, if we \emph{first} replace $\lambda_{\sigma(k+1)}$ 
by 
$\lambda_{\sigma(k+1)}^\dagger
=-\lambda_1-\stackrel{\sigma(k+1)}{\check{\cdots}}-\lambda_{k+1}-\partial^M$,
and \emph{then} we replace 
$\lambda_{k+1}$ by 
$\lambda_{k+1}^\dagger
=-\lambda_1\cdots-\lambda_k-\partial^M$,
as a result $\lambda_{\sigma(k+1)}$ stays unchanged.
Notice, moreover, that, for $1\leq i\leq k+1$, 
$\{\sigma(1),\stackrel{i}{\check{\cdots}},\sigma(k+1)\}$
is a permutation of $\{1,\stackrel{\sigma(i)}{\check{\cdots}},k+1\}$,
and its sign is $(-1)^{i+\sigma(i)}\text{sign}(\sigma)$.
Hence, using the assumption B3. on $c$, we get
\begin{eqnarray}\label{eq:paolo_1}
&& {a_{\sigma(i)}}_{\lambda_{\sigma(i)}}
\big\{{a_{\sigma(1)}}_{\lambda_{\sigma(1)}} \stackrel{i}{\check{\cdots}}
{a_{\sigma(k)}}_{\lambda_{\sigma(k)}} a_{\sigma(k+1)}\big\}_{c}\,
\Big|_{\lambda_{k+1}=\lambda_{k+1}^\dagger} \\
&& =
\text{sign}(\sigma)(-1)^{i+\sigma(i)}
{a_{\sigma(i)}}_{\lambda_{\sigma(i)}}
\big\{{a_1}_{\lambda_1} \stackrel{\sigma(i)}{\check{\cdots}}
{a_k}_{\lambda_k} a_{k+1}\big\}_{c}\,
\Big|_{\lambda_{k+1}=\lambda_{k+1}^\dagger}\,.\nonumber
\end{eqnarray}
Similarly, for the second term in \eqref{eq:d>>}, we notice that 
$\{\sigma(1),\stackrel{i}{\check{\cdots}}\stackrel{j}{\check{\cdots}},\sigma(k+1)\}$
is a permutation of $\{1,\stackrel{\sigma(i)}{\check{\cdots}}\stackrel{\sigma(j)}{\check{\cdots}},k+1\}$,
and its sign is $(-1)^{i+j+\sigma(i)+\sigma(j)}\text{sign}(\sigma)$ if $\sigma(i)<\sigma(j)$,
and it is $(-1)^{i+j+\sigma(i)+\sigma(j)+1}\text{sign}(\sigma)$ if $\sigma(i)>\sigma(j)$.
Hence, for $\sigma(i)<\sigma(j)$ we have
\begin{eqnarray}\label{eq:paolo_2}
&& 
\big\{{a_{\sigma(1)}}_{\lambda_{\sigma(1)}} 
\stackrel{i}{\check{\cdots}}\stackrel{j}{\check{\cdots}}
{a_{\sigma(k+1)}}_{\lambda_{\sigma(k+1)}} 
[{a_{\sigma(i)}}_{\lambda_{\sigma(i)}} a_{\sigma(j)}]
\big\}_{c}\,
\Big|_{\lambda_{k+1}=\lambda_{k+1}^\dagger} \\
&& 
= \text{sign}(\sigma)(-1)^{i+j+\sigma(i)+\sigma(j)}
\big\{{a_1}_{\lambda_1} 
\stackrel{\sigma(i)}{\check{\cdots}}\stackrel{\sigma(j)}{\check{\cdots}}
{a_{k+1}}_{\lambda_{k+1}} 
[{a_{\sigma(i)}}_{\lambda_{\sigma(i)}} a_{\sigma(j)}]
\big\}_{c}\,
\Big|_{\lambda_{k+1}=\lambda_{k+1}^\dagger} \,,\nonumber
\end{eqnarray}
while for $\sigma(i)>\sigma(j)$ we have, by the skew-symmetry of the $\lambda$-bracket in $A$,
\begin{eqnarray}\label{eq:paolo_3}
&& 
\big\{{a_{\sigma(1)}}_{\lambda_{\sigma(1)}} 
\stackrel{i}{\check{\cdots}}\stackrel{j}{\check{\cdots}}
{a_{\sigma(k+1)}}_{\lambda_{\sigma(k+1)}} 
[{a_{\sigma(i)}}_{\lambda_{\sigma(i)}} a_{\sigma(j)}]
\big\}_{c}\,
\Big|_{\lambda_{k+1}=\lambda_{k+1}^\dagger} \nonumber\\
&& 
= \text{sign}(\sigma)(-1)^{i+j+\sigma(i)+\sigma(j)}
\big\{{a_1}_{\lambda_1} 
\stackrel{\sigma(j)}{\check{\cdots}}\stackrel{\sigma(i)}{\check{\cdots}}
{a_{\sigma(k+1)}}_{\lambda_{k+1}} 
[{a_{\sigma(j)}}_{-\lambda_{\sigma(i)}-\partial} a_{\sigma(i)}]
\big\}_{c}\,
\Big|_{\lambda_{k+1}=\lambda_{k+1}^\dagger} \\
&&
= \text{sign}(\sigma)(-1)^{i+j+\sigma(i)+\sigma(j)}
\big\{{a_1}_{\lambda_1} 
\stackrel{\sigma(j)}{\check{\cdots}}\stackrel{\sigma(i)}{\check{\cdots}}
{a_{\sigma(k+1)}}_{\lambda_{k+1}} 
[{a_{\sigma(j)}}_{\lambda_{\sigma(j)}} a_{\sigma(i)}]
\big\}_{c}\,
\Big|_{\lambda_{k+1}=\lambda_{k+1}^\dagger} \,.\nonumber
\end{eqnarray}
In the last identity we used the assumption that $c$ satisfies condition B2.
Clearly, equations \eqref{eq:paolo_1}, \eqref{eq:paolo_2} and \eqref{eq:paolo_3},
together with the definition \eqref{eq:d>>} of $dc$,
imply that $dc$ satisfies condition B3.
We are left to prove that $d^2c=0$.
We have, by \eqref{eq:d>>},
\begin{eqnarray}\label{eq:fe_1}
&\displaystyle{
\{{a_1}_{\lambda_1}\cdots {a_{k+1}}_{\lambda_{k+1}} a_{k+2}\}_{d^2c} 
\,=\,
\bigg(
\sum_{i=1}^{k+2} (-1)^{i+1} {a_i}_{\lambda_i}
\big\{{a_1}_{\lambda_1}\stackrel{i}{\check{\cdots}}{a_{k+1}}_{\lambda_{k+1}} a_{k+2}\big\}_{dc} 
}\nonumber\\
&\displaystyle{
+ \sum_{\substack{i,j=1\\i<j}}^{k+2} (-1)^{k+i+j} 
\big\{{a_1}_{\lambda_1}\stackrel{i}{\check{\cdots}}\stackrel{j}{\check{\cdots}}
{a_{k+2}}_{\lambda_{k+2}}[{a_i}_{\lambda_i} a_j]\big\}_{dc} 
\bigg)\bigg|_{\lambda_{k+2}=\lambda_{k+2}^\dagger}\,,
}
\end{eqnarray}
where, in the RHS, we replace $\lambda_{k+2}$ by 
$\lambda_{k+2}^\dagger=-\sum_{j=1}^{k+1}\lambda_j-\partial^M$,
and $\partial^M$ is moved to the left.
Again by \eqref{eq:d>>} and by sesquilinearity of the $\lambda$-action of $A$ on $M$,
the first term in the RHS of \eqref{eq:fe_1} is
\begin{eqnarray}\label{eq:fe_2}
&&\displaystyle{
\Bigg(
\sum_{\substack{i,j=1\\i\neq j}}^{k+2} (-1)^{i+j} \epsilon(i,j)
{a_j}_{\lambda_j} \Big({a_i}_{\lambda_i}
\big\{{a_1}_{\lambda_1}
\stackrel{i}{\check{\cdots}}\stackrel{j}{\check{\cdots}}
{a_{k+1}}_{\lambda_{k+1}} a_{k+2}\big\}_{c} \Big)
}\\
&&\displaystyle{
+ \sum_{\substack{i,j,h=1\\i<j \\ i,j\neq h}}^{k+2} (-1)^{k+i+j+h} \epsilon(i,h)\epsilon(j,h)
{a_h}_{\lambda_h}
\big\{{a_1}_{\lambda_1}\stackrel{i}{\check{\cdots}}\stackrel{j}{\check{\cdots}}\stackrel{h}{\check{\cdots}}
{a_{k+2}}_{\lambda_{k+2}}[{a_i}_{\lambda_i} a_j]\big\}_{c} 
\Bigg)\Bigg|_{\lambda_{k+2}=\lambda_{k+2}^\dagger}\,,
}\nonumber
\end{eqnarray}
where $\epsilon(i,j)$ is $+1$ if $i<j$ and $-1$ if $i>j$.
Similarly, by \eqref{eq:d>} the second term in the RHS of \eqref{eq:fe_1} is
\begin{eqnarray}\label{eq:fe_3}
&\displaystyle{
\bigg(
\sum_{\substack{i,j,h=1\\i<j \\ i,j\neq h}}^{k+2} (-1)^{k+i+j+h+1} \epsilon(h,i)\epsilon(h,j)
{a_h}_{\lambda_h}
\big\{{a_1}_{\lambda_1}
\stackrel{h}{\check{\cdots}}\stackrel{i}{\check{\cdots}}\stackrel{j}{\check{\cdots}}
{a_{k+2}}_{\lambda_{k+2}} [{a_i}_{\lambda_i}a_j]\big\}_{c} 
}\nonumber\\
&\displaystyle{
+\sum_{\substack{i,j=1\\i< j}}^{k+2} (-1)^{i+j} 
{[{a_i}_{\lambda_i}a_j]}_{\lambda_i+\lambda_j}
\big\{{a_1}_{\lambda_1}
\stackrel{i}{\check{\cdots}}\stackrel{j}{\check{\cdots}}
{a_{k+1}}_{\lambda_{k+1}} a_{k+2}\big\}_{c} 
}\\
&\displaystyle{
\!\!\!\!\!\!\!\!\!\!\!\!\!\!\!\!\!\!\!\!\!\!\!\!\!\!\!\!\!\!\!\!\!\!\!\!\!\!\!\!\!\!\!\!\!\!\!\!\!\!\!\!\!\!\!\!\!\!\!\!\!\!\!\!\!\!
+ \sum_{\substack{i,j,p,q=1\\i<j,p<q \\ \{i,j\}\cap\{p,q\}=\emptyset}}^{k+2} 
(-1)^{i+j+p+q} \epsilon(p,i)\epsilon(p,j)\epsilon(q,i)\epsilon(q,j)
}\nonumber\\
&\displaystyle{
\,\,\,\,\,\,\,\,\,\,\,\,\,\,\,\,\,\,
\times\,\big\{{a_1}_{\lambda_1}
\stackrel{p}{\check{\cdots}}\stackrel{q}{\check{\cdots}}
\stackrel{i}{\check{\cdots}}\stackrel{j}{\check{\cdots}}
{a_{k+2}}_{\lambda_{k+2}}
[{a_i}_{\lambda_i} a_j]_{\lambda_i+\lambda_j}
[{a_p}_{\lambda_p} a_q]\big\}_{c} 
}\nonumber\\
&\displaystyle{
+ \sum_{\substack{i,j,h=1\\i<j \\ i,j\neq h}}^{k+2} 
(-1)^{k+i+j+h} \epsilon(h,i)\epsilon(h,j)
\big\{{a_1}_{\lambda_1}
\stackrel{h}{\check{\cdots}}\stackrel{i}{\check{\cdots}}\stackrel{j}{\check{\cdots}}
{a_{k+2}}_{\lambda_{k+2}}
[{a_h}_{\lambda_h}[{a_i}_{\lambda_i} a_j]]
\big\}_{c} 
\bigg)\bigg|_{\lambda_{k+2}=\lambda_{k+2}^\dagger}\,.
}\nonumber
\end{eqnarray}
Notice that the first term in \eqref{eq:fe_2} is the negative of the second term in \eqref{eq:fe_3},
and the second term in \eqref{eq:fe_2} is the negative of the first term in \eqref{eq:fe_3}.
Moreover, it is not hard to check,
using the Jacobi identity for the $\lambda$-bracket on $A$,
that the last tern in \eqref{eq:fe_3} is identically zero,
and, using the skew-symmetry condition B3. on $c$,
that also the third term in \eqref{eq:fe_3} is zero.
In conclusion, $d^2c=0$, as we wanted.
\end{proof}
In the next section we shall embed the cohomology complex
$\Gamma^\bullet$, introduced in Section~1.1, in the cohomology
complex $\bar{C}^\bullet$, and we shall prove that,
if the $\mb F[\partial]$-module $A$ decomposes as a direct sum of the torsion
and a free submodule,
then this embedding is an isomorphism.
We believe that the (slightly) bigger cohomology complex $C^\bullet$ 
is a more natural and a more correct definition for the Lie conformal algebra 
cohomolgy complex.
This will be clear when interpreting in Section \ref{sec:2} the cohomology $H(C^\bullet,d)$
in terms of abelian Lie conformal algebra extensions of $A$ by
the module $M$.

\vspace{3ex}
\subsection{Isomorphism of the cohomology complexes $\Gamma^\bullet$ and $\bar C^\bullet$.}~~
\label{sec:1.4}
We define, for $k\geq1$, 
an $\mb F$-linear map $\psi^k:\,{\tilde\Gamma}^k\to C^k$, as follows.
Given $\tilde\gamma\in\tilde\Gamma^k$, we define 
$\psi^k(\tilde\gamma):\,A^{\otimes k}\to\mb F[\lambda_1,\dots,\lambda_{k-1}]\otimes M$, by:
\begin{equation}\label{eq:5}
\{{a_1}_{\lambda_1}\cdots{a_{k-1}}_{\lambda_{k-1}}a_k\}_{\psi^k(\tilde\gamma)}
=
\tilde\gamma_{\lambda_1,\cdots,\lambda_{k-1},\lambda_k^\dagger}(a_1,\cdots,a_{k})\,,
\end{equation}
where, as before, 
$\lambda_k^\dagger=-\sum_{j=1}^{k-1}\lambda_j-\partial^M$,
and $\partial^M$ is moved to the left.
\begin{lemma}\label{lem:laif}
\begin{enumerate}
\alphaparenlist
\item 
For $\tilde\gamma\in\tilde\Gamma^k$, we have $\psi^k(\tilde\gamma)\in\bar C^k$.
\item 
We have $\Ker(\psi^k)=\partial\tilde\Gamma^k$.
Hence $\psi^k$ induces an injective $\mb F$-linear map 
$\psi^k:\,\Gamma^k=\tilde\Gamma^k/\partial\tilde\Gamma^k\hookrightarrow\bar C^k\subset C^k$.
\item 
Suppose that the Lie conformal algebra $A$ decomposes, 
as $\mb F[\partial]$-module, as
\begin{equation}\label{eq:anniv_1}
A=T\oplus\big(\mb F[\partial]\otimes U\big)\,,
\end{equation}
where $T$ is the torsion of $A$ and $\bar A=\mb F[\partial]\otimes U$ is a complementary
free submodule.
Then $\psi^k(\tilde\Gamma^k)=\bar C^k$,
hence $\psi^k$ induces a bijective $\mb F$-linear map 
$\psi^k:\,\Gamma^k\stackrel{\sim}{\rightarrow}\bar C^k$.
\end{enumerate}
\end{lemma}
\begin{proof}
Let $\tilde\gamma\in\tilde\Gamma^k$, and consider $c=\psi^k(\tilde\gamma)$.
We want to prove that $c\in \bar C^k$.
It is clear that $c$ satisfies conditions B1. and B2. Let us check that it also satisfies condition B3.
Let $\sigma$ be a permutation of the set $\{1,\dots,k\}$, and let $i=\sigma(k)$.
Since $\tilde\gamma$ satisfies the skew-symmetry condition A2., 
we have
\begin{eqnarray}\label{eq:anniv_4}
& 
\{{a_{\sigma(1)}}_{\lambda_{\sigma(1)}}\cdots{a_{\sigma(k-1)}}_{\lambda_{\sigma(k-1)}}a_{\sigma(k)}\}_c
=\tilde\gamma_{\lambda_{\sigma(1)},\cdots,\lambda_{\sigma(k-1)},\lambda_{\sigma(k)}^\dagger}
(a_{\sigma(1)},\cdots,a_{\sigma(k)}) \nonumber\\
& =\text{sign}(\sigma) \tilde\gamma_{\lambda_{1},\cdots,\lambda_i^\dagger,\cdots,\lambda_{k}}
(a_1,\cdots,a_{k})\,.
\end{eqnarray}
If we then replace $\lambda_k$ by $\lambda_k^\dagger$, as prescribed by condition B3.,
we get
\begin{equation}\label{eq:anniv_3}
\lambda_i^\dagger\,\mapsto\,
-\lambda_1-\stackrel{i}{\check{\cdots}}-\lambda_{k-1}-\lambda_k^\dagger-\partial^M=\lambda_i\,.
\end{equation}
Therefore the RHS of \eqref{eq:anniv_4} becomes 
sign$(\sigma)\{{a_1}_{\lambda_1}\cdots{a_{k-1}}_{\lambda_{k-1}}a_k\}_c$,
as required.
It is also clear that $c$ vanishes on the torsion of $A$, thanks to Remark \ref{rem:anniv},
so that $c\in\bar C^k$.
This proves part (a).

By the definition \eqref{eq:july24_8} of the action of $\partial$ on $\tilde\Gamma^k$,
and the definition \eqref{eq:5} of $\psi^k$, we have
$$
\{{a_1}_{\lambda_1}\cdots{a_{k-1}}_{\lambda_{k-1}}a_k\}_{\psi^k\partial\tilde\gamma}
\,=\,
(\partial\tilde\gamma)_{\lambda_1,\cdots,\lambda_{k-1},\lambda_k^\dagger}(a_1,\cdots,a_{k})
\,=\,0\,,
$$
since $-\lambda_1-\cdots-\lambda_{k-1}-\lambda_k^\dagger-\partial^M=0$.
Hence $\partial\tilde\Gamma^k\subset\Ker(\psi^k)$.
For the opposite inclusion,
let $\tilde\gamma\in\Ker(\psi^k)$.
Namely,
$$
\tilde\gamma_{\lambda_1,\cdots,\lambda_{k-1},\lambda_k^\dagger}(a_1,\cdots,a_k)\,=\,0\,.
$$
By Taylor expanding in $\lambda_k^\dagger-\lambda_k$, we have
\begin{equation}\label{eq:laif_1}
\sum_{n=0}^\infty
\frac1{n!}(-\Lambda-\partial^M)^n\frac{d^n}{d\lambda_k^n}
\tilde\gamma_{\lambda_1,\cdots,\lambda_k}(a_1,\cdots,a_k)\,=\,0\,,
\end{equation}
where $\Lambda=\sum_{j=1}^k\lambda_j$.
We denote by $\tilde\vartheta:\,A^{\otimes k}\to\mb F[\lambda_1,\dots,\lambda_k]\otimes M$
the following $\mb F$-linear map
$$
\tilde\vartheta_{\lambda_1,\cdots,\lambda_k}(a_1,\cdots,a_k)
\,=\,
\sum_{n=1}^\infty
\frac1{n!}(-\Lambda-\partial^M)^{n-1}\frac{d^n}{d\lambda_k^n}
\tilde\gamma_{\lambda_1,\cdots,\lambda_k}(a_1,\cdots,a_k)\,.
$$
Equation \eqref{eq:laif_1} can then be rewritten as
\begin{equation}\label{eq:laif_2}
\tilde\gamma_{\lambda_1,\cdots,\lambda_k}(a_1,\cdots,a_k)
\,=\,
(\partial^M+\lambda_1+\cdots+\lambda_k)
\tilde\vartheta_{\lambda_1,\cdots,\lambda_k}(a_1,\cdots,a_k)\,.
\end{equation}
It follows from equation \eqref{eq:laif_2} that
$\tilde\vartheta$ satisfies conditions A1. and A2., 
since $\tilde\gamma$ does.
Hence $\tilde\vartheta\in\tilde\Gamma^k$.
Equation \eqref{eq:laif_2} then implies that 
$\tilde\gamma=\partial\tilde\vartheta\in\partial\tilde\Gamma^k$,
thus proving (b).

Assume next that $A$ decomposes as in \eqref{eq:anniv_1}.
We need to prove that, for $c\in\bar C^k$,
we can find $\tilde\gamma\in\tilde\Gamma^k$ such that 
\begin{equation}\label{eq:anniv_6}
\psi^k(\tilde\gamma)\,=\,c\,.
\end{equation}
Such a $k$-cochain can be constructed as follows.
For $u_1,\dots,u_{k}\in U$, we let
\begin{equation}\label{eq:anniv_5}
\tilde\gamma_{\lambda_1,\cdots,\lambda_{k}}(u_1,\cdots,u_{k})
\,=\, 
\{{u_1}_{\lambda_1-\frac{\Lambda+\partial^M}{k}}
\cdots
{u_{k-1}}_{\lambda_{k-1}-\frac{\Lambda+\partial^M}{k}}u_k\}_c\,,
\end{equation}
where $\Lambda=\sum_{i=0}^{k-1}\lambda_i$,
and we extend it to $\big(\mb F[\partial]\otimes U\big)^{\otimes k}$ 
by the sesquilinearity condition A1., and to $A^{\otimes k}$ letting it zero 
if one of the arguments is in the torsion $T$.
We need to check that $\tilde\gamma$ satisfies conditions A1., A2. and \eqref{eq:anniv_6}.
Condition A1. is obvious.
It suffices to check condition A2. for elements $a_i=u_i\in U,\,i=1,\dots,k$.
Let $\sigma$ be a permutation of the indices $\{1,\dots,k\}$.
We have,
\begin{equation}\label{eq:anniv_7}
\tilde\gamma_{\lambda_{\sigma(1)},\cdots,\lambda_{\sigma(k)}}(u_{\sigma(1)},\cdots,u_{\sigma(k)})
=
\{{u_{\sigma(1)}}_{\lambda_{\sigma(1)}-\frac{\Lambda+\partial^M}{k}}
\cdots
{u_{\sigma(k-1)}}_{\lambda_{\sigma(k-1)}-\frac{\Lambda+\partial^M}{k}}u_{\sigma(k)}\}_c\,.
\end{equation}
We then observe that
$$
\Big(\lambda_k-\frac{\Lambda+\partial^M}{k}\Big)^\dagger
=-\sum_{i=1}^{k-1}\Big(\lambda_i-\frac{\Lambda+\partial^M}{k}\Big)-\partial^M
=\lambda_k-\frac{\Lambda+\partial^M}{k}\,.
$$
Hence, since $c$ satisfies the skew-symmetry condition B3., the RHS of \eqref{eq:anniv_7}
is equal to
$$
\text{sign}(\sigma) 
\{{u_1}_{\lambda_1-\frac{\Lambda+\partial^M}{k}}
\cdots
{u_{k-1}}_{\lambda_{k-1}-\frac{\Lambda+\partial^M}{k}}u_k\}_c
=
\text{sign}(\sigma) 
\tilde\gamma_{\lambda_1,\cdots,\lambda_{k}}(u_1,\cdots,u_{k})\,.
$$
Finally, we prove that \eqref{eq:anniv_6} holds. We have, for $u_1,\dots,u_{k}\in U$,
\begin{equation}\label{eq:anniv_8}
\{{u_1}_{\lambda_1}\cdots{u_{k-1}}_{\lambda_{k-1}}u_k\}_{\psi^k(\tilde\gamma)}
=
\tilde\gamma_{\lambda_1,\cdots,\lambda_{k-1},\lambda^\dagger_k}(u_1,\cdots,u_{k})\,.
\end{equation}
Note that, if we replace $\lambda_k$ by $\lambda_k^\dagger$, $\Lambda+\partial^M$ becomes 0.
Hence, by the definition \eqref{eq:anniv_5} of $\tilde\gamma$, the RHS 
of \eqref{eq:anniv_8} is equal to 
$\{{u_1}_{\lambda_1}\cdots{u_{k-1}}_{\lambda_{k-1}}u_k\}_c$.
This proves that \eqref{eq:anniv_6} holds for elements of $U$.
Clearly both sides of \eqref{eq:anniv_6} are zero if one of the elements $a_i$ is in $T$.
Since both $\psi^k(\tilde\gamma)$ and $c$ satisfy the sesquilinearity conditions B1. and B2.,
we conclude that \eqref{eq:anniv_6} holds for every $a_i\in A$.
\end{proof}
\begin{theorem}\label{th:red}
The identity map on $M/\partial M$ and the maps $\psi^k,\,k\geq1$, 
induce an embedding of cohomology complexes $\Gamma^\bullet\hookrightarrow\bar C^\bullet$.
If, moreover, the Lie conformal algebra $A$ decomposes, as $\mb F[\partial]$-module, in
a direct sum of a free module and the torsion,
this map is an isomorphism of complexes: $\Gamma^\bullet\simeq\bar C^\bullet$.
\end{theorem}
\begin{proof}
By Lemma \ref{lem:laif} we already know that $\psi^k$ factors through an injective
$\mb F$-linear map $\psi^k:\,\Gamma^k\hookrightarrow \bar C^k$,
and that, if $A$ decomposes as in \eqref{eq:anniv_1},
this map is bijective.
Hence, in order to prove the theorem, we only have to prove that the following diagrams
are commutative:
\begin{equation}\label{eq:diagr}
\UseTips
\xymatrix{
& \bar C^{1} & \\
M/\partial^MM \ar[ur]^{d}  \ar[r]_{\delta}   & \Gamma^1 \ar[u]_{\psi^1} & ,
}\qquad
\xymatrix{
\bar C^k \ar[r]^{d}   
& \bar C^{k+1} & \\
\ar@{>}[u]^{\psi^k} \Gamma^k \ar[r]_{\delta}   
& \Gamma^{k+1} \ar@{>}[u]_{\psi^{k+1}} & ,\,\forall k\geq1\,.
}
\end{equation}
First, given $\tint m\in M/\partial^MM$, we have
$\big(\delta m\big)_\lambda(a)=a_\lambda m$,
so that
$\big(\psi^1\delta m\big)(a)=a_{-\partial^M} m=\big(d\tint m\big)(a)$,
namely the first diagram in \eqref{eq:diagr} is indeed commutative.
Next, given $k\geq1$, 
let $\tilde\gamma\in\tilde\Gamma^k$
be a representative of $\gamma\in\Gamma^k$.
We need to prove that
\begin{equation}\label{eq:ago20_8}
d\psi^k(\tilde\gamma)\,=\,\psi^{k+1}(\delta\tilde\gamma)\,.
\end{equation}
From \eqref{eq:d>} and \eqref{eq:5}, we have
\begin{eqnarray}\label{eq:ago21_1}
&&\displaystyle{
\{{a_1}_{\lambda_1}\cdots{a_{k}}_{\lambda_{k}}a_{k+1}\}_{d\psi^k(\tilde\gamma)}
}\nonumber\\
&&\displaystyle{
= \sum_{i=1}^k (-1)^{i+1} {a_i}_{\lambda_i}
\big\{{a_1}_{\lambda_1}\stackrel{i}{\check{\cdots}}
{a_{k}}_{\lambda_{k}} a_{k+1}\big\}_{\psi^k(\tilde\gamma)} 
+(-1)^k {a_{k+1}}_{\lambda_{k+1}^\dagger}
\big\{{a_1}_{\lambda_1}\cdots{a_{k-1}}_{\lambda_{k-1}} a_{k}\big\}_{\psi^k(\tilde\gamma)} 
}\nonumber\\
&&\displaystyle{
\,\,\,\,\,\,\,
+ \sum_{\substack{i,j=1\\i<j}}^k (-1)^{k+i+j+1} 
\big\{{a_1}_{\lambda_1}\stackrel{i}{\check{\cdots}}\stackrel{j}{\check{\cdots}}
{a_{k}}_{\lambda_{k}} {a_{k+1}}_{\lambda_{k+1}^\dagger}
[{a_i}_{\lambda_i} a_j]\big\}_{\psi^k(\tilde\gamma)} 
}\nonumber\\
&&\displaystyle{
\,\,\,\,\,\,\,
+ \sum_{i=1}^k (-1)^{i}
\big\{{a_1}_{\lambda_1}\stackrel{i}{\check{\cdots}}{a_{k}}_{\lambda_{k}} 
[{a_i}_{\lambda_i} a_{k+1}]\big\}_{\psi^k(\tilde\gamma)}
}\\
&&\displaystyle{
= \sum_{i=1}^k (-1)^{i+1} {a_i}_{\lambda_i}
\tilde\gamma_{\lambda_1,\stackrel{i}{\check{\cdots}},\lambda_{k},\lambda_{k+1}^\dagger}
(a_1,\stackrel{i}{\check{\cdots}},a_{k+1})
+(-1)^k {a_{k+1}}_{\lambda_{k+1}^\dagger}
\tilde\gamma_{\lambda_1,\cdots,\lambda_{k}}(a_1,\cdots,a_{k})
}\nonumber\\
&&\displaystyle{
\,\,\,\,\,\,\,
+ \sum_{\substack{i,j=1\\i<j}}^k (-1)^{k+i+j+1} 
\tilde\gamma_{\lambda_1,
\stackrel{i}{\check{\cdots}}\stackrel{j}{\check{\cdots}},
\lambda_{k},\lambda_{k+1}^\dagger,\lambda_i+\lambda_j}
(a_1,\stackrel{i}{\check{\cdots}}\stackrel{j}{\check{\cdots}},a_{k+1},[{a_i}_{\lambda_i} a_j])
}\nonumber\\
&&\displaystyle{
\,\,\,\,\,\,\,
+ \sum_{i=1}^k (-1)^{i}
\tilde\gamma_{\lambda_1,
\stackrel{i}{\check{\cdots}},
\lambda_{k},\lambda_i+\lambda_{k+1}^\dagger}
(a_1,\stackrel{i}{\check{\cdots}},a_{k},[{a_i}_{\lambda_i} a_{k+1}])\,.
}\nonumber
\end{eqnarray}
In the last equality we used the sesquilinearity of the $\lambda$-action of $A$ on $M$.
Clearly, the RHS of \eqref{eq:ago21_1} is the same as 
$\{{a_1}_{\lambda_1}\cdots{a_{k}}_{\lambda_{k}}a_{k+1}\}_{\psi^{k+1}(\delta\tilde\gamma})$.
This proves equation \eqref{eq:ago20_8} and the theorem.
\end{proof}

\vspace{3ex}
\subsection{Exterior multiplication on $\tilde\Gamma^\bullet$.}~~
\label{sec:3.1}
To complete the section, we review the definition of the wedge product 
on the basic Lie conformal algebra cohomology complex  $\tilde\Gamma^\bullet(A,M)$
(cf. \cite{BKV}).
We assume that 
$A$ is a Lie conformal algebra
and $M$ is an $A$-module endowed with a commutative, associative product
such that $\partial^M:\,M\to M$, and 
$a^M_\lambda:\,M\to\mb F[\lambda]\otimes M$, 
are (ordinary) derivations of this product.

Consider the basic Lie conformal algebra 
cohomology complex $\tilde\Gamma^\bullet(A,M)$
introduced in Section \ref{sec:1.1}.
Given two cochains $\tilde\alpha\in\tilde\Gamma^h$ and $\tilde\beta\in\tilde\Gamma^k$,
we define their \emph{exterior multiplication} $\tilde\alpha\wedge\tilde\beta\in\tilde\Gamma^{k+h}$ 
by the following formula:
\begin{eqnarray}\label{eq:sfor_2}
&& (\tilde\alpha\wedge\tilde\beta)_{\lambda_1,\cdots,\lambda_{h+k}}(a_1,\cdots,a_{h+k}) \\
&& =
\sum_{\sigma\in S_{h+k}} \frac{\text{sign}(\sigma)}{h!k!}
\tilde\alpha_{\lambda_{\sigma(1)},\cdots,\lambda_{\sigma(h)}}
(a_{\sigma(1)},\cdots,a_{\sigma(h)})
\tilde\beta_{\lambda_{\sigma(h+1)},\cdots,\lambda_{\sigma(h+k)}}
(a_{\sigma(h+1)},\cdots,a_{\sigma(h+k)})\,,\nonumber
\end{eqnarray}
where the sum is over the set $S_{h+k}$ of all permutations of $\{1,\dots,h+k\}$.
\begin{proposition}\label{prop:sfor}
\begin{enumerate}
\alphaparenlist
\item The exterior multiplication \eqref{eq:sfor_2} 
makes $\tilde\Gamma^\bullet$ into a $\mb Z$-graded commutative associative superalgebra,
generated by $\tilde\Gamma^0\oplus\tilde\Gamma^1$,
$M=\tilde\Gamma^0$ being an even subalgebra.
\item The operator $\partial$, acting on $\tilde\Gamma^\bullet$ by \eqref{eq:july24_8},
is an even derivation of the superalgebra $\tilde\Gamma^\bullet$.
\item The differential $\delta$, defined by \eqref{eq:july24_7},
is an odd derivation of the superalgebra $\tilde\Gamma^\bullet$.
\end{enumerate}
\end{proposition}
\begin{proof}
Parts (a) and (b) are straightforward.
For part (c), we will need the following simple combinatorial lemma.
\begin{lemma}\label{lem:ago15}
Let $S=\{x_1,\cdots,x_{N+1}\}$ be an ordered set of $N+1$ elements.
Let $\varepsilon$ be the order preserving bijective map 
$\{1,\cdots,{N+1}\}\stackrel{\sim}{\to}S$,
given by $\varepsilon(i)=x_i$.
It induces 
a bijective map between the set of all permutations $\sigma$ of $S$, 
and the set of all bijective maps $\tau:\,\{1,\cdots,N+1\}\stackrel{\sim}{\to}S$, 
given by $\sigma\mapsto\tau=\sigma\circ\varepsilon$.
We then define the sign of $\tau=\sigma\circ\varepsilon$ as
${\rm sign}(\tau):={\rm sign}(\sigma)$.

Let $s,t\in\{1,\cdots,N+1\}$, 
and let $\varepsilon_s:\,\{1,\cdots,N\}\stackrel{\sim}{\to}\{x_1,\stackrel{s}{\check{\cdots}},x_{N+1}\}$,
be the order preserving bijective map.
There is a bijective map
between the set of all permutations $\sigma$ of $S$ such that $\sigma(x_s)=x_t$,
and the set of all bijective maps 
$\tau:\,\{1,\cdots,N\}\stackrel{\sim}{\to}\{x_1,\stackrel{t}{\check{\cdots}},x_{N+1}\}$,
given by $\sigma\mapsto\tau=\sigma\circ\varepsilon_s$.
Moreover, if $\tau=\sigma\circ\varepsilon_s$, we have
${\rm sign}(\tau)=(-1)^{s+t}{\rm sign}(\sigma)$.
\end{lemma}
\begin{proof}
The statement of the lemma is obvious if $s=t=1$.
In the general case, we just notice that  there is a natural bijection
between the set of all permutations $\sigma$ of $S$ such that $\sigma(x_s)=x_t$,
and the set of all permutations $\sigma'$ of $S$ such that $\sigma(x_1)=x_1$,
given by $\sigma\mapsto\sigma'=(x_1,x_2,\cdots,x_t)\circ\sigma\circ(x_s,\cdots,x_2,x_1)$,
so that ${\rm sign}(\sigma')=(-1)^{s+t}{\rm sign}(\sigma)$.
\end{proof}
Going back to the proof of Proposition \ref{prop:sfor},
we have, from \eqref{eq:july24_7},
\begin{eqnarray}\label{eq:ago14_1}
&& (\delta(\tilde\alpha\wedge\tilde\beta))_{\lambda_1,\cdots,\lambda_{h+k+1}}(a_1,\cdots,a_{h+k+1}) 
\nonumber\\
&& =\,\sum_{i=1}^{h+k+1} (-1)^{i+1} {a_i}_{\lambda_i}
\Big((\tilde\alpha\wedge\tilde\beta)_{\lambda_1,\stackrel{i}{\check{\cdots}},\lambda_{h+k+1}}
(a_1,\stackrel{i}{\check{\cdots}},a_{h+k+1}) \Big) \\
&& + \sum_{\substack{i,j=1\\i<j}}^{h+k+1} (-1)^{h+k+i+j+1} 
(\tilde\alpha\wedge\tilde\beta)_{
\lambda_1,\stackrel{i}{\check{\cdots}}\stackrel{j}{\check{\cdots}},\lambda_{h+k+1},\lambda_i+\lambda_j}
(a_1,\stackrel{i}{\check{\cdots}}\stackrel{j}{\check{\cdots}},a_{h+k+1},[{a_i}_{\lambda_i} a_j])\,. 
\nonumber
\end{eqnarray}
Consider the first term in the RHS of \eqref{eq:ago14_1}.
As in Lemma \ref{lem:ago15}, we can identify the set of all permutations 
of $\{1,\stackrel{i}{\check{\cdots}},h+k+1\}$,
with the set ${\mc T}_i$ of bijective maps 
$\tau:\,\{1,\cdots,h+k\}\stackrel{\sim}{\to}\{1,\stackrel{i}{\check{\cdots}},h+k+1\}$.
Hence, by \eqref{eq:sfor_2}, we have
\begin{eqnarray}\label{eq:ago14_2}
&& (\tilde\alpha\wedge\tilde\beta)_{\lambda_1,\stackrel{i}{\check{\cdots}},\lambda_{h+k+1}}
 (a_1,\stackrel{i}{\check{\cdots}},a_{h+k+1}) \\
&& =
\sum_{\tau\in{\mc T}_i} \frac{\text{sign}(\tau)}{h!k!}
\tilde\alpha_{\lambda_{\tau(1)},\cdots,\lambda_{\tau(h)}}
(a_{\tau(1)},\cdots,a_{\tau(h)})
\tilde\beta_{\lambda_{\tau(h+1)},\cdots,\lambda_{\tau(h+k)}}
(a_{\tau(h+1)},\cdots,a_{\tau(h+k)})\,.\nonumber
\end{eqnarray}
Similarly, 
we can identify the set of all permutations 
of $\{1,\stackrel{i}{\check{\cdots}}\stackrel{j}{\check{\cdots}},h+k+1,h+k+2\}$,
and the set ${\mc T}^{h+k+2}_{i,j}$ of all bijective maps
$\tilde\tau:\,\{1,\cdots,h+k\}
\stackrel{\sim}{\to}\{1,\stackrel{i}{\check{\cdots}}\stackrel{j}{\check{\cdots}},h+k+1,h+k+2\}$.
Hence, by \eqref{eq:sfor_2}, we have
\begin{eqnarray}\label{eq:ago15_1}
&&
(\tilde\alpha\wedge\tilde\beta)_{
\lambda_1,\stackrel{i}{\check{\cdots}}\stackrel{j}{\check{\cdots}},\lambda_{h+k+1},\lambda_i+\lambda_j}
(a_1,\stackrel{i}{\check{\cdots}}\stackrel{j}{\check{\cdots}},a_{h+k+1},[{a_i}_{\lambda_i} a_j]) \\
&& 
=\sum_{\tilde\tau\in{\mc T}^{h+k+2}_{i,j}} \frac{\text{sign}(\tilde\tau)}{h!k!}
\tilde\alpha_{\lambda_{\tilde\tau(1)},\cdots,\lambda_{\tilde\tau(h)}}
(a_{\tilde\tau(1)},\cdots,a_{\tilde\tau(h)})
\tilde\beta_{\lambda_{\tilde\tau(h+1)},\cdots,\lambda_{\tilde\tau(h+k)}}
(a_{\tilde\tau(h+1)},\cdots,a_{\tilde\tau(h+k)})\,,\nonumber
\end{eqnarray}
where, in the RHS, we have to replace $\lambda_{h+k+2}=\lambda_i+\lambda_j$
and $a_{h+k+2}=[{a_i}_{\lambda_i}a_j]$.
In particular, if $\tilde\tau\in{\mc T}^{h+k+2}_{i,j}$ is such that $\tilde\tau(s)=h+k+2$,
with $1\leq s\leq h$,
we can use the skew-symmetry condition A2. on $\tilde\alpha$ to get
$$
\tilde\alpha_{\lambda_{\tilde\tau(1)},\cdots,\lambda_{\tilde\tau(h)}}
(a_{\tilde\tau(1)},\cdots,a_{\tilde\tau(h)})
=
(-1)^{h+s}
\tilde\alpha_{
\lambda_{\tilde\tau(1)},\stackrel{s}{\check{\cdots}},\lambda_{\tilde\tau(h)},\lambda_i+\lambda_j}
(a_{\tilde\tau(1)},\stackrel{s}{\check{\cdots}},a_{\tilde\tau(h)},[{a_i}_{\lambda_i}a_j])\,,
$$
and similarly, if $h+1\leq s\leq h+k$, we have
\begin{eqnarray*}
&& \tilde\beta_{\lambda_{\tilde\tau(h+1)},\cdots,\lambda_{\tilde\tau(h+k)}}
(a_{\tilde\tau(h+1)},\cdots,a_{\tilde\tau(h+k)}) \\
&& =
(-1)^{h+k+s}
\tilde\beta_{
\lambda_{\tilde\tau(h+1)},\stackrel{s}{\check{\cdots}},\lambda_{\tilde\tau(h+k)},\lambda_i+\lambda_j}
(a_{\tilde\tau(h+1)},\stackrel{s}{\check{\cdots}},a_{\tilde\tau(h+k)},[{a_i}_{\lambda_i}a_j])\,.
\end{eqnarray*}
Moreover, 
by Lemma \ref{lem:ago15},
we then identify the set of  elements
$\tilde\tau\in{\mc T}^{h+k+2}_{i,j}$ such that $\tilde\tau(s)=h+k+2$,
with the set ${\mc T}_{i,j}$ of all bijective maps
$\tau:\,\{1,\cdots,h+k-1\}\stackrel{\sim}{\to}
\{1,\stackrel{i}{\check{\cdots}}\stackrel{j}{\check{\cdots}},h+k+1\}$.
The corresponding relation among the signs is 
${\rm sign}(\tau)={\rm sign}(\tilde\tau)(-1)^{h+k+s}$.
We can therefore rewrite equation \eqref{eq:ago15_1} as follows
\begin{eqnarray}\label{eq:ago15_2}
&&
\!\!\!\!\!\!\!\!\!\!
(\tilde\alpha\wedge\tilde\beta)_{
\lambda_1,\stackrel{i}{\check{\cdots}}\stackrel{j}{\check{\cdots}},\lambda_{h+k+1},\lambda_i+\lambda_j}
(a_1,\stackrel{i}{\check{\cdots}}\stackrel{j}{\check{\cdots}},a_{h+k+1},[{a_i}_{\lambda_i} a_j]) \\
&& 
= (-1)^k \sum_{s=1}^{h} \sum_{\tau\in{\mc T}_{i,j}} \frac{\text{sign}(\tau)}{h!k!}
\tilde\alpha_{
\lambda_{\tau(1)},\cdots,\lambda_{\tau(h-1)},\lambda_i+\lambda_j}
(a_{\tau(1)},\cdots,a_{\tau(h-1)},[{a_i}_{\lambda_i}a_j]) \nonumber\\
&& \qquad\qquad\qquad\qquad 
\times \tilde\beta_{\lambda_{\tau(h)},\cdots,\lambda_{\tau(h+k-1)}}
(a_{\tau(h)},\cdots,a_{\tau(h+k-1)}) \nonumber\\
&& 
+ \sum_{s=h+1}^{h+k} \sum_{\tau\in{\mc T}_{i,j}} \frac{\text{sign}(\tau)}{h!k!}
\tilde\alpha_{\lambda_{\tau(1)},\cdots,\lambda_{\tau(h)}}
(a_{\tau(1)},\cdots,a_{\tau(h)}) \nonumber\\
&& \qquad\qquad\qquad\qquad \times 
\tilde\beta_{
\lambda_{\tau(h+1)},\cdots,\lambda_{\tau(h+k-1)},\lambda_i+\lambda_j}
(a_{\tau(h+1)},\cdots,a_{\tau(h+k-1)},[{a_i}_{\lambda_i}a_j])\,. \nonumber
\end{eqnarray}
Combining \eqref{eq:ago14_1}, \eqref{eq:ago14_2} and \eqref{eq:ago15_2}
we get
\begin{eqnarray}\label{eq:ago15_3}
&& (\delta(\tilde\alpha\wedge\tilde\beta))_{\lambda_1,\cdots,\lambda_{h+k+1}}(a_1,\cdots,a_{h+k+1}) \\
&&\qquad  =
\sum_{i=1}^{h+k+1} \sum_{\tau\in{\mc T}_i} (-1)^{i+1}  \frac{\text{sign}(\tau)}{h!k!} {a_i}_{\lambda_i}
\Big(
\tilde\alpha_{\lambda_{\tau(1)},\cdots,\lambda_{\tau(h)}}
(a_{\tau(1)},\cdots,a_{\tau(h)}) \nonumber\\
&&\qquad  \qquad\qquad\qquad\qquad 
\times \tilde\beta_{\lambda_{\tau(h+1)},\cdots,\lambda_{\tau(h+k)}}
(a_{\tau(h+1)},\cdots,a_{\tau(h+k)})
\Big) \nonumber\\
&&\qquad  + h \sum_{\substack{i,j=1\\i<j}}^{h+k+1} 
\sum_{\tau\in{\mc T}_{i,j}} (-1)^{h+i+j+1}  \frac{\text{sign}(\tau)}{h!k!}
\tilde\alpha_{
\lambda_{\tau(1)},\cdots,\lambda_{\tau(h-1)},\lambda_i+\lambda_j}
(a_{\tau(1)},\cdots,a_{\tau(h-1)},[{a_i}_{\lambda_i}a_j]) \nonumber\\
&&\qquad  \qquad\qquad\qquad\qquad 
\times \tilde\beta_{\lambda_{\tau(h)},\cdots,\lambda_{\tau(h+k-1)}}
(a_{\tau(h)},\cdots,a_{\tau(h+k-1)}) \nonumber\\
&&\qquad  
+ k \sum_{\substack{i,j=1\\i<j}}^{h+k+1} 
\sum_{\tau\in{\mc T}_{i,j}} (-1)^{k+i+j+1}  \frac{\text{sign}(\tau)}{h!k!}
\tilde\alpha_{\lambda_{\tau(1)},\cdots,\lambda_{\tau(h)}}
(a_{\tau(1)},\cdots,a_{\tau(h)}) \nonumber\\
&&\qquad  \qquad\qquad\qquad\qquad \times 
\tilde\beta_{
\lambda_{\tau(h+1)},\cdots,\lambda_{\tau(h+k-1)},\lambda_i+\lambda_j}
(a_{\tau(h+1)},\cdots,a_{\tau(h+k-1)},[{a_i}_{\lambda_i}a_j])
\,. \nonumber
\end{eqnarray}
On the other hand, by \eqref{eq:july24_7} and \eqref{eq:sfor_2} we have
\begin{eqnarray}\label{eq:ago15_4}
&& ((\delta\tilde\alpha)\wedge\tilde\beta)_{\lambda_1,\cdots,\lambda_{h+k+1}}(a_1,\cdots,a_{h+k+1}) \\
&& =
\sum_{i=1}^{h+1} \sum_{\sigma\in S_{h+k+1}} \frac{\text{sign}(\sigma)}{(h+1)!k!}(-1)^{i+1}
{a_{\sigma(i)}}_{\lambda_{\sigma(i)}}
\Big(\tilde\alpha_{\lambda_{\sigma(1)},\stackrel{i}{\check{\cdots}},\lambda_{\sigma(h+1)}}
(a_{\sigma(1)},\stackrel{i}{\check{\cdots}},a_{\sigma(h+1)}) \Big) \nonumber\\
&&\,\,\,\,\,\,\,\,\,\,\,\,\,\,\,\,\,\,\,\,\,\,\,\,\,\,\,\,\,\,\,\,\,\,\,\,\,\,\,\,\,\,\,\,\,\,\,\,\,\,\,\,\,\,
\times \tilde\beta_{\lambda_{\sigma(h+2)},\cdots,\lambda_{\sigma(h+k+1)}}
(a_{\sigma(h+2)},\cdots,a_{\sigma(h+k+1)}) \nonumber\\
&& +
\sum_{\substack{i,j=1\\i<j}}^{h+1} \sum_{\sigma\in S_{h+k+1}} 
\frac{\text{sign}(\sigma)}{(h+1)!k!} (-1)^{h+i+j+1} \nonumber\\
&&\,\,\,\,\,\,\,\,\,\,\,\,\,\,\,\,\,\,\,\,\,\,\,\,\,\,\,\,\,\,\,\,\,\,\,\,\,
\times\tilde\alpha_{
\lambda_{\sigma(1)},\stackrel{i}{\check{\cdots}}\stackrel{j}{\check{\cdots}},\lambda_{\sigma(h+1)},
\lambda_{\sigma(i)}+\lambda_{\sigma(j)}}
(a_{\sigma(1)},\stackrel{i}{\check{\cdots}}\stackrel{j}{\check{\cdots}},a_{\sigma(h+1)},
[{a_{\sigma(i)}}_{\lambda_{\sigma(i)}} a_{\sigma(j)}]) \nonumber\\
&&\,\,\,\,\,\,\,\,\,\,\,\,\,\,\,\,\,\,\,\,\,\,\,\,\,\,\,\,\,\,\,\,\,\,\,\,\,\,\,\,\,\,\,\,\,\,\,\,\,\,\,\,\,\,
\times \tilde\beta_{\lambda_{\sigma(h+2)},\cdots,\lambda_{\sigma(h+k+1)}}
(a_{\sigma(h+2)},\cdots,a_{\sigma(h+k+1)})\,.\nonumber
\end{eqnarray}
By Lemma \ref{lem:ago15}, we can identify the set of all permutations
$\sigma$ such that $\sigma(i)=s$,
with the set ${\mc T}_s$ of all bijective maps 
$\tau:\,\{1,\cdots,h+k\}\stackrel{\sim}{\to}\{1,\stackrel{s}{\check{\cdots}},h+k+1\}$,
and the relation among the corresponding signs is ${\rm sign}{\tau}=(-1)^{i+s}{\rm sign}(\sigma)$.
Hence, the first term in the RHS of \eqref{eq:ago15_4} can be rewritten as follows
\begin{eqnarray}\label{eq:ago15_5}
&& \sum_{i=1}^{h+1} \sum_{s=1}^{h+k+1}
\sum_{\tau\in{\mc T}_s} 
\frac{\text{sign}(\tau)}{(h+1)!k!}(-1)^{s+1}
{a_s}_{\lambda_s}
\Big(\tilde\alpha_{\lambda_{\tau(1)},\cdots,\lambda_{\tau(h)}}
(a_{\tau(1)},\cdots,a_{\tau(h)}) \Big) \\
&&\,\,\,\,\,\,\,\,\,\,\,\,\,\,\,\,\,\,\,\,\,\,\,\,\,\,\,\,\,\,\,\,\,\,\,\,\,\,\,\,\,\,\,\,\,\,\,\,\,\,\,\,\,\,
\times \tilde\beta_{\lambda_{\tau(h+1)},\cdots,\lambda_{\tau(h+k)}}
(a_{\tau(h+1)},\cdots,a_{\tau(h+k)}) \,.\nonumber
\end{eqnarray}
Similarly, applying twice Lemma \ref{lem:ago15},
we can identify the set of all permutations
$\sigma$ such that $\sigma(i)=s$ and $\sigma(j)=t$,
with the set ${\mc T}_{s,t}$ of all bijective maps 
$\tau:\,\{1,\cdots,h+k-1\}\stackrel{\sim}{\to}\{1,\stackrel{s,t}{\check{\cdots}},h+k+1\}$,
and the relation among the corresponding signs is 
${\rm sign}(\tau)=(-1)^{i+j+s+t}\epsilon(s,t){\rm sign}(\sigma)$,
where, as in the proof of Proposition \ref{prop:anniv},
we let $\epsilon(s,t)$ be $+1$ if $s<t$ and $-1$ if $s>t$.
Hence, the second term in the RHS of \eqref{eq:ago15_4} becomes
\begin{eqnarray*}
&&
\!\!\!\!\!\!\!\!\!\!\!\!
\sum_{\substack{i,j=1\\i<j}}^{h+1} \sum_{\substack{s,t=1\\s\neq t}}^{h+k+1}
\sum_{\tau\in{\mc T}_{s,t}}
\frac{\text{sign}(\tau)}{(h+1)!k!} (-1)^{h+s+t+1}\epsilon(s,t) 
\tilde\alpha_{
\lambda_{\tau(1)},\cdots,\lambda_{\tau(h-1)},\lambda_{s}+\lambda_{t}}
(a_{\tau(1)},\cdots,a_{\tau(h-1)},[{a_s}_{\lambda_s} a_t]) \\
&&\,\,\,\,\,\,\,\,\,\,\,\,\,\,\,\,\,\,\,\,\,\,\,\,\,\,\,\,\,\,\,\,\,\,\,\,\,\,\,\,\,\,\,\,\,\,\,\,\,\,\,\,\,\,
\times \tilde\beta_{\lambda_{\tau(h)},\cdots,\lambda_{\tau(h+k-1)}}
(a_{\tau(h)},\cdots,a_{\tau(h+k-1)})\,.
\end{eqnarray*}
By skew-symmetry of the $\lambda$-bracket 
and the sesquilinearity condition A1. for $\tilde\alpha$,
the above expression can be rewritten as
\begin{eqnarray}\label{eq:ago15_6}
&& 2 \sum_{\substack{i,j=1\\i<j}}^{h+1} \sum_{\substack{s,t=1\\s< t}}^{h+k+1}
\sum_{\tau\in{\mc T}_{s,t}}
\frac{\text{sign}(\tau)}{(h+1)!k!} (-1)^{h+s+t+1}
\tilde\alpha_{
\lambda_{\tau(1)},\cdots,\lambda_{\tau(h-1)},\lambda_{s}+\lambda_{t}}
(a_{\tau(1)},\cdots,a_{\tau(h-1)},[{a_s}_{\lambda_s} a_t]) \nonumber\\
&&\,\,\,\,\,\,\,\,\,\,\,\,\,\,\,\,\,\,\,\,\,\,\,\,\,\,\,\,\,\,\,\,\,\,\,\,\,\,\,\,\,\,\,\,\,\,\,\,\,\,\,\,\,\,
\times \tilde\beta_{\lambda_{\tau(h)},\cdots,\lambda_{\tau(h+k-1)}}
(a_{\tau(h)},\cdots,a_{\tau(h+k-1)})\,.
\end{eqnarray}
Combining equations \eqref{eq:ago15_4}, \eqref{eq:ago15_5} and \eqref{eq:ago15_6},
we then get
\begin{eqnarray}\label{eq:ago15_7}
&& ((\delta\tilde\alpha)\wedge\tilde\beta)_{\lambda_1,\cdots,\lambda_{h+k+1}}(a_1,\cdots,a_{h+k+1}) \\
&& \qquad=
\sum_{s=1}^{h+k+1}
\sum_{\tau\in{\mc T}_s} 
\frac{\text{sign}(\tau)}{h!k!}(-1)^{s+1}
{a_s}_{\lambda_s}
\Big(\tilde\alpha_{\lambda_{\tau(1)},\cdots,\lambda_{\tau(h)}}
(a_{\tau(1)},\cdots,a_{\tau(h)}) \Big) \nonumber\\
&&\qquad \,\,\,\,\,\,\,\,\,\,\,\,\,\,\,\,\,\,\,\,\,\,\,\,\,\,\,\,\,\,\,\,\,\,\,\,\,\,\,\,\,\,\,\,\,\,\,\,\,\,\,\,\,\,
\times \tilde\beta_{\lambda_{\tau(h+1)},\cdots,\lambda_{\tau(h+k)}}
(a_{\tau(h+1)},\cdots,a_{\tau(h+k)}) \nonumber\\
&&\qquad +
h \sum_{\substack{s,t=1\\s< t}}^{h+k+1}
\sum_{\tau\in{\mc T}_{s,t}}
\frac{\text{sign}(\tau)}{h!k!} (-1)^{h+s+t+1}
\tilde\alpha_{\lambda_{\tau(1)},\cdots,\lambda_{\tau(h-1)},\lambda_{s}+\lambda_{t}}
(a_{\tau(1)},\cdots,a_{\tau(h-1)},[{a_s}_{\lambda_s} a_t]) \nonumber\\
&&\qquad \,\,\,\,\,\,\,\,\,\,\,\,\,\,\,\,\,\,\,\,\,\,\,\,\,\,\,\,\,\,\,\,\,\,\,\,\,\,\,\,\,\,\,\,\,\,\,\,\,\,\,\,\,\,
\times \tilde\beta_{\lambda_{\tau(h)},\cdots,\lambda_{\tau(h+k-1)}}
(a_{\tau(h)},\cdots,a_{\tau(h+k-1)})\,.\nonumber
\end{eqnarray}
With a similar computation we also get
\begin{eqnarray}\label{eq:ago15_8}
&& (\tilde\alpha\wedge(\delta\tilde\beta))_{\lambda_1,\cdots,\lambda_{h+k+1}}(a_1,\cdots,a_{h+k+1}) \\
&& \qquad=
(-1)^h\sum_{s=1}^{h+k+1}
\sum_{\tau\in{\mc T}_s} 
\frac{\text{sign}(\tau)}{h!k!}(-1)^{s+1}
\tilde\alpha_{\lambda_{\tau(1)},\cdots,\lambda_{\tau(h)}}
(a_{\tau(1)},\cdots,a_{\tau(h)}) \nonumber\\
&&\qquad \,\,\,\,\,\,\,\,\,\,\,\,\,\,\,\,\,\,\,\,\,\,\,\,\,\,\,\,\,\,\,\,\,\,\,\,\,\,\,\,\,\,\,\,\,\,\,\,\,\,\,\,\,\,
\times {a_s}_{\lambda_s}
\Big(\tilde\beta_{\lambda_{\tau(h+1)},\cdots,\lambda_{\tau(h+k)}}
(a_{\tau(h+1)},\cdots,a_{\tau(h+k)}) \Big) \nonumber\\
&&\qquad +
(-1)^h k \sum_{\substack{s,t=1\\s< t}}^{h+k+1}
\sum_{\tau\in{\mc T}_{s,t}}
\frac{\text{sign}(\tau)}{h!k!} (-1)^{k+s+t+1}
\tilde\alpha_{\lambda_{\tau(1)},\cdots,\lambda_{\tau(h)}}
(a_{\tau(1)},\cdots,a_{\tau(h)}) \nonumber\\
&&\qquad \,\,\,\,\,\,\,\,\,\,\,\,\,\,\,\,\,\,\,\,\,\,\,\,\,\,\,\,\,\,\,\,\,\,\,\,\,\,\,\,\,\,\,\,\,\,\,\,\,\,\,\,\,\,
\times \tilde\beta_{
\lambda_{\tau(h+1)},\cdots,\lambda_{\tau(h+k-1)},\lambda_{s}+\lambda_{t}}
(a_{\tau(h+1)},\cdots,a_{\tau(h+k-1)},[{a_s}_{\lambda_s} a_t])\,.\nonumber
\end{eqnarray}
Part (c) follows from equations \eqref{eq:ago15_3},
\eqref{eq:ago15_7} and \eqref{eq:ago15_8}.
\end{proof}

\section{Cohomology and extenstions}\label{sec:2}

In this section we interpret the chomology of the complex $(C^\bullet,d)$
in terms of extensions of the Lie conformal algebra $A$ and its modules.

We start by reviewing the notions of extensions of a module over a Lie conformal algebra
and of a Lie conformal algebra.
Let $A$ be a Lie conformal algebra and let $M,\,N$ be $A$-modules.
We denote by $\partial^M$ and $\partial^N$ the $\mb F[\partial]$-module structure
on $M$ and $N$ respectively,
and by $a_\lambda^M$ and $a_\lambda^N$
the $\lambda$-action of $a\in A$ on $M$ and $N$ respectively.
An \emph{extension} of $M$ by $N$ is, by definition, an $A$-module $E$
together with a short exact sequence of $A$-modules
$$
0\,\to\,N\,\to\,E\,\to\,M\,\to\,0\,.
$$
We can fix a splitting $E=M\oplus N$ as $\mb F$-vector spaces.
This space is an $A$-module extension of $M$ by $N$
if it is endowed with:
\begin{enumerate}
\item an endomorphism $\partial^E$ of $M\oplus N$, 
such that $\partial^E|_N=\partial^N$
and $\partial^Em-\partial^Mm\in N$ for every $m\in M$,
which makes $M\oplus N$ into an $\mb F[\partial]$-module;
\item a $\lambda$-action of $A$ on $M\oplus N$ such that
$a_\lambda^E|_N=a_\lambda^N$ for every $a\in A$
and $a_\lambda^E m-a_\lambda^M m\in \mb F[\lambda]\otimes N$ for every $a\in A$ and $m\in M$,
which makes $M\oplus N$ into an $A$-module.
\end{enumerate}
In this setting, two structures of $A$-module extensions $E$ and $E'$ on $M\oplus N$
are \emph{isomorphic} if there is an $A$-module isomorphism
$\sigma:\,E\to E'$ such that
$\sigma|_N=\id_N$ 
and $\sigma(m)-m\in N$ for every $m\in M$.
An extension $E$ is \emph{split}
if it is isomorphic to $M\oplus N$ as an $A$-module,
and it is said to be $\mb F[\partial]$-\emph{split}
if it is isomorphic to $M\oplus N$ as an $\mb F[\partial]$-module,
namely if we can chose the $\mb F$-vector space splitting $E=M\oplus N$
such that $\partial^E=\partial^M\oplus\partial^N$.

We can also talk about extensions of a Lie conformal algebra.
Let $A,\, B$ be two Lie conformal algebras, and assume that
the $\mb F[\partial]$-module $B$ is endowed with a structure of an $A$-module.
We denote by $\partial^A$ and $\partial^B$ the $\mb F[\partial]$-module structure
on $A$ and $B$ respectively,
and by $[\cdot\,_\lambda\,\cdot]^A$ and $[\cdot\,_\lambda\,\cdot]^B$
the $\lambda$-brackets on $A$ and $B$ respectively.
An \emph{extension} of $A$ by $B$ is, by definition, 
a Lie conformal algebra $E$
together with a short exact sequence of Lie conformal algebras
$$
0\,\to\,B\,\to\,E\,\to\,A\,\to\,0\,.
$$
In other words, if we fix a splitting $E=A\oplus B$ as $\mb F$-vector spaces,
the structure of a Lie conformal algebra extension on $E$ consists of:
\begin{enumerate}
\item an endomorphism $\partial^E$ of $A\oplus B$, 
such that $\partial^E|_B=\partial^B$
and $\partial^Ea-\partial^Aa\in B$ for every $a\in A$,
which makes $A\oplus B$ into an $\mb F[\partial]$-module,
\item a $\lambda$-bracket on $A\oplus B$ such that
$[\cdot\,_\lambda\,\cdot]^E|_B=[\cdot\,_\lambda\,\cdot]^B$,
$[a_\lambda b]^E=a_\lambda b$ for every $a\in A$ and $b\in B$,
and $[a_\lambda a']^E-[a_\lambda a']^A\in \mb F[\lambda]\otimes B$ for every $a,a'\in A$,
which makes $A\oplus B$ into a Lie conformal algebra.
\end{enumerate}
As before, two structures of Lie conformal algebra extensions $E$ and $E'$ 
on $A\oplus B$ are \emph{isomorphic}
if there is a Lie conformal algebra isomorphism $\sigma:\,E\to E'$
such that
$\sigma|_B=\id_B$ and $\sigma(a)-a\in B$ for every $a\in A$.
$E$ is a \emph{split} extension
if it isomorphic, as a Lie conformal algebra, 
to the semi-direct sum of $A$ and $B$,
and it is said to be $\mb F[\partial]$-\emph{split}
if it is isomorphic to $A\oplus B$ as an $\mb F[\partial]$-module, 
namely if we can chose the $\mb F$-vector space splitting $E=A\oplus B$
such that $\partial^E=\partial^A\oplus\partial^B$.

We next review the construction of the module $\chom(M,N)$ of conformal homomorphisms
\cite{K}.
A \emph{conformal homomorphism} from the $\mb F[\partial]$-module $M$
to the $\mb F[\partial]$-module $N$
is an $\mb F$-linear map $\varphi_\lambda:\,M\to \mb F[\lambda]\otimes N$
such that
$$
\varphi_\lambda(\partial^M m)=(\partial^N +\lambda)\varphi_\lambda(m)\,.
$$
We denote by $\chom(M,N)$ the space of all conformal homomorphisms from $M$ to $N$.
It has the structure of an $\mb F[\partial]$-module given by
$$
(\partial \varphi)_\lambda=-\lambda\varphi_\lambda\,.
$$
If, moreover, $M$ and $N$ are modules over the Lie conformal algebra $A$,
then $\chom(M,N)$ has the structure of an $A$-module, given by
$$
(a_\lambda\varphi)_\mu
\,=\,a^N_\lambda\circ\varphi_{\mu-\lambda}-\varphi_{\mu-\lambda}\circ a^M_\lambda\,.
$$

In the following theorem, we denote by $H^\bullet(A,M)=\bigoplus_{k\in\mb Z_+} H^k(A,M)$ 
the cohomology of the complex $(C^\bullet,d)$ associated to the Lie conformal algebra $A$
and the $A$-module $M$ (see Section \ref{sec:1.3_b}).
\begin{theorem}\label{th:ago8}
\begin{enumerate}
\alphaparenlist
\item $H^0(A,M)$ is naturally identified with the set of isomorphism classes of extensions of $\mb F$,
considered as $A$-module with trivial action of $\partial$ and trivial $\lambda$-action of $A$,
by the $A$-module $M$.
\item $H^1(A,\chom(M,N))$ is naturally identified with the set of 
isomorphism classes of $\mb F[\partial]$-split 
extensions of the $A$-module $M$ by the $A$-module $N$.
\item $H^2(A,M)$ is naturally identified with the set of 
isomorphism classes of $\mb F[\partial]$-split 
extensions of the Lie conformal algebra $A$ by the $A$-module $M$,
viewed as a Lie conformal algebra with the zero $\lambda$-bracket.
\end{enumerate}
\end{theorem}
\begin{proof}
By definition, $H^0(A,M)$ consists of elements $\tint m\in M/\partial^MM$ in the kernel of $d$,
namely such that $a_{-\partial^M}m=0$ for every $a\in A$.
In other words,
\begin{equation}\label{eq:vic_1}
H^0(A,M)=\big\{m\in M\,|\,a_{-\partial^M}m=0\,,\,\,\forall a\in A\big\}/\partial^MM\,.
\end{equation}
On the other hand, as discussed above, a structure of an $A$-module extension $E$
of $\mb F$ by $M$ on the space $M\oplus\mb F$ 
is uniquely defined by an element $\partial^E1=m\in M$
such that $a^M_\lambda m\in (\partial^M+\lambda)\mb F[\lambda]\otimes M$ 
(or, equivalently, such that $a_{-\partial^M}m=0$)
for every $a\in M$.
Indeed, the corresponding $\lambda$-action of $a^E_\lambda 1\in M[\lambda]$,
is then uniquely defined by the equation
$(\partial^M +\lambda) \big(a^E_\lambda 1\big)=a^M_\lambda m$,
imposed by sesquilinearity.
It is immediate to check that this construction makes $E=M\oplus\mb F$
into an $A$-module.
Furthermore,  let $m,\,m'\in M$ be such that $a_{-\partial^M}m=a_{-\partial^M}m'=0$
for every $a\in A$, 
and consider the corresponding structures of $A$-module extensions $E$ and $E'$ 
on $M\oplus\mb F$.
An isomorphism of $A$-module extensions $\sigma:\,E\to E'$
is completely defined by an element $\sigma(1)-1=n\in M$,
such that $\partial^E \sigma(1)=\sigma(\partial^{E'}1)$,
or, equivalently, $m=m'+\partial n$.
Hence, $m$ and $m'$ correspond to isomorphic extensions
if and only if they differ by an element of $\partial M$.
This proves part (a).

By definition, $C^1(A,\chom(M,N))$ is the space of $\mb F[\partial]$-linear maps $c:\,A\to\chom(M,N)$.
It can be identified, letting $c(a)_\lambda(m)=a^c_\lambda m$,
with the space of $\mb F$-linear maps $A\otimes M\to \mb F[\lambda]\otimes N$,
satisfying the following sesquilinearity conditions (for $a\in A,\, m\in M$):
\begin{equation}\label{eq:ago9_1}
(\partial a)^c_\lambda m=-\lambda a^c_\lambda m
\,\,\,\,\,\,\,\,,\,\,\,\,\,\,\,\,\,\,\,\,
a^c_\lambda(\partial m)=(\lambda+\partial^N)(a^c_\lambda m)\,.
\end{equation}
The equation $dc=0$ for $c$ to be closed can then be rewritten,
recalling \eqref{eq:d>} and using the above notation, 
as follows ($a,b\in A,\,m\in M$):
\begin{equation}\label{eq:ago9_2}
a^N_\lambda (b^c_\mu m) + a^c_\lambda (b^M_\mu m)
- b^N_\mu (a^c_\lambda m)- b^c_\mu (a^M_\lambda m)
-[a_\lambda b]^c_{\lambda+\mu}m\,=\,0\,.
\end{equation}
Notice that, if $\varphi_\lambda\in\chom(M,N)$, then $\varphi_0$ is an $\mb F[\partial]$-linear 
map from $M$ to $N$,
and conversely, any $\mb F[\partial]$-linear map $\varphi:\,M\to N$
can be thought of as an element of $\chom(M,N)$ which is independent of $\lambda$.
It follows that any element in $d(C^0(A,\chom(M,N)))$,
when written in the above notation, is of the form
\begin{equation}\label{eq:ago9_3}
a^{(d\varphi)}_\lambda m \,=\, a^N_\lambda \varphi(m)-\varphi(a^M_\lambda m)\,,
\end{equation}
for an $\mb F[\partial]$-linear map $\varphi:\,M\to N$.
In conclusion,
\begin{equation}\label{eq:vic_2}
H^1(A,\chom(M,N))=
\Big\{c:\, A\otimes M\to \mb F[\lambda]\otimes N 
\,\Big|\, \eqref{eq:ago9_1}-\eqref{eq:ago9_2} \text{ hold}\Big\}
\Big/\Big\{c \text{ of the form } \eqref{eq:ago9_3}\Big\}\,.
\end{equation}
On the other hand, as discussed at the beginning of the section,
a structure of $\mb F[\partial]$-split extension $E$ of $M$ by $N$
on the space $M\oplus N$
is uniquely determined by the elements 
$a^E_\lambda m-a^M_\lambda m=:\,a^c_\lambda m\in \mb F[\lambda]\otimes N$,
and the requirement that $E=M\oplus N$ is an $A$-module
exactly says that $a^c_\lambda m$ satisfies conditions \eqref{eq:ago9_1}
and \eqref{eq:ago9_2}.
Furthermore, let $E$ and $E'$ be two such extensions,
associated to the closed elements $c$ and $c'$ respectively.
An isomorphism $\sigma:\,E\to E'$ is uniquely determined by
the elements $\sigma(m)-m=:\varphi(m)\in N$.
The condition that $\sigma$ commutes with the action of $\partial=\partial^M\oplus\partial^N$,
i.e. $\sigma(\partial^M m)=(\partial^M\oplus\partial^N)\sigma(m)$,
is equivalent to $\varphi(\partial^M m)=\partial^N\varphi(m)$,
namely $\varphi:\,M\to N$ is an $\mb F[\partial]$-linear map.
The condition that $\sigma$ commutes with the $\lambda$-action of $A$,
i.e. $\sigma(a^E_\lambda m)=a^{E'}_\lambda \sigma(m)$,
is equivalent to
$$
a^c_\lambda m+\varphi(a^M_\lambda m)
=
a^{c'}_\lambda m+a^N_\lambda\varphi(m)\,,
$$
which means that $c$ and $c'$ differ by an exact element.
This proves part (b).

We are left to prove part (c).
The space $C^2(A,M)$ consists of $\mb F$-linear maps 
$c:\,A^{\otimes 2}\to\mb F[\lambda]\otimes M$,
denoted by
$a\otimes b\mapsto\{a_\lambda b\}_c$,
satisfying the conditions of sesquilinearity
\begin{equation}\label{eq:ago9_4}
\{\partial a_\lambda b\}_c=-\lambda\{a_\lambda b\}_c
\,\,\,\,\,\,\,\,,\,\,\,\,\,\,\,\,\,\,\,\,
\{a_\lambda\partial  b\}_c=(\lambda+\partial^M)\{a_\lambda b\}_c\,,
\end{equation}
and skew-symmetry
\begin{equation}\label{eq:ago9_5}
\{b_\lambda a\}_c = -\{a_{-\lambda-\partial^M}b\}_c\,.
\end{equation}
Recalling the definition \eqref{eq:d>} of $d$ 
and using the skew-symmetry of the $\lambda$-bracket on $A$,
the equation $dc=0$ for $c$ to be closed can be written as follows:
\begin{eqnarray}\label{eq:ago9_6}
&&
a_\lambda\{b_\mu z\}_c - b_\mu\{a_\lambda z\}_c + z_{-\lambda-\mu-\partial^M}\{a_\lambda b\}_c \\
&&
+\{a_\lambda[b_\mu z]\}_c - \{b_\mu[a_\lambda z]\}_c + \{z_{-\lambda-\mu-\partial^M}[a_\lambda b]\}_c
\,=\,0\,,\nonumber
\end{eqnarray}
for every $a,b,z\in A$.
Recall that $C^1(A,M)$ consists of $\mb F[\partial]$-linear maps $\varphi:\,A\to M$.
Hence exact elements $c=d\varphi$ are of the form
\begin{equation}\label{eq:ago9_7}
\{a_\lambda b\}_{d\varphi}
\,=\,
a_\lambda\varphi(b)-b_{-\lambda-\partial^M}\varphi(a)-\varphi([a_\lambda b])\,.
\end{equation}
We thus have
\begin{equation}\label{eq:vic_3}
H^2(A,M)=
\Big\{c:\, A^{\otimes 2}\to \mb F[\lambda]\otimes M 
\,\Big|\, \eqref{eq:ago9_4}-\eqref{eq:ago9_6} \text{ hold}\Big\}
\Big/\Big\{c \text{ of the form } \eqref{eq:ago9_7}\Big\}\,.
\end{equation}
Once we fix an $\mb F[\partial]$-splitting $E=M\oplus A$,
an ``abelian" extension $E$ of the Lie conformal algebra $A$ 
by the $A$-module $M$
is determined by a $\lambda$-bracket 
$[\cdot\,_\lambda\,\cdot]^E:\, (M\oplus A)^{\otimes 2}\to\mb F[\lambda]\otimes(M\oplus A)$,
satisfying the axioms of a Lie conformal algebra,
and such that $[m_\lambda m']^E=0$ for $m,m'\in M$,
$[a_\lambda m]^E=a_\lambda m$ for $a\in A$ and $m\in M$,
and $[a_\lambda b]^E-[a_\lambda b]\in\mb F[\lambda]\otimes M$.
Let
$$
\{a_\lambda b\}_c
\,:=\,
[{a}_\lambda b]^E-[{a}_\lambda b]\,\in\mb F[\lambda]\otimes M\,.
$$
It is not hard to check that
the axioms of sesquilinearity, skew-symmetry and Jacobi identity
for $[\cdot\,_\lambda\, \cdot ]^E$
become equations \eqref{eq:ago9_4}, \eqref{eq:ago9_5} and \eqref{eq:ago9_6} respectively.
Namely $[\cdot\,_\lambda\,\cdot ]^E$ defines a structure of Lie conformal algebra extension on $E$
if and only if $c$ is a closed element of $C^2(A,M)$.
Let $E$ and $E'$ be two such extensions,
associated to the closed elements $c$ and $c'$ respectively.
An isomorphism $\sigma:\,E\to E'$ is uniquely determined by
the elements $\sigma(a)-a=:\varphi(a)\in M$.
It is easy to check that $\sigma$ commutes with the action of $\partial=\partial^M\oplus\partial^A$
if and only if $\varphi(\partial a)=\partial^M\varphi(a)$,
namely $\varphi:\,A\to M$ is an $\mb F[\partial]$-linear map.
Finally, 
$\sigma$ defines a Lie conformal algebra isomorphism,
i.e. $\sigma([a_\lambda b]^E)=[\sigma(a)_\lambda \sigma(b)]^{E'}$,
if and only if
$$
\{a_\lambda b\}_c+\varphi([a_\lambda b])
=
\{a_\lambda b\}_{c'}+a_\lambda\varphi(b)-b_{-\lambda-\partial^M}\varphi(a)\,,
$$
which means that $c$ and $c'$ differ by an exact element.
\end{proof}
\begin{remark}\label{rem:ago8}
Part (a) of Theorem \ref{th:ago8} is the same statement as \cite[Theorem 3.1-2]{BKV}.
Part (b) is equivalent to \cite[Theorem 3.1-3]{BKV}. This is due to Theorem \ref{th:red}
and the fact that $\chom(M,N)$ is free as an $\mb F[\partial]$-module,
hence $C^\bullet(A,\chom(M,N))=\bar C^\bullet(A,\chom(M,N))$.
However, \cite[Theorem 3.1-4]{BKV} is false,
unless $A$ is free as an $\mb F[\partial]$-module.
Part (c) of Theorem \ref{th:ago8} is the corrected version of it.
This lends some support to our opinion that 
the cohomology complex $(C^\bullet,d)$ is a more correct definition
of a Lie conformal algebra cohomology complex.
Moreover, as it appears from the proof of Theorem \ref{th:ago8}, the identification of the cohomology
of the complex $(C^\bullet,d)$ with the extensions of Lie conformal algebras and their representations
is more direct and natural than for the complex $(\Gamma^\bullet,\delta)$.
\end{remark}
\begin{example}\label{ex:ago8}
Consider the centerless Virasoro Lie conformal algebra $\Vir^0=\mb F[\partial]L$,
with $\lambda$-bracket given by $[L_\lambda L]^0=(\partial+2\lambda)L$.
We have $C^1(\Vir^0,\mb F)=\mb Fa$,
where $a:\,\Vir^0\to\mb F[\lambda]$ is determined by $a(L)=1$,
and $C^2(\Vir^0,\mb F)=\mb F\alpha\oplus\mb F\beta$,
where $\alpha,\beta:\,{\Vir^0}^{\otimes 2}\to\mb F[\lambda]$
are determined by $\{L_\lambda L\}_\alpha=\lambda$ 
and $\{L_\lambda L\}_\beta=\lambda^3$.
In particular $da=2\alpha$.
Therefore $H^2(\Vir^0,\mb F)$ is one-dimensional,
meaning that, up to isomorphism, there is a unique 1-dimensional
central extension of $\Vir^0$,
namely $\Vir=\mb F[\partial]L\oplus\mb FF$,
with $C$ central and $[L_\lambda L]=(\partial+2\lambda)L+\frac{\lambda^3}{12}C$.
Note that, since $\Vir^0$ is free as $\mb F[\partial]$-module, this is the same answer that
we get if we consider the cohomology complex $\Gamma^\bullet\simeq\bar C^\bullet$.

On the other hand, we have $C^1(\Vir,\mb F)=\mb Fa\oplus\mb Fb$
where $a,b:\,\Vir\to\mb F$ are determined by $a(L)=1,\,a(C)=0,\,b(L)=0,\,b(C)=1$,
which is strictly bigger than $\bar C^1(\Vir,\mb F)=\mb Fa$.
The 2-cochains are as before: $C^2(\Vir,\mb F)=\mb F\alpha\oplus\mb F\beta$,
with $\alpha$ and $\beta$
determined by $\{L_\lambda L\}_\alpha=\lambda$ and $\{L_\lambda L\}_\beta=\lambda^3$.
In particular $da=2\alpha$ and $db=\frac1{12}\beta$.
Therefore $H^2(\Vir,\mb F)=0$, 
which corresponds to the fact that there are no non-trivial
1-dimensional central extensions of $\Vir$.
On the contrary, the second cohomology of the complex 
$\Gamma^\bullet\simeq\bar C^\bullet$ is one-dimensional.
\end{example}
\begin{remark}\label{rem:vg1}
Since $\bar C^k= C^k$ for $k\neq1$,
the corresponding cohomologies $H^n(\bar C^\bullet)$ and $H^n(C^\bullet)$
are isomorphic unless $n=1$ or $2$.
In particular, it follows from \cite[Theorem 7.1]{BKV},
that for the complex $C^\bullet(\Vir,\mb F)$ we have:
$H^n(\Vir,\mb F)=\mb F$ for $n=0$ or $3$, and $H^n(\Vir,\mb F)=0$ otherwise.
\end{remark}

\section{The space of $k$-chains, contractions and Lie derivatives}\label{sec:3}

\subsection{$\mf g$-complexes.}~~
\label{sec:-1}
Recall that a cohomology complex is a $\mb Z$-graded vector space
$B^\bullet$, endowed with an endomorphism $d$,
such that $d(B^k)\subset B^{k+1}$ and $d^2=0$.
We view $B^\bullet$ as a vector superspace,
where elements of $ B^k$ have the same parity as $k$ in $\mb Z/2\mb Z$,
so that $d$ is an odd operator.

Let $\mf g$ be a Lie algebra,
and let $\hat{\mf g}=\eta\mf g\oplus\mf g\oplus\mb F\partial_\eta$,
where $\eta$ is odd such that $\eta^2=0$,
be the associated $\mb Z$-graded Lie superalgebra.
A $\mf{g}$-\emph{structure} on the complex $B^\bullet$ is 
a $\mb Z$-grading preserving Lie algebra homomorphism
$$
\varphi\,:\,\,\hat{\mf{g}}\,\to\,\End B^\bullet\,,
$$
such that $\varphi(\partial_\eta)=d$.
A complex with a given $\mf g$-structure is called a $\mf g$-\emph{complex}.

Given $X\in\mf g$, the operator $\iota_X=\varphi(\eta X)$ on $B^\bullet$
is called the \emph{contraction},
and the operator $L_X=\varphi(X)$
is called the \emph{Lie derivative} (along $X$).
Note that we have Cartan's formula
\begin{equation}\label{eq:3.1}
L_X\,=\,[d,\iota_X]\,,
\end{equation}
and the commutation relations
\begin{equation}\label{eq:3.2}
[d,L_X]=0\,,\,\,
[\iota_X,\iota_Y]=0\,,\,\,
[L_X,\iota_Y]=[\iota_X,L_Y]=\iota_{[X,Y]}\,,\,\,
[L_X,L_Y]=L_{[X,Y]}\,.
\end{equation}
\begin{remark}\label{rem:vsep}
In order to construct a $\mf g$-structure on a complex $(B^\bullet,d)$,
it suffices to construct commuting odd operators $\iota_X$ on $B^\bullet$,
depending linearly on $X$, such that
$\iota_X(B^k)\subset B^{k-1}$, and
\begin{equation}\label{eq:3.3}
[[d,\iota_X],\iota_Y]=\iota_{[X,Y]}\,,\,\,\forall X,Y\in\mf g\,.
\end{equation}
Indeed, if we define $L_X$ by \eqref{eq:3.1},
all commutation relations \eqref{eq:3.2} hold.
\end{remark}

Let $\partial$ be an endomorphism of the complex $(B^\bullet,d)$,
i.e. such that $\partial(B^k)\subset B^k$ and $[d,\partial]=0$.
Let
$$
\mf g^\partial\,=\,\big\{X\in\mf g\,|\,[\iota_X,\partial]=0\big\}\subset\mf g\,.
$$
Notice that $[L_X,\partial]=0$ for all $X\in\mf g^\partial$.
It follows that $\mf g^\partial$ is a Lie subalgebra of $\mf g$,
and that $(\partial B^\bullet,d)$ is a subcomplex of $(B^\bullet,d)$
with a $\mf g^\partial$-structure.
The corresponding quotient complex
$$
(B^\bullet/\partial B^\bullet,d)\,,
$$
has an induced $\mf g^\partial$-structure,
and it is called the \emph{reduced} $\mf g^\partial$-\emph{complex}.

A \emph{morphism} of a $\mf g$-complex $(B^\bullet, \varphi)$ to an $\mf h$-complex
$(C^\bullet , \psi)$ is a Lie algebra homomorphism $\pi: \mf g \to \mf h$
and a $\mb Z$-grading preserving linear map $\rho : B^\bullet \to C^\bullet$, 
such that
$$
\rho (\varphi(g)b)\,=\,\psi(\pi(g))\rho(b)\,,
$$
for all $b \in B^\bullet$ and $g\in\hat{\mf g}$, 
where $\pi$ is extended to a Lie superalgebra homomorphism
$\hat{\mf g}\to\hat{\mf h}$ by letting $\pi(\eta X)=\eta\pi(X)$ and $\pi(\partial_\eta)=\partial_\eta$.
Such a morphism is an isomorphism if both $\pi$ and $\rho$ are isomorphisms.

\vspace{3ex}
\subsection{The basic and reduced spaces of chains 
$\tilde\Gamma_\bullet$ and $\Gamma_\bullet$.}~~
\label{sec:3.2}
The definitions of the basic and reduced spaces of $k$-chains are obtained,
following \cite{BKV}, by dualizing, respectively, the definitions 
of the spaces $\tilde\Gamma^k$ and $\Gamma^k$
introduced in Section \ref{sec:1.1}.
In particular, the basic space $\tilde\Gamma_k(A,M)$ of $k$-\emph{chains} 
of the Lie conformal algebra $A$
with coefficients in the $A$-module $M$ is, by definition, the quotient of the space
$A^{\otimes k} \otimes \Hom(\mb F[\lambda_1,\dots,\lambda_{k}],M)$,
where $\Hom(\mb F[\lambda_1,\dots,\lambda_{k}],M)$ is the space of $\mb F$-linear
maps from $\mb F[\lambda_1,\dots,\lambda_{k}]$ to $M$,
by the following relations:
\begin{enumerate}
\item[C1.] 
$a_1\otimes\cdots\partial a_i\cdots\otimes a_{k}\otimes\phi
= -a_1\otimes\cdots\otimes a_{k}\otimes(\lambda^*_i\phi)$,
where $\lambda_i^*\phi\in \Hom(\mb F[\lambda_1,\dots,\lambda_{k}],M)$ 
is defined by 
\begin{equation}\label{eq:ago20_2}
(\lambda_i^*\phi)(f(\lambda_1,\cdots,\lambda_{k}))
=\phi(\lambda_if(\lambda_1,\cdots,\lambda_{k}))\,;
\end{equation}
\item[C2.] 
$a_{\sigma(1)}\otimes\cdots\otimes a_{\sigma(k)}\otimes(\sigma^*\phi)
= \text{sign}(\sigma) a_1\otimes\cdots\otimes a_{k}\otimes\phi$,
for every permutation $\sigma\in S_k$,
where $\sigma^*\phi\in \Hom(\mb F[\lambda_1,\dots,\lambda_{k}],M)$ 
is defined by 
\begin{equation}\label{eq:ago20_3}
(\sigma^*\phi)(f(\lambda_1,\cdots,\lambda_{k}))
=\phi(f(\lambda_{\sigma(1)},\cdots,\lambda_{\sigma(k)}))\,.
\end{equation}
\end{enumerate}
We let, for brevity, $\tilde\Gamma_k=\tilde\Gamma_k(A,M)$ and
$\tilde\Gamma_\bullet=\tilde\Gamma_\bullet(A,M)=\bigoplus_{k\in\mb Z_+}\tilde\Gamma_k$.

The following statement is the analogue of Remark \ref{rem:anniv} for the space of $k$-chains.
\begin{lemma}\label{rem:ago19}
If one of the elements $a_i$ is a torsion element of the $\mb F[\partial]$-module $A$,
we have $a_1\otimes\cdots\otimes a_{k}\otimes\phi=0$ in $\tilde\Gamma_k$.
In particular, $\tilde\Gamma_k$ can be identified with the quotient of the space
$\bar A^{\otimes k}\otimes \Hom(\mb F[\lambda_1,\dots,\lambda_{k}],M)$
by the relations C1. and C2. above,
where $\bar A=A/\Tor A$
denotes the quotient of the $\mb F[\partial]$-module $A$ by its torsion.
\end{lemma}
\begin{proof}
If $P(\partial)a_i=0$ for some polynomial $P$,
we have, by the relation C1.,
$$
0=a_1\otimes\cdots(P(\partial)a_i)\cdots\otimes a_{k}\otimes\phi
\equiv a_1\otimes\cdots a_i\cdots\otimes a_{k}\otimes(P(-\lambda_i^*)\phi)\,.
$$
To conclude the lemma we are left to prove that the linear endomorphism
$P(-\lambda_i^*)$ of the space $\Hom(\mb F[\lambda_1,\dots,\lambda_{k}],M)$
is surjective.
For this, consider the subspace
$P(-\lambda_i)\mb F[\lambda_1,\dots,\lambda_{k}]\subset \mb F[\lambda_1,\dots,\lambda_{k}]$,
and fix a complementary subspace $U\subset\mb F[\lambda_1,\dots,\lambda_{k}]$, 
so that 
$\mb F[\lambda_1,\dots,\lambda_{k}]=P(-\lambda_i)\mb F[\lambda_1,\dots,\lambda_{k}]\oplus U$.
Given $\phi\in\Hom(\mb F[\lambda_1,\dots,\lambda_{k}],M)$,
we define the linear map 
$\psi:\,\mb F[\lambda_1,\dots,\lambda_{k}]\to M$ 
by letting $\psi|_U=0$ and 
$\psi(P(-\lambda_i)f(\lambda_1,\cdots,\lambda_{k}))=\phi(f(\lambda_1,\cdots,\lambda_{k}))$
for every $f\in\mb F[\lambda_1,\dots,\lambda_{k}]$.
Clearly, $P(-\lambda_i^*)\psi=\phi$.
\end{proof}

The space $\tilde\Gamma_\bullet$ is endowed with a structure 
of a $\mb Z$-graded $\mb F[\partial]$-module,
with the action of $\partial$ induced by the natural action on
$A^{\otimes k} \otimes \Hom(\mb F[\lambda_1,\dots,\lambda_{k}],M)$:
\begin{eqnarray}\label{eq:sfor_1}
\partial\big(a_1\otimes\cdots\otimes a_{k}\otimes\phi\big)
&=&
\sum_{i=1}^{k}a_1\otimes\cdots(\partial a_i)\cdots\otimes a_{k}\otimes\phi
+ a1\otimes\cdots\otimes a_{k}\otimes(\partial\phi) \nonumber\\
&=&
a_1\otimes\cdots\otimes a_{k}\otimes
\big((-\lambda_1^*-\cdots-\lambda_{k}^*+\partial)\phi\big) \,,
\end{eqnarray}
where $\partial\phi\in\Hom(\mb F[\lambda_1,\dots,\lambda_k],M)$ is defined by
$(\partial\phi)(f)=\partial^M(\phi(f))$.
The \emph{reduced} space of chains $\Gamma_\bullet=\bigoplus_{k\in\mb Z_+}\Gamma_k$
is, by definition, the subspace of $\partial$-invariant chains:
$\Gamma_k=
\{\xi\in\tilde\Gamma_k\,|\,\partial \xi=0\}\subset\tilde\Gamma_k$.

For example, for $k=0$ we have $\tilde\Gamma_0=M$ and $\Gamma_0=\{m\in M\,|\,\partial m=0\}$.
Next, consider the case $k=1$ and
suppose that the $\mb F[\partial]$-module $A$ admits the decomposition \eqref{eq:anniv_1}, 
as a direct sum of $\Tor A$ and a complementary free submodule $\mb F[\partial]\otimes U$.
We already pointed out in Lemma \ref{rem:ago19} that $a\otimes\phi=0$ in $\tilde\Gamma_1$
if $a\in\Tor(A)$.
Moreover, by the sesquilinearity condition C1. we have
$\big(P(\partial)u\big)\otimes\phi= u\otimes\big(P(-\lambda^*)\phi\big)$ in $\tilde\Gamma_1$,
for every $u\in U,\,\phi\in\Hom(\mb F[\lambda],M)$ and every polynomial $P$.
Hence we can identify $\tilde\Gamma_1\simeq U\otimes\Hom(\mb F[\lambda],M)$.
Under this identification, an element $u\otimes\phi\in\tilde\Gamma_1$ is annihilated by $\partial$
if and only if the map $\phi:\,\mb F[\lambda]\to M$, satisfies the equation
$(-\lambda^*+\partial)\phi=0$,
namely if $\phi(\lambda^n)=\partial^n\phi(1)$ for every $n\in\mb Z_+$.
Clearly, there is a bijective correspondence between such maps and the elements of $M$,
given by $\phi\mapsto\phi(1)\in M$.
In conclusion, we have an isomorphism $\Gamma_1\simeq U\otimes M$.

\begin{remark}\label{rem:sfor}
Apparently, there is no natural way to define 
a differential $\delta$ on $k$-chains,
making $\tilde\Gamma_\bullet$ and $\Gamma_\bullet$ homology complexes.
The one given in \cite[Section 4]{BKV} is divergent,
unless any $m\in M$ is annihilated by a power of $\partial^M$.
\end{remark}

\vspace{3ex}
\subsection{Contraction operators acting on $\tilde\Gamma^\bullet$ and $\Gamma^\bullet$.}~~
\label{sec:3.3}
Assume, as in Section \ref{sec:3.1}, that $A$ is a Lie conformal algebra
and $M$ is an $A$-module endowed with a commutative, associative product
$\mu:\,M\otimes M\to M$,
such that $\partial^M:\, M\to M$, 
and $a_\lambda:\, M\to\mb C[\lambda]\otimes M$, satisfy the Leibniz rule.
Given an $h$-chain $\xi\in\tilde\Gamma_h$, we define the \emph{contraction operator}
$\iota_\xi:\,\tilde\Gamma^k\to\tilde\Gamma^{k-h},\,k\geq h$, as follows.
If $a_1\otimes\cdots\otimes a_{h}\otimes\phi
\in A^{\otimes h}\otimes\Hom(\mb F[\lambda_1,\dots,\lambda_{h}],M)$ 
is a representative of $\xi\in\tilde\Gamma_h$,
and $\tilde\gamma\in\tilde\Gamma^k$, we let
\begin{equation}\label{eq:ago15p_1}
(\iota_\xi\tilde\gamma)_{\lambda_{h+1},\cdots,\lambda_{k}}(a_{h+1},\cdots,a_{k})
=
\phi^\mu\big(\tilde\gamma_{\lambda_1,\cdots,\lambda_{k}}(a_1,\cdots,a_{k})\big)\,,
\end{equation}
where, in the RHS, $\phi^\mu$ denotes the composition of maps, 
commuting with $\lambda_{h+1},\dots,\lambda_{k}$,
\begin{equation}\label{eq:ago16_2}
\mb F[\lambda_1,\dots,\lambda_{h}]\otimes M
\stackrel{\phi\otimes\id}{\longrightarrow}
M\otimes M
\stackrel{\mu}{\longrightarrow}
M\,.
\end{equation}
We extend the definition of $\iota_\xi$ to all elements $\xi\in\tilde\Gamma_h$ by linearity on $\xi$,
and we let $\iota_\xi(\tilde\gamma)=0$ if $k<h$.
We also define the Lie derivative $L_\xi$ by Cartan's formula:
$L_\xi=[\delta,\iota_\xi]$.

It is immediate to check, using the sesquilinearity and skew-symmetry 
conditions A1. and A2. for $\tilde\gamma$ (cf. Section \ref{sec:1.1}),
that the RHS in \eqref{eq:ago15p_1} does not depend on the choice of the representative
for $\xi$ in $A^{\otimes h}\otimes\Hom(\mb F[\lambda_1,\dots,\lambda_{h}],M)$.
Moreover, if $\tilde\gamma\in\tilde\Gamma^k$, it follows that $\iota_\xi(\tilde\gamma)$
satisfies both conditions A1. and A2., 
namely $\iota_\xi(\tilde\gamma)\in\tilde\Gamma^{k-h}$.

\begin{proposition}\label{prop:dic17_1}
The contraction operators on the superspace $\tilde\Gamma^\bullet$ commute, i.e.
for $\xi\in\tilde\Gamma_h$ and $\zeta\in\tilde\Gamma_j$ we have
$$
\iota_\xi\iota_\zeta\,=\,(-1)^{hj}\iota_\zeta\iota_\xi\,.
$$
\end{proposition}
\begin{proof}
Let $a_1\otimes \cdots\otimes a_h\otimes\phi
\in A^{\otimes h}\otimes\Hom(\mb F[\lambda_1,\cdots,\lambda_h],M)$ 
be a representative for $\xi\in\tilde\Gamma_h$,
$b_1\otimes \cdots\otimes b_j\otimes\psi
\in A^{\otimes j}\otimes\Hom(\mb F[\mu_1,\cdots,\mu_j],M)$ 
be a representative for $\zeta\in\tilde\Gamma_j$,
and let $\tilde\gamma\in\tilde\Gamma^k$.
By the definition \eqref{eq:ago15p_1} of the contraction operators, we have
\begin{eqnarray*}
&& (\iota_\zeta\iota_\xi\tilde\gamma)_{\nu_1,\cdots,\nu_{k-h-j}}(c_1,\cdots,c_{k-h-j}) \\
&& =
\psi^\mu\big(\phi^\mu\big(
\tilde\gamma_{\lambda_1,\cdots,\lambda_h,\mu_1\cdots,\mu_j,\nu_1,\cdots,\nu_{k-h-j}}
(a_1,\cdots,a_h,b_1,\cdots,b_j,c_1,\cdots,c_{k-h-j})\big)\big)\,.
\end{eqnarray*}
Since obviously $\phi^\mu$ and $\psi^\mu$ commute, 
the proposition follows from condition A2. for $\tilde\gamma$.
\end{proof}

\begin{proposition}\label{prop:ago16_1}
For every basic $h$-chain $\xi\in\tilde\Gamma_h$, we have
\begin{equation}\label{eq:ago16_1}
[\partial,\iota_\xi] = \partial\circ\iota_\xi-\iota_\xi\circ\partial=\iota_{\partial \xi}\,.
\end{equation}
In particular, if $\xi\in\Gamma_h$ is a reduced $h$-chain, then
$\iota_\xi$ commutes with $\partial$,
and it induces a well-defined contraction operator on the reduced cohomology complex:
$\iota_\xi:\,\Gamma^k\to\Gamma^{k-h}$.
\end{proposition}
\begin{proof}
Let $a_1\otimes\cdots\otimes a_{h}\otimes\phi$ be a representative of $\xi\in\tilde\Gamma_h$,
and let $\tilde\gamma\in\tilde\Gamma^k$.
By the definition \eqref{eq:july24_8} of the action of $\partial$ on $\tilde\Gamma^k$, we have
\begin{equation}\label{eq:ago16_3}
\big(\partial\iota_\xi\tilde\gamma\big)_{\lambda_{h+1},\cdots,\lambda_{k}}(a_{h+1},\cdots,a_{k})
=
(\partial^M+\lambda_{h+1}+\cdots+\lambda_{k})
\phi^\mu\big(\tilde\gamma_{\lambda_1,\cdots,\lambda_{k}}(a_1,\cdots,a_{k})\big)\,,
\end{equation}
and, similarly,
\begin{equation}\label{eq:ago16_4}
\big(\iota_\xi\partial\tilde\gamma\big)_{\lambda_{h+1},\cdots,\lambda_{k}}(a_{h+1},\cdots,a_{k})
=
\phi^\mu\big((\partial^M+\lambda_1+\cdots+\lambda_{k})
\tilde\gamma_{\lambda_1,\cdots,\lambda_{k}}(a_1,\cdots,a_{k})\big)\,.
\end{equation}
On the other hand, by the definition \eqref{eq:sfor_1} of the action of $\partial$ on $\tilde\Gamma_h$,
we have
\begin{eqnarray}\label{eq:ago16_5}
& \displaystyle{
\big(\iota_{\partial \xi}\tilde\gamma\big)_{\lambda_{h+1},\cdots,\lambda_{k}}(a_{h+1},\cdots,a_{k})
=
(\partial \phi)^\mu\big(\tilde\gamma_{\lambda_1,\cdots,\lambda_{k}}(a_1,\cdots,a_{k})\big) 
} \\
& \displaystyle{
-\phi^\mu\big((\lambda_1+\cdots+\lambda_{h})
\tilde\gamma_{\lambda_1,\cdots,\lambda_{k}}(a_1,\cdots,a_{k})\big)
\,.}\nonumber
\end{eqnarray}
Equation \eqref{eq:ago16_1} then follows by \eqref{eq:ago16_3}, \eqref{eq:ago16_4}, 
\eqref{eq:ago16_5}, and by the following result.
\begin{lemma}\label{lem:ago16}
For every linear map $\phi:\,\mb F[\lambda_1,\dots,\lambda_{h}]\to M$,
we have
\begin{equation}\label{eq:ago16_6}
[\partial^M,\phi^\mu]=\partial^M\circ\phi^\mu-\phi^\mu\circ(id\otimes\partial^M)
= (\partial\phi)^\mu\,,
\end{equation}
where $\phi^\mu:\,\mb F[\lambda_1,\dots,\lambda_{h}]\otimes M\to M$
is defined in \eqref{eq:ago16_2}.
\end{lemma}
\begin{proof}
Given $f\otimes m\in\mb F[\lambda_1,\dots,\lambda_{h}]\otimes M$ we have
\begin{eqnarray*}
&& \big(\partial^M\circ\phi^\mu\big)(f\otimes m) \,=\, \partial^M\big(\phi(f)\cdot m)\,,\\
&& \big(\phi^\mu\circ(\id\otimes\partial^M)\big)(f\otimes m) \,=\, \phi(f)\cdot(\partial^M m)\,,\\
&& (\partial\phi)^\mu(f\otimes m) \,=\, (\partial^M\phi(f))\cdot m\,.
\end{eqnarray*}
Equation \eqref{eq:ago16_6} follows since, by assumption,
$\partial^M$ is a derivation of $M$.
\end{proof}
\end{proof}

For example, for $h=0$, the contraction by $m\in M=\tilde\Gamma_0$
is given by the commutative associative product in $M$:
$(\iota_m\tilde\gamma)_{\lambda_1,\cdots,\lambda_{k}}(a_1,\cdots,a_{k})
= m\tilde\gamma_{\lambda_1,\cdots,\lambda_{k}}(a_1,\cdots,a_{k})$.
(Which is the same as the exterior multiplication by $m\in \tilde\Gamma^0=M$).
If, moreover, $m\in M$ is such that $\partial m=0$, we have 
$\iota_m\partial\tilde\gamma=\partial\iota_m\tilde\gamma$,
so that $\iota_m$ induces a well-defined map $\Gamma^k\to\Gamma^k$.
Next, consider the case $h=1$.
Recall from the previous section that, if $A$ decomposes as in \eqref{eq:anniv_1},
we have $\tilde\Gamma_1\simeq U\otimes\Hom(\mb F[\lambda],M)$.
The contraction operator associated to $\xi=u\otimes\phi\in\tilde\Gamma_1$
is given by
$(\iota_\xi\tilde\gamma)_{\lambda_2,\cdots,\lambda_{k}}(a_2,\cdots,a_{k})
=\phi^\mu\big(\tilde\gamma_{\lambda,\lambda_2,\cdots,\lambda_{k}}(u,a_2,\cdots,a_{k})\big)$.
Moreover, we have $\Gamma_1\simeq U\otimes M$,
and the contraction operator associated to $\xi=u\otimes m$ is given by
\begin{equation}\label{eq:ago19_1}
(\iota_\xi\tilde\gamma)_{\lambda_2,\cdots,\lambda_{k}}(a_2,\cdots,a_{k})
=\tilde\gamma_{\partial^M,\lambda_2,\cdots,\lambda_{k}}(u,a_2,\cdots,a_{k})_\to m\,,
\end{equation}
where the arrow in the RHS means that $\partial^M$ should be moved to the right.
Clearly, $\iota_\xi\partial\tilde\gamma=\partial\iota_\xi\tilde\gamma$, and
$\iota_\xi$ induces a well-defined map $\Gamma^k\to\Gamma^{k-1}$.

\vspace{3ex}
\subsection{The Lie algebra structure on $\mf g=\Pi\tilde\Gamma_1$
and the $\mf g$-structure on the complex $(\tilde\Gamma^\bullet,\delta)$.}~~
\label{sec:3.2.5}
In this section we want to define a Lie algebra structure on the space of 1-chains $\tilde\Gamma_1$,
thus making $\tilde\Gamma^\bullet$ a $\Pi\tilde\Gamma_1$-complex
(recall the definition in Section \ref{sec:-1}),
where $\Pi$ means
that we take opposite parity, namely we consider $\tilde\Gamma_1$ as an even vector space.
We start by describing the space of 1-chains in a slightly different form.
Recall that $\tilde\Gamma_1$ is the quotient of the space $A\otimes\Hom(\mb F[\lambda],M)$
by the image of the operator $\partial\otimes1+1\otimes\lambda^*$.
We shall identify $\Hom(\mb F[\lambda],M)$ with $M[[x]]$ via the map
$$
\phi\,\mapsto\,\sum_{n\in\mb Z_+}\frac1{n!}\phi(\lambda^n)x^n\,.
$$
It is immediate to check that, under this identification,
the action of $\partial$ on $\Hom(\mb F[\lambda],M)$ corresponds to the natural action 
of $\partial$ on $M[[x]]$,
while the operator $\lambda^*$ acting on $\Hom(\mb F[\lambda],M)$ 
corresponds to the operator $\partial_x=\frac d{dx}$ on $M[[x]]$.
Thus, the space of 1-chains is
$$
\tilde\Gamma_1
\,=\, (A\otimes M[[x]])\big/(\partial\otimes1+1\otimes\partial_x)(A\otimes M[[x]])\,.
$$
Recalling \eqref{eq:sfor_1}, 
the corresponding action of $\partial$ on $\tilde\Gamma_1$ is given by
\begin{equation}\label{eq:nov22_1}
\partial(a\otimes m(x))
\,=\, a\otimes(\partial-\partial_x) m(x)\,,
\end{equation}
and the reduced space of 1-chains is 
$\Gamma_1=\{\xi=a\otimes m(x)\in\tilde\Gamma_1\,|\,\partial\xi=0\}$.
In particular, if $A$ admits a decomposition \eqref{eq:anniv_1}
as a direct sum of $\Tor A$ and a complementary free submodule $\mb F[\partial]\otimes U$,
we have $\tilde\Gamma_1\simeq U\otimes M[[x]]$,
and the reduced subspace $\Gamma_1\subset\tilde\Gamma_1$ consists
of elements of the form
\begin{equation}\label{eq:nov22_4}
\xi = u\otimes (e^{x\partial}m)\,,\,\,u\in U,\,m\in M\,.
\end{equation}

Given $\xi\in\tilde\Gamma_1$,
we can write the action of the contraction operator
$\iota_\xi:\,\tilde\Gamma^k\to\tilde\Gamma^{k-1}$, defined by \eqref{eq:ago15p_1}.
Consider the pairing $M[[x]]\otimes\mb F[\lambda]\to M$ given by
\begin{equation}\label{eq:nov23_6}
\langle x^m,\lambda^n\rangle=n!\, \delta_{m,n}\,,\,\,m,n\in\mb Z_+\,.
\end{equation}
It induces a pairing
$\langle\,,\,\rangle:\,M[[x]]\otimes (\mb F[\lambda]\otimes M)\to M$, given by
\begin{equation}\label{eq:nov22_3}
M[[x]]\otimes \mb F[\lambda]\otimes M
\stackrel{\langle\,,\,\rangle\otimes\id}{\longrightarrow}
M\otimes M
\stackrel{\mu}{\longrightarrow} M\,,
\end{equation}
where $\mu$ in the last step denotes the commutative associative product on $M$.
Then, if $a_1\otimes m(x_1)\in A\otimes M[[x_1]]$ 
is a representative of $\xi\in\tilde\Gamma_1$,
the contraction operator $\iota_\xi:\,\tilde\Gamma^k\to\tilde\Gamma^{k-1}$
acts as follows:
\begin{equation}\label{eq:nov22_2}
(\iota_\xi\tilde\gamma)_{\lambda_{2},\cdots,\lambda_{k}}(a_{2},\cdots,a_{k})
\,=\,
\big\langle  m(x_1),
\tilde\gamma_{\lambda_1,\lambda_2,\cdots,\lambda_{k}}(a_1,a_2,\cdots,a_{k})\big\rangle\,,
\end{equation}
where, in the RHS, $\langle\,,\,\rangle$ denotes the contraction of $x_1$ with $\lambda_1$
defined in \eqref{eq:nov22_3}.
Clearly, if $\xi$ is as in \eqref{eq:nov22_4},
equation \eqref{eq:nov22_2} reduces to \eqref{eq:ago19_1}.

We can also write down the formula for the Lie derivative 
$L_\xi=\delta\circ\iota_\xi+\iota_\xi\circ\delta$.
Let $a_1\otimes m(x_1)\in A\otimes M[[x_1]]$ be a representative of $\xi\in\tilde\Gamma_1$.
Recalling the expression \eqref{eq:july24_7} of the differential $\delta$,
we have
\begin{eqnarray*}
(\delta\iota_\xi\tilde\gamma)_{\lambda_2,\cdots,\lambda_{k+1}}(a_2,\cdots,a_{k+1})
=\sum_{i=2}^{k+1} (-1)^{i} {a_i}_{\lambda_i}
\Big\langle m(x_1),\tilde\gamma_{\lambda_1,\lambda_2,\stackrel{i}{\check{\cdots}},\lambda_{k+1}}
(a_1,a_2,\stackrel{i}{\check{\cdots}},a_{k+1}) \Big\rangle \\
+ \sum_{\substack{i,j=2\\i<j}}^{k+1} (-1)^{k+i+j} 
\Big\langle m(x_1),
\tilde\gamma_{\lambda_1,\lambda_2,\stackrel{i}{\check{\cdots}}\stackrel{j}{\check{\cdots}},
\lambda_{k+1},\lambda_i+\lambda_j}
(a_1,a_2,\stackrel{i}{\check{\cdots}}\stackrel{j}{\check{\cdots}},a_{k+1},[{a_i}_{\lambda_i} a_j])
\Big\rangle\,, \nonumber
\end{eqnarray*}
and
\begin{eqnarray*}
(\iota_\xi\delta\tilde\gamma)_{\lambda_2,\cdots,\lambda_{k+1}}(a_2,\cdots,a_{k+1})
=\sum_{i=1}^{k+1} (-1)^{i+1} 
\Big\langle m(x_1),
{a_i}_{\lambda_i}
\Big(\tilde\gamma_{\lambda_1,\stackrel{i}{\check{\cdots}},\lambda_{k+1}}
(a_1,\stackrel{i}{\check{\cdots}},a_{k+1}) \Big) 
\Big\rangle\\
+ \sum_{\substack{i,j=1\\i<j}}^{k+1} (-1)^{k+i+j+1} 
\Big\langle m(x_1),
\tilde\gamma_{\lambda_1,\stackrel{i}{\check{\cdots}}\stackrel{j}{\check{\cdots}},
\lambda_{k+1},\lambda_i+\lambda_j}
(a_1,\stackrel{i}{\check{\cdots}}\stackrel{j}{\check{\cdots}},a_{k+1},[{a_i}_{\lambda_i} a_j])
\Big\rangle\,. \nonumber
\end{eqnarray*}
We then use the assumption that the $\lambda$-action of $A$ on $M$ is by derivations
of the commutative associative product of $M$, to get, from the above two equations,
\begin{eqnarray}\label{eq:nov23_1}
& \displaystyle{
(L_\xi\tilde\gamma)_{\lambda_2,\cdots,\lambda_{k+1}}(a_2,\cdots,a_{k+1})
=
\Big\langle m(x_1),
{a_1}_{\lambda_1}
\Big(\tilde\gamma_{\lambda_2,\cdots,\lambda_{k+1}}(a_2,\cdots,a_{k+1}) \Big) \Big\rangle 
}\nonumber\\
& \displaystyle{
+
\sum_{i=2}^{k+1} (-1)^{i} 
\Big\langle \big({a_i}_{\lambda_i}m(x_1)\big),
\tilde\gamma_{\lambda_1,\lambda_2,\stackrel{i}{\check{\cdots}},\lambda_{k+1}}
(a_1,a_2,\stackrel{i}{\check{\cdots}},a_{k+1}) \Big\rangle 
}\\
& \displaystyle{
-
\sum_{j=2}^{k+1}
\Big\langle m(x_1),
\tilde\gamma_{\lambda_2,\cdots,\lambda_1+\lambda_j,\cdots,\lambda_{k+1}}
(a_2,\cdots,[{a_1}_{\lambda_1} a_j],\cdots,a_{k+1})\Big\rangle\,. 
}\nonumber
\end{eqnarray}

We next introduce a Lie algebra structure on $\mf g=\Pi\tilde\Gamma_1$ 
and the corresponding
$\mf g$-structure on the complex $(\tilde\Gamma^\bullet,\delta)$.
Define the following bracket on the space $A\otimes M[[x]]$:
\begin{eqnarray}\label{eq:nov23_2}
[a\otimes m(x),b\otimes n(x)]
=
[a_{\partial_{x_1}}b]\otimes m(x_1)n(x)\,\big|_{x_1=x} \\
-a\otimes \big\langle n(x_1),b_{\lambda_1}m(x)\big\rangle
+b\otimes \big\langle m(x_1),a_{\lambda_1}n(x)\big\rangle\,,\nonumber
\end{eqnarray}
where, as before, $\langle\,,\,\rangle$ in the RHS denotes the contraction of $x_1$ with $\lambda_1$
defined in \eqref{eq:nov22_3}.
\begin{lemma}\label{lem:nov23}
\begin{enumerate}
\alphaparenlist
\item The bracket \eqref{eq:nov23_2} on $A\otimes M[[x]]$
induces a well-defined Lie algebra bracket on the space $\mf g=\Pi\tilde\Gamma_1$.
\item The operator $\partial\otimes1+1\otimes\partial$ on $A\otimes M[[x]]$
is a derivation of the bracket \eqref{eq:nov23_2}.
In particular, $\partial$ defined in \eqref{eq:nov22_1} is a derivation 
of the Lie algebra $\mf g=\Pi\tilde\Gamma_1$,
and $\mf g^\partial=\Pi\Gamma_1\subset\mf g$ is a Lie subalgebra.
\end{enumerate}
\end{lemma}
\begin{proof}
Notice that, by the definition \eqref{eq:nov22_3} of the inner product $\langle\,,\,\rangle$,
\begin{equation}\label{eq:nov23_4}
\langle f(x_1),\lambda_1g(x_1)\rangle = \langle \partial_{x_1}f(x_1),g(\lambda_1)\rangle\,.
\end{equation}
Hence, by \eqref{eq:nov23_2} and the sesquilinearity conditions, we have
$$
\big[(\partial\otimes1+1\otimes\partial_x)\big(a\otimes m(x)\big),b\otimes n(x)\big]
=
-(\partial\otimes1+1\otimes\partial_x)
\Big(
a\otimes \big\langle n(x_1),b_{\lambda_1}m(x)\big\rangle
\big)\,,
$$
and
\begin{eqnarray*}
&& \big[a\otimes m(x),(\partial\otimes1+1\otimes\partial_x)\big(b\otimes n(x)\big)\big] \\
&& =
(\partial\otimes1+1\otimes\partial_x)
\Big(
[a_{\partial_{x_1}}b]\otimes m(x_1)n(x)\,\big|_{x_1=x} 
+b\otimes \big\langle m(x_1),a_{\lambda_1}n(x)\big\rangle
\Big)\,.
\end{eqnarray*}
It follows that $(\partial\otimes1+1\otimes\partial_x)$ is a derivation of the bracket \eqref{eq:nov23_2},
and that \eqref{eq:nov23_2} induces a well-defined bracket on the quotient
$\tilde\Gamma_1=A\otimes M[[x]]\big/(\partial\otimes1+1\otimes\partial_x)(A\otimes M[[x]])$.
Next, let us prove skew-symmetry. We have
\begin{eqnarray*}
[a\otimes m(x),b\otimes n(x)]
+ [b\otimes n(x),a\otimes m(x)]
=
\Big(\big([a_{\partial_{x_1}}b]+ [b_{\partial_{x}}a]\big)\otimes m(x_1)n(x)\Big)\,\Big|_{x_1=x} 
\,,\nonumber
\end{eqnarray*}
and the RHS belongs to $(\partial\otimes1+1\otimes\partial_x)(A\otimes M[[x]])$,
due to the skew-symmetry of the $\lambda$-bracket on $A$.
For part (a), we are left to prove the Jacobi identity.
Applying twice \eqref{eq:nov23_2}, we have
\begin{eqnarray}\label{eq:nov23_3}
&&
 [a\otimes m(x),[b\otimes n(x),c\otimes p(x)]] 
=
[a_{\partial_{x_1}}[b_{\partial_{x_2}}c]]\otimes m(x_1)n(x_2)p(x)\,\Big|_{x_1=x_2=x} 
\\
&&\,\,\,\,\,\,\,\,\, -
[a_{\partial_{x_1}}b]\otimes m(x_1)\big\langle p(x_2),c_{\lambda_2}n(x)\big\rangle\,\big|_{x_1=x} 
+
[a_{\partial_{x_1}}c]\otimes m(x_1)\big\langle n(x_2),b_{\lambda_2}p(x)\big\rangle\,\big|_{x_1=x}
\nonumber \\
&&\,\,\,\,\,\,\,\,\, +
[b_{\partial_{x_2}}c]\otimes \big\langle m(x_1),a_{\lambda_1}(n(x_2)p(x))\big\rangle\,\big|_{x_2=x}
-
a\otimes \big\langle\big\langle n(x_1)p(x_2),
[b_{\lambda_1}c]_{\lambda_1+\lambda_2}m(x)\big\rangle\big\rangle 
\nonumber \\
&&\,\,\,\,\,\,\,\,\, +
a\otimes \big\langle \big\langle p(x_2),c_{\lambda_2}n(x_1)\big\rangle,b_{\lambda_1}m(x)\big\rangle 
-
a\otimes \big\langle \big\langle n(x_2),b_{\lambda_2}p(x_1)\big\rangle,c_{\lambda_1}m(x)\big\rangle 
\nonumber \\
&&\,\,\,\,\,\,\,\,\, -
b\otimes \big\langle m(x_1),a_{\lambda_1}\big\langle p(x_2),c_{\lambda_2}n(x)\big\rangle\big\rangle
+
c\otimes \big\langle m(x_1),a_{\lambda_1}\big\langle n(x_2),b_{\lambda_2}p(x)\big\rangle\big\rangle\,,\nonumber
\end{eqnarray}
For the fifth term in the RHS we used \eqref{eq:nov23_4} and the following obvious
identity:
\begin{equation}\label{eq:nov23_5}
\langle f(x_1)g(x_2)\big|_{x_1=x_2},h(\lambda_2)\rangle 
= \langle\langle f(x_1)g(x_2),h(\lambda_1+\lambda_2)\rangle\rangle\,,
\end{equation}
where, in the RHS, we denote by $\langle\langle\,,\,\rangle\rangle$ the pairing 
of $\mb F[[x_1,x_2]]$ and $\mb F[\lambda_1,\lambda_2]$,
defined by contracting $x_1$ with $\lambda_1$ and $x_2$ with $\lambda_2$,
as in \eqref{eq:nov23_6}.
Similarly, we have
\begin{eqnarray}\label{eq:nov23_7}
&&
 [b\otimes n(x),[a\otimes m(x),c\otimes p(x)]] 
=
[b_{\partial_{x_2}}[a_{\partial_{x_1}}c]]\otimes m(x_1)n(x_2)p(x)\,\Big|_{x_1=x_2=x} 
\\
&&\,\,\,\,\,\,\,\,\, -
[b_{\partial_{x_1}}a]\otimes n(x_1)\big\langle p(x_2),c_{\lambda_2}m(x)\big\rangle\,\big|_{x_1=x} 
+
[b_{\partial_{x_2}}c]\otimes n(x_2)\big\langle m(x_1),a_{\lambda_1}p(x)\big\rangle\,\big|_{x_2=x}
\nonumber \\
&&\,\,\,\,\,\,\,\,\, +
[a_{\partial_{x_1}}c]\otimes \big\langle n(x_2),b_{\lambda_2}(m(x_1)p(x))\big\rangle\,\big|_{x_1=x}
-
b\otimes \big\langle\big\langle m(x_1)p(x_2),
[a_{\lambda_1}c]_{\lambda_1+\lambda_2}n(x)\big\rangle\big\rangle 
\nonumber \\
&&\,\,\,\,\,\,\,\,\, +
b\otimes \big\langle \big\langle p(x_2),c_{\lambda_2}m(x_1)\big\rangle,a_{\lambda_1}n(x)\big\rangle 
-
b\otimes \big\langle \big\langle m(x_1),a_{\lambda_1}p(x_2)\big\rangle,c_{\lambda_2}n(x)\big\rangle 
\nonumber \\
&&\,\,\,\,\,\,\,\,\, -
a\otimes \big\langle n(x_1),b_{\lambda_1}\big\langle p(x_2),c_{\lambda_2}m(x)\big\rangle\big\rangle
+
c\otimes \big\langle n(x_2),b_{\lambda_2}\big\langle m(x_1),a_{\lambda_1}p(x)\big\rangle\big\rangle\,,\nonumber
\end{eqnarray}
and, for the third term of Jacobi identity,
\begin{eqnarray}\label{eq:nov23_8}
&&
[[a\otimes m(x),b\otimes n(x)],c\otimes p(x)] 
=
[[a_{\partial_{x_1}}b]_{\partial_{x_1}+\partial_{x_2}}c]\otimes m(x_1)n(x_2)p(x)
\,\big|_{x_1=x_2=x} 
\\
&&\,\,\,\,\,\,\,\,\, -
[a_{\partial_{x_1}}b]\otimes \big\langle p(x_2),c_{\lambda_2}(m(x_1)n(x))\big\rangle \,\big|_{x_1=x}
-
[a_{\partial_{x_1}}c]\otimes 
\big\langle n(x_2),b_{\lambda_2}m(x_1)\big\rangle p(x)\,\big|_{x_1=x} 
\nonumber \\
&&\,\,\,\,\,\,\,\,\, +
[b_{\partial_{x_2}}c]\otimes 
\big\langle m(x_1),a_{\lambda_1}n(x_2)\big\rangle
p(x)\,\big|_{x_2=x} 
+
c\otimes \big\langle\big\langle 
m(x_1)n(x_2),[a_{\lambda_1}b]_{\lambda_1+\lambda_2}p(x)
\big\rangle\big\rangle
\nonumber \\
&&\,\,\,\,\,\,\,\,\, -
c\otimes \big\langle 
\big\langle n(x_2),b_{\lambda_2}m(x_1)\big\rangle,a_{\lambda_1}p(x)\big\rangle
+
c\otimes \big\langle 
\big\langle m(x_1),a_{\lambda_1}n(x_2)\big\rangle,b_{\lambda_2}p(x)\big\rangle
\nonumber \\
&&\,\,\,\,\,\,\,\,\, +
a\otimes \big\langle p(x_2),c_{\lambda_2}\big\langle n(x_1),b_{\lambda_1}m(x)\big\rangle\big\rangle 
-
b\otimes \big\langle p(x_2),c_{\lambda_2}
\big\langle m(x_1),a_{\lambda_1}n(x)\big\rangle\big\rangle 
\nonumber\,.
\end{eqnarray}
We now combine equations \eqref{eq:nov23_3}, \eqref{eq:nov23_7} and \eqref{eq:nov23_8},
to get Jacobi identity.
In particular, 
the first terms in the RHS of \eqref{eq:nov23_3}, \eqref{eq:nov23_7} and \eqref{eq:nov23_8}
combine to zero, due to the Jacobi identity for the $\lambda$-bracket on $A$.
For the second terms in the RHS of \eqref{eq:nov23_3}, \eqref{eq:nov23_7} and \eqref{eq:nov23_8},
we use the skew-symmetry of the $\lambda$-bracket on $A$
and the Leibniz rule for the $\lambda$-action of $A$ on $M$,
to conclude that their combination belongs to $(\partial\otimes1+1\otimes\partial_x)(A\otimes M[[x]])$.
The third term in the RHS of \eqref{eq:nov23_3} combines with
the fourth term in the RHS of \eqref{eq:nov23_7} and the third term in the RHS of \eqref{eq:nov23_8}
to give zero,
and similarly for the combination of 
the fourth term in the RHS of \eqref{eq:nov23_3}, 
the third term in the RHS of \eqref{eq:nov23_7} 
and the fourth term in the RHS of \eqref{eq:nov23_8}.
Furthermore, the combination of the fifth, sixth and seventh terms in the RHS of \eqref{eq:nov23_3},
the eighth term in the RHS of \eqref{eq:nov23_7} 
and the eighth term in the RHS of \eqref{eq:nov23_8} give
$$
a\otimes \Big\langle\Big\langle n(x_1)p(x_2),
\Big\{
-[b_{\lambda_1}c]_{\lambda_1+\lambda_2}m(x)
+b_{\lambda_1}c_{\lambda_2}m(x)
-c_{\lambda_2}b_{\lambda_1}m(x)
\Big\}
\Big\rangle\Big\rangle \,,
$$
which is zero due to the Jacobi identity for the $\lambda$-actio of $A$ on $M$.
Similarly for the remaining terms in \eqref{eq:nov23_3}, \eqref{eq:nov23_7} and \eqref{eq:nov23_8}.
We are left to prove part (b). We have
\begin{eqnarray}\label{eq:nov24_1}
&& \big[(\partial\otimes1+1\otimes\partial)(a\otimes m(x)),b\otimes n(x)\big] \\
&& =
[\partial a_{\partial_{x_1}}b]\otimes m(x_1)n(x)\,\big|_{x_1=x} 
+[a_{\partial_{x_1}}b]\otimes (\partial m(x_1))n(x)\,\big|_{x_1=x} \nonumber\\
&& -(\partial a)\otimes \big\langle n(x_1),b_{\lambda_1}m(x)\big\rangle
-a\otimes \big\langle n(x_1),b_{\lambda_1}(\partial m(x))\big\rangle \nonumber\\
&& +b\otimes \big\langle m(x_1),(\partial a)_{\lambda_1}n(x)\big\rangle
+b\otimes \big\langle (\partial m(x_1)),a_{\lambda_1}n(x)\big\rangle\,,\nonumber
\end{eqnarray}
and
\begin{eqnarray}\label{eq:nov24_2}
&& \big[a\otimes m(x),(\partial\otimes1+1\otimes\partial)(b\otimes n(x))\big] \\
&& =
[a_{\partial_{x_1}}\partial b]\otimes m(x_1)n(x)\,\big|_{x_1=x} 
+[a_{\partial_{x_1}}b]\otimes m(x_1)(\partial n(x))\,\big|_{x_1=x}  \nonumber\\
&& -a\otimes \big\langle n(x_1),(\partial b)_{\lambda_1}m(x)\big\rangle
-a\otimes \big\langle (\partial n(x_1)),b_{\lambda_1}m(x)\big\rangle \nonumber\\
&& +(\partial b)\otimes \big\langle m(x_1),a_{\lambda_1}n(x)\big\rangle
+b\otimes \big\langle m(x_1),a_{\lambda_1}(\partial n(x))\big\rangle\,.\nonumber
\end{eqnarray}
Putting equations \eqref{eq:nov24_1} and \eqref{eq:nov24_2} together we get
\begin{eqnarray*}
& (\partial\otimes1+1\otimes\partial)\big[a\otimes m(x),b\otimes n(x)\big]
=
\big[(\partial\otimes1+1\otimes\partial)(a\otimes m(x)),b\otimes n(x)\big] \\
& +
\big[a\otimes m(x),(\partial\otimes1+1\otimes\partial)(b\otimes n(x))\big]\,.
\end{eqnarray*}
This completes the proof of the lemma.
\end{proof}
\begin{proposition}\label{prop:vsep30_1}
The basic cohomoloy complex $(\tilde\Gamma^\bullet,\delta)$ admits a $\mf g$-structure,
$\varphi:\,\hat{\mf g}\to\End\tilde\Gamma^\bullet$,
where $\mf g=\Pi\tilde\Gamma_1$ is the Lie algebra with the Lie bracket induced by \eqref{eq:nov23_2},
given by
$\varphi(\partial_\eta)=\delta,\,\varphi(\eta\xi)=\iota_\xi,\,\varphi(\xi)=L_\xi,\,\xi\in\tilde\Gamma_1$.
The corresponding reduced (by $\partial$) $\mf g^\partial$-complex is $(\Gamma^\bullet,\delta)$.
\end{proposition}
\begin{proof}
In view of Remark \ref{rem:vsep} and Proposition \ref{prop:dic17_1}, we only have to check that
\begin{equation}\label{eq:nov24_3}
[L_{\xi_1},\iota_{\xi_2}]\,=\,\iota_{[\xi_1,\xi_2]}\,,
\end{equation}
where, for $\xi\in\tilde\Gamma_1$, $\iota_\xi$ is given by \eqref{eq:nov22_2} 
and $L_\xi$ is given by \eqref{eq:nov23_1}.
For $i=1,2$, let then $a_i\otimes m_i(x)\in A\otimes M[[x]]$ be a representative 
of $\xi_i\in\tilde\Gamma_1$.
We have
\begin{eqnarray}\label{eq:nov24_4}
&& (L_{\xi_1}\iota_{\xi_2}\tilde\gamma)_{\lambda_3,\cdots,\lambda_{k+1}}(a_3,\cdots,a_{k+1}) 
\\
&&\,\,\,\,\,\,\,\,\,
=
\big\langle m_1(x_1),
{a_1}_{\lambda_1}
\big\langle m_2(x_2),
\tilde\gamma_{\lambda_2,\lambda_3,\cdots,\lambda_{k+1}}(a_2,a_3,\cdots,a_{k+1})
\big\rangle\big\rangle
\nonumber\\
&&\,\,\,\,\,\,\,\,\,
+
\sum_{i=3}^{k+1} (-1)^i
\big\langle ({a_i}_{\lambda_i}m_1(x_1)),
\big\langle m_2(x_2),
\tilde\gamma_{\lambda_1,\lambda_2,\lambda_3,\stackrel{i}{\check{\cdots}},\lambda_{k+1}}
(a_1,a_2,a_3,\stackrel{i}{\check{\cdots}},a_{k+1})
\big\rangle\big\rangle 
\nonumber\\
&&\,\,\,\,\,\,\,\,\,
-
\sum_{j=3}^{k+1} 
\big\langle m_1(x_1),
\big\langle m_2(x_2),
\tilde\gamma_{\lambda_2,\lambda_3,\cdots,\lambda_1+\lambda_j,\cdots,\lambda_{k+1}}
(a_2,a_3,\cdots,[{a_1}_{\lambda_1}a_j],\cdots,a_{k+1})
\big\rangle\big\rangle\,,\nonumber
\end{eqnarray}
where, as in\eqref{eq:nov23_6}, with $\langle\,,\,\rangle$ we contract $x_1$ with $\lambda_1$
and $x_2$ with $\lambda_2$.
Similarly, we have
\begin{eqnarray}\label{eq:nov24_5}
&& (\iota_{\xi_2}L_{\xi_1}\tilde\gamma)_{\lambda_3,\cdots,\lambda_{k+1}}(a_3,\cdots,a_{k+1}) 
\\
&&\,\,\,\,\,\,\,\,\,
=
\big\langle m_2(x_2),
\big\langle m_1(x_1),
{a_1}_{\lambda_1}\big(
\tilde\gamma_{\lambda_2,\cdots,\lambda_{k+1}}(a_2,\cdots,a_{k+1})
\big)
\big\rangle\big\rangle
\nonumber\\
&&\,\,\,\,\,\,\,\,\,
+
\sum_{i=2}^{k+1} (-1)^i
\big\langle m_2(x_2),
\big\langle ({a_i}_{\lambda_i}m_1(x_1)),
\tilde\gamma_{\lambda_1,\lambda_2,\stackrel{i}{\check{\cdots}},\lambda_{k+1}}
(a_1,a_2,\stackrel{i}{\check{\cdots}},a_{k+1})
\big\rangle\big\rangle 
\nonumber\\
&&\,\,\,\,\,\,\,\,\,
-
\sum_{j=2}^{k+1} 
\big\langle m_2(x_2),
\big\langle m_1(x_1),
\tilde\gamma_{\lambda_2,\cdots,\lambda_1+\lambda_j,\cdots,\lambda_{k+1}}
(a_2,\cdots,[{a_1}_{\lambda_1}a_j],\cdots,a_{k+1})
\big\rangle\big\rangle\,.\nonumber
\end{eqnarray}
Combining equations \eqref{eq:nov24_4} and \eqref{eq:nov24_5}, we get
\begin{eqnarray}\label{eq:nov24_6}
&& ([L_{\xi_1},\iota_{\xi_2}]\tilde\gamma)_{\lambda_3,\cdots,\lambda_{k+1}}(a_3,\cdots,a_{k+1}) 
\\
&&\,\,\,\,\,\,\,\,\,
=
\big\langle\big\langle m_1(x_1),
 ({a_1}_{\lambda_1}m_2(x_2))\big\rangle,
\tilde\gamma_{\lambda_2,\lambda_3,\cdots,\lambda_{k+1}}(a_2,a_3,\cdots,a_{k+1})
\big\rangle
\nonumber\\
&&\,\,\,\,\,\,\,\,\,
-
\big\langle\big\langle m_2(x_2),
 ({a_2}_{\lambda_2}m_1(x_1))\big\rangle,
\tilde\gamma_{\lambda_1,\lambda_3,\cdots,\lambda_{k+1}}
(a_1,a_3,\cdots,a_{k+1})
\big\rangle 
\nonumber\\
&&\,\,\,\,\,\,\,\,\,
+
\big\langle\big\langle 
m_1(x_1) m_2(x_2),
\tilde\gamma_{\lambda_1+\lambda_2,\lambda_3,\cdots,\lambda_{k+1}}
([{a_1}_{\lambda_1}a_2],a_3,\cdots,a_{k+1})
\big\rangle\big\rangle\,,\nonumber
\end{eqnarray}
where, for the first term, we used the fact that the $\lambda$-action of $A$ on $M$
is by derivations of the commutative associative product on $M$.
To conclude, we use equations \eqref{eq:nov23_4} and \eqref{eq:nov23_5} 
to rewrite the RHS of \eqref{eq:nov24_6}
as $(\iota_\xi\tilde\gamma)_{\lambda_3,\cdots,\lambda_{k+1}}(a_3,\cdots,a_{k+1})$,
where
\begin{eqnarray*}
\xi &=& 
a_2\otimes\big\langle m_1(x_1),({a_1}_{\lambda_1}m_2(x))\big\rangle
-a_1\otimes\big\langle m_2(x_2),({a_2}_{\lambda_2}m_1(x))\big\rangle \\
&& +[{a_1}_{\partial_{x_1}}a_2]\otimes m_1(x_1) m_2(x)\,\big|_{x_1=x}
\,=\,
[\xi_1,\xi_2]\,.
\end{eqnarray*}
\end{proof}

\vspace{3ex}
\subsection{The space of chains $C_\bullet$.}~~
\label{sec:3.4}
Recall from Theorem \ref{th:red} that the cohomology complex $\Gamma^\bullet$
is a subcomplex of the cohomology complex $C^\bullet$
defined in Section \ref{sec:1.3_b}.
One may ask whether, for a reduced $h$-chain $\xi\in\Gamma_h$,
there is a natural way to extend the definition of the contraction operator 
$\iota_\xi$ to the complex $C^\bullet$.
In order to formulate the statement, in Theorem \ref{th:ago21} below, 
we first define a new space of chains, 
obtained by dualizing the definition of the complex $C^\bullet$.

We let $C_\bullet=\bigoplus_{k\in\mb Z_+}C_k$,
where
$C_0=\{m\in M\,|\,\partial m=0\}\,\big(=\Gamma_0\big)$,
and, for $k\geq1$, 
we define the space $C_k$ of $k$-\emph{chains} of $A$ with coefficients in $M$
as the quotient of the space
$A^{\otimes k} \otimes \Hom(\mb F[\lambda_1,\dots,\lambda_{k-1}],M)$
by the following relations:
\begin{enumerate}
\item[D1.] 
$a_1\otimes\cdots\partial a_i\cdots\otimes a_{k}\otimes\phi
\equiv -a_1\otimes\cdots\otimes a_{k}\otimes(\lambda^*_i\phi)$,
for every $1\leq i\leq k-1$;
\item[D2.] 
$a_1\otimes\cdots\otimes a_{k-1}\otimes(\partial a_k)\otimes \phi
\equiv a_1\otimes\cdots\otimes a_{k}\otimes((\lambda_1^*+\cdots+\lambda_{k-1}^*-\partial)\phi)$;
\item[D3.] 
$a_{\sigma(1)}\otimes\cdots\otimes a_{\sigma(k)}\otimes(\sigma^*\phi)
\equiv \text{sign}(\sigma) a_1\otimes\cdots\otimes a_{k}\otimes\phi$,
for every permutation $\sigma\in S_k$,
where $\sigma^*\phi\in \Hom(\mb F[\lambda_1,\dots,\lambda_{k-1}],M)$ 
is defined by 
\begin{equation}\label{eq:ago20_4}
(\sigma^*\phi)(f(\lambda_1,\cdots,\lambda_{k-1}))
=\phi\big(f(\lambda_{\sigma(1)},\cdots,\lambda_{\sigma(k-1)})
\big|_{\lambda_k\mapsto{\lambda_k}_\dagger}
\big)\,,
\end{equation}
where in the RHS we have to replace $\lambda_k$ by 
${\lambda_k}_\dagger=-\lambda_1-\cdots-\lambda_{k-1}+\partial^M$ and 
move $\partial^M$ to the left of $\phi$.
\end{enumerate}

For example, $C_1=(A\otimes M)/\partial(A\otimes M)$.
In particular, in $C_1$ it is not necessarily true that $a\otimes m$ is equivalent to zero 
for every torsion element $a$ of the $\mb F[\partial]$-module $A$.
On the other hand the analogue of Lemma \ref{rem:ago19} holds for $k\geq 2$:
\begin{lemma}\label{lem:ago20_1}
If $k\geq 2$ and $a_i\in\Tor A$ for some $i$,
we have $a_1\otimes\cdots\otimes a_{k}\otimes\phi = 0$ in $C_k$.
\end{lemma}
\begin{proof}
For $1\leq i\leq k-1$, relation D1. is the same as relation C1., hence the same argument 
as in the proof of Lemma \ref{rem:ago19} works.
Similarly, for $i=k$, if $P(\partial)a_k=0$, we have by the relation D2.,
$$
0=a_1\otimes\cdots\otimes a_{k-1}\otimes(P(\partial)a_k)\otimes\phi
= a_1\otimes\cdots\otimes a_{k}\otimes(P(\lambda_1^*+\cdots+\lambda_{k-1}^*-\partial)\phi)\,,
$$
and to conclude the lemma we need to prove that the linear endomorphism
$P(\lambda_1^*+\cdots+\lambda_{k-1}^*-\partial)$ 
of $\Hom(\mb F[\lambda_1,\dots,\lambda_{k-1}],M)$ is surjective.
In other words, given $\phi\in\Hom(\mb F[\lambda_1,\dots,\lambda_{k-1}],M)$,
we want to find $\psi\in\Hom(\mb F[\lambda_1,\dots,\lambda_{k-1}],M)$
such that $P(\lambda_1^*+\cdots+\lambda_{k-1}^*-\partial)\psi=\phi$.
Suppose, for simplicity, that the polynomial $P$ is monic of degree $N$.
Hence
$$
P(\lambda_1^*+\cdots+\lambda_{k-1}^*-\partial)
=(\lambda_1^*)^N+\sum_{n=0}^N\partial^n R_n(\lambda_1^*,\cdots,\lambda_{k-1}^*)\,,
$$
where the polynomials $R_n\in\mb F[\lambda_1^*,\dots,\lambda_{k-1}^*]$,
considered as polynomials in $\lambda_1$, have degree strictly less than $N$.
Then $\psi$ can be constructed recursively by saying that
$\psi(\lambda_1^{n_1}\lambda_2^{n_2}\cdots\lambda_{k-1}^{n-1})=0$ for $n_1<N$,
and 
$$
\psi(\lambda_1^{N+n_1}\lambda_2^{n_2}\cdots\lambda_{k-1}^{n-1})
=
\phi(\lambda_1^{n_1}\cdots\lambda_{k-1}^{n-1})
-\sum_{n=0}^N{\partial^M}^n
\psi\big(R_n(\lambda_1,\cdots,\lambda_{k-1})
\lambda_1^{n_1}\cdots\lambda_{k-1}^{n-1}\big)\,.
$$
Since the RHS only depends on 
$\psi(\lambda_1^{m_1}\lambda_2^{m_2}\cdots\lambda_{k-1}^{m-1})$
with $m_1<N+n_1$, the above equation defines $\psi$ by induction on $n_1$.
Clearly, 
$P(\lambda_1^*+\cdots+\lambda_{k-1}^*-\partial)\psi=\phi$.
\end{proof}
In analogy with the notation used in Section \ref{sec:1.3_a},
we introduce the space $\bar C_\bullet=\bigoplus_{k\in\mb Z_+}\bar C_k$,
by taking the quotient of the space $C_\bullet$ by the torsion of $A$.
More precisely, 
let $\bar C_0=\{m\in M\,|\,\partial m=0\}=C_0$
and, for $k\geq 1$, $\bar C_k$ is the quotient of the space
$\bar A^{\otimes k}\otimes\Hom(\mb F[\lambda_1,\dots,\lambda_{k-1}],M)$,
where $\bar A=A/\Tor A$,
by the relations D1., D2. and D3. above.
In particular, by Lemma \ref{lem:ago20_1},
$\bar C_k=C_k$ for $k\neq 1$,
and there is a natural surjective map $C_1\twoheadrightarrow\bar C_1$.

We next want to describe the relation between the spaces $C_k$ and $\Gamma_k$.
In particular, 
we are going to define a canonical map $\chi_k:\,C_k\to\Gamma_k$,
and we will prove in Proposition \ref{th:ago20} that,
if the $\mb F[\partial]$-module $A$ decomposes as direct sum of its torsion
and a free submodule, $\chi_k$ factors through an isomorphism 
$\bar C_k\simeq\Gamma_k$.

For $k\geq1$, let $\rho_k:\,\Hom(\mb F[\lambda_1,\dots,\lambda_{k}],M)
\twoheadrightarrow \Hom(\mb F[\lambda_1,\dots,\lambda_{k-1}],M)$,
be the restriction map associated to the inclusion 
$\mb F[\lambda_1,\dots,\lambda_{k-1}]\subset\mb F[\lambda_1,\dots,\lambda_{k}]$.
Let 
$$
\chi_k\,:\,\,\Hom(\mb F[\lambda_1,\dots,\lambda_{k-1}],M)
\,\hookrightarrow\, \Hom(\mb F[\lambda_1,\dots,\lambda_{k}],M)\,,
$$
be the injective linear map defined by
\begin{equation}\label{eq:ago20_1}
(\chi_k\phi)\big(f(\lambda_1,\cdots,\lambda_{k})\big)
=\phi\big(f(\lambda_1,\cdots,\lambda_{k-1},{\lambda_k}_\dagger)\big)\,,
\end{equation}
where in the RHS we let ${\lambda_k}_\dagger=-\sum_{j=1}^{k-1}\lambda_j+\partial^M$
and we move $\partial^M$ to the left.
\begin{lemma}\label{lem:ago20_2}
\begin{enumerate}
\alphaparenlist
\item
We have $\rho_k\circ\chi_k=\id$ on $\Hom(\mb F[\lambda_1,\dots,\lambda_{k-1}],M)$.
Hence $\chi_k\circ\rho_k$ is a projection operator 
on $\Hom(\mb F[\lambda_1,\dots,\lambda_{k}],M)$,
whose image is naturally isomorphic to $\Hom(\mb F[\lambda_1,\dots,\lambda_{k-1}],M)$.
\item
The image of $\chi_k$ consists of the elements 
$\phi\in\Hom(\mb F[\lambda_1,\dots,\lambda_{k}],M)$ such that
\begin{equation}\label{eq:ago20_5}
(\lambda_1^*+\cdots+\lambda_{k}^*)\phi\,=\,\partial\phi\,.
\end{equation}
\item
We have the commutation relations
\begin{equation}\label{eq:ago20_6}
\lambda_i^*\circ\chi_k=\chi_k\circ\lambda_i^*\,\,\,\,\forall 1\leq i\leq k-1
\,\,,\,\,\,\,
\lambda_k^*\circ\chi_k=\chi_k\circ(-\lambda_1^*-\cdots-\lambda_{k-1}^*+\partial)\,,
\end{equation}
where $\lambda_i^*$ is the linear endomorphism of $\Hom(\mb F[\lambda_1,\dots,\lambda_{k}],M)$
defined by \eqref{eq:ago20_2}.
\item
For every permutation $\sigma\in S_k$ we have
\begin{equation}\label{eq:ago20_7}
\sigma^*\circ\chi_k=\chi_k\circ\sigma^*\,,
\end{equation}
where $\sigma^*$ in the LHS denotes the endomorphism 
of $\Hom(\mb F[\lambda_1,\dots,\lambda_{k}],M)$
defined by \eqref{eq:ago20_3},
while in the RHS it denotes the endomorphism 
of $\Hom(\mb F[\lambda_1,\dots,\lambda_{k-1}],M)$
defined by \eqref{eq:ago20_4}.
\end{enumerate}
\end{lemma}
\begin{proof}
Part (a) is obvious.
Given $\phi\in\Hom(\mb F[\lambda_1,\dots,\lambda_{k-1}],M)$, we have, 
by the definition \eqref{eq:ago20_1} of $\chi_k$,
$$
\big((\lambda_1^*+\cdots+\lambda_{k}^*-\partial)\chi_k\phi\big)
\big(f(\lambda_1,\cdots,\lambda_{k})\big)
=
(\chi_k\phi)
\big((\lambda_1+\cdots+\lambda_{k}-\partial^M)
f(\lambda_1,\cdots,\lambda_{k})\big) = 0\,,
$$
namely $\chi_k\phi$ satisfies equation \eqref{eq:ago20_5}.
Conversely, if $\phi\in\Hom(\mb F[\lambda_1,\dots,\lambda_{k}],M)$ 
solves equation \eqref{eq:ago20_5}, we have,
by Taylor expanding in 
${\lambda_k}_\dagger-\lambda_k=-\lambda_1-\cdots-\lambda_{k}+\partial^M$,
\begin{eqnarray*}
&& (\chi_k\rho_k\phi)\big(f(\lambda_1,\cdots,\lambda_{k})\big)
=
\phi\big(f(\lambda_1\cdots,\lambda_{k-1},{\lambda_k}_\dagger)\big) \\
&& =
\sum_{n\in\mb Z_+}\frac1{n!}
\big((-\lambda_1^*-\cdots-\lambda_{k}^*+\partial)^n\phi\big)
\big((\partial_{\lambda_k}^nf)(\lambda_1,\cdots,\lambda_{k})\big) 
=
\phi\big(f(\lambda_1,\cdots,\lambda_{k})\big)\,.
\end{eqnarray*}
Hence, $\phi$ is in the image of $\chi_k$, as we wanted.
This proves part (b).
For part (c), the first equation in \eqref{eq:ago20_6} is clear. 
The second equation follows by part (b).
We are left to prove part (d).
Given $\phi\in\Hom(\mb F[\lambda_1,\dots,\lambda_{k-1}],M)$ we have,
for every permutation $\sigma\in S_k$,
$$
(\sigma^*\chi_k\phi)\big(f(\lambda_1,\cdots,\lambda_{k})\big)
=
\phi\big(f(\lambda_{\sigma(1)},\cdots,\lambda_{\sigma(k)})
\big|_{\lambda_k\mapsto{\lambda_k}_\dagger}\big)\,,
$$
and
$$
(\chi_k\sigma^*\phi)\big(f(\lambda_1,\cdots,\lambda_{k})\big)
=
\phi\Big(f(
\lambda_{\sigma(1)},\cdots,\lambda_{\sigma(k-1)},
-\lambda_{\sigma(1)}-\cdots-\lambda_{\sigma(k-1)}+\partial^M)
\Big|_{\lambda_k\mapsto{\lambda_k}_\dagger}\Big)\,.
$$
Equation \eqref{eq:ago20_7} follows by the fact that, for $\sigma(k)\neq k$,
when we replace
$\lambda_k$ by ${\lambda_k}_\dagger=-\lambda_1-\cdots-\lambda_{k-1}+\partial^M$,
the expression 
$-\lambda_{\sigma(1)}-\cdots-\lambda_{\sigma(k-1)}+\partial^M$ 
is replaced by $\lambda_{\sigma(k)}$.
\end{proof}

We extend $\chi_k$ to an injective linear map
$\chi_k:\,A^{\otimes k}\otimes\Hom(\mb F[\lambda_1,\dots,\lambda_{k-1}],M)
\hookrightarrow A^{\otimes k}\otimes\Hom(\mb F[\lambda_1,\dots,\lambda_{k}],M)$,
given by
\begin{equation}\label{eq:sat_4}
\chi_k(a_1\otimes\cdots\otimes a_{k}\otimes\phi)
= a_1\otimes\cdots\otimes a_{k}\otimes\chi_k(\phi)\,.
\end{equation}
Moreover, we denote by 
$\langle C1,C2\rangle\subset A^{\otimes k}\otimes\Hom(\mb F[\lambda_1,\dots,\lambda_{k}],M)$
the subspace generated by the relations C1. and C2. from Section \ref{sec:3.2},
and by 
$\langle D1,D2,D3\rangle\subset A^{\otimes k}\otimes\Hom(\mb F[\lambda_1,\dots,\lambda_{k-1}],M)$
the subspace generated by the relations D1., D2. and D3.
\begin{proposition}\label{prop:ago20}
\begin{enumerate}
\alphaparenlist
\item
$\chi_k\big(\langle D1,D2,D3\rangle\big)\subset \langle C1,C2\rangle$.
\item
For every $x\in A^{\otimes k}\otimes\Hom(\mb F[\lambda_1,\dots,\lambda_{k-1}],M)$,
we have $\partial\chi_k(x)\in\langle C1,C2\rangle$.
\item
$\chi_k$ induces a well-defined linear map $\chi_k:\,C_k\to\Gamma_k$.
\end{enumerate}
\end{proposition}
\begin{proof}
For $1\leq i\leq k-1$, we have
\begin{eqnarray*}
&& \chi_k\big(
a_1\otimes\cdots(\partial a_i)\cdots\otimes a_{k}\otimes\phi
+a_1\otimes\cdots\otimes a_{k}\otimes(\lambda_i^*\phi)\big) \\
&& =
a_1\otimes\cdots(\partial a_i)\cdots\otimes a_{k}\otimes\chi_k(\phi)
+a_1\otimes\cdots\otimes a_{k}\otimes\chi_k(\lambda_i^*\phi)\big)\,,
\end{eqnarray*}
and this is in $\langle C1,C2\rangle$ thanks to Lemma \ref{lem:ago20_2}(c).
Similarly, by the second equation in \eqref{eq:ago20_6},
\begin{eqnarray*}
&& \chi_k\big(
 a_1\otimes\cdots\otimes a_{k-1}\otimes(\partial a_k)\otimes\phi
-a_1\otimes\cdots\otimes a_{k}\otimes(\lambda_1^*+\cdots+\lambda_{k-1}^*-\partial)\phi)\big) \\
&& =
a_1\otimes\cdots\otimes a_{k-1}\otimes(\partial a_k)\otimes \chi_k(\phi)
+a_1\otimes\cdots\otimes a_{k}\otimes
\lambda_k^*\chi_k(\phi)\,\in\langle C1,C2\rangle\,.
\end{eqnarray*}
Furthermore, by Lemma \ref{lem:ago20_2}(d), we have,
for every permutation $\sigma\in S_k$,
\begin{eqnarray*}
&& \chi_k\big(
a_1\otimes\cdots\otimes a_{k}\otimes\phi
-\text{sign}(\sigma)a_{\sigma(1)}\otimes\cdots\otimes a_{\sigma(k)}\otimes(\sigma^*\phi)\big) \\
&& 
= a_1\otimes\cdots\otimes a_{k}\otimes\chi_k(\phi)
-\text{sign}(\sigma)a_{\sigma(1)}\otimes\cdots\otimes a_{\sigma(k)}\otimes(\sigma^*\chi_k(\phi))
\,\in\langle C1,C2\rangle\,.
\end{eqnarray*}
This proves part (a).
From \eqref{eq:sfor_1} and Lemma \ref{lem:ago20_2}(b), we have
$$
\partial\chi_k(a_1\otimes\cdots\otimes a_{k}\otimes\phi)
\equiv a_1\otimes\cdots\otimes a_{k}\otimes
(-\lambda_1^*-\cdots-\lambda_{k}^*+\partial)\chi_k(\phi)
=0\,,
$$
thus proving (b). Part (c) follows from (a) and (b).
\end{proof}
\begin{proposition}\label{th:ago20}
If the Lie conformal algebra $A$ decomposes, as an $\mb F[\partial]$-module,
in a direct sum of the torsion and a complementary free $\mb F[\partial]$-submodule,
the identity map on $C_0=\{m\in M\,|\,\partial m=0\}$
and the maps $\chi_k:\,C_k\to\Gamma_k,\,k\geq1$,
factor through a bijective map $\bar C_\bullet\simeq\Gamma_\bullet$.
\end{proposition}
\begin{proof}
Suppose that the $\mb F[\partial]$-module $A$ decomposes as in \eqref{eq:anniv_1}.
By definition,
in the space $\bar C_k$ we have that
$a_1\otimes\cdots\otimes a_{k}\otimes\phi\equiv 0$
if one of the elements $a_i$ is in $T=\Tor A$.
The same is true in the space $\tilde\Gamma_k$ by Lemma \ref{rem:ago19}.
It follows that $\chi_k$ induces a well-defined map 
\begin{equation}\label{eq:ago21_2}
\chi_k:\,\bar C_k\to\Gamma_k\subset\tilde\Gamma_k\,.
\end{equation}
Moreover, in the space $\tilde\Gamma_k$ we have,
using relation C1., that
$$
(P_1(\partial)u_1)\otimes\cdots\otimes(P_{k}(\partial)u_{k})\otimes\phi
\equiv u_1\otimes\cdots\otimes u_{k}\otimes
\big(P_{1}(-\lambda_1^*)\cdots P_{k}(-\lambda_{k}^*)\phi\big)\,,
$$
for every $u_i\in U$ and $\phi\in\Hom(\mb F[\lambda_1,\dots,\lambda_{k}],M)$.
Hence,
we can identify the space $\tilde\Gamma_k$ with the quotient 
of the space $U^{\otimes k}\otimes\Hom(\mb F[\lambda_1,\dots,\lambda_{k}],M)$
by the relation C2.
Similarly, in the space $\bar C_k$ we have,
using the relations D1. and D2., that
\begin{eqnarray*}
&& (P_1(\partial)u_1)\otimes\cdots\otimes(P_{k}(\partial)u_{k})\otimes\phi \\
&& \equiv u_1\otimes\cdots\otimes u_{k}\otimes
\big(
P_{1}(-\lambda_1^*)\cdots P_{k-1}(-\lambda_{k-1}^*)
P_{k}(\lambda_1^*+\cdots+\lambda_{k-1}^*-\partial)
\phi\big)\,,
\end{eqnarray*}
for every $u_i\in U$ and $\phi\in\Hom(\mb F[\lambda_1,\dots,\lambda_{k-1}],M)$.
Hence,
we can identify the space $\bar C_k$ with the quotient 
of the space $U^{\otimes k}\otimes\Hom(\mb F[\lambda_1,\dots,\lambda_{k-1}],M)$
by the relation D3.
The map $\chi_k$ in \eqref{eq:ago21_2} is then induced by the map
$U^{\otimes k}\otimes\Hom(\mb F[\lambda_1,\dots,\lambda_{k-1}],M)
\to
U^{\otimes k}\otimes\Hom(\mb F[\lambda_1,\dots,\lambda_{k}],M)$,
given by
$$
u_1\otimes\cdots\otimes u_{k}\otimes\phi
\,\mapsto\,
u_1\otimes\cdots\otimes u_{k}\otimes\chi_k(\phi)\,,
$$
for every $u_i\in U$ and $\phi\in\Hom(\mb F[\lambda_1,\dots,\lambda_{k-1}],M)$.
Recalling \eqref{eq:sfor_1},
the action of $\partial$ on $\tilde\Gamma_k$ is induced by the map
$U^{\otimes k}\otimes\Hom(\mb F[\lambda_1,\dots,\lambda_{k}],M)
\to
U^{\otimes k}\otimes\Hom(\mb F[\lambda_1,\dots,\lambda_{k}],M)$,
given by
$$
u_1\otimes\cdots\otimes u_{k}\otimes\phi
\,\mapsto\,
u_1\otimes\cdots\otimes u_{k}\otimes
\big((-\lambda_1^*-\cdots-\lambda_{k}^*)\phi\big)\,.
$$
Hence,  the subspace $\Gamma_k\subset\tilde\Gamma_k$
is spanned by elements of the form
$u_1\otimes\cdots\otimes u_{k}\otimes\phi$,
such that $(-\lambda_1^*-\cdots-\lambda_{k}^*)\phi=0$.
By Lemma \ref{lem:ago20_2}(b), this is the same as the image of $\chi_k$.
Therefore the map \eqref{eq:ago21_2} is surjective.
Finally, injectiveness of \eqref{eq:ago21_2} is clear since, by Lemma \eqref{lem:ago20_2}(d),
relation D3. corresponds, via $\chi_k$, to relation C2.
\end{proof}

\vspace{3ex}
\subsection{Contraction operators acting on $C^\bullet$.}~~
\label{sec:3.5}
Assume, as in Section \ref{sec:3.3}, that $A$ is a Lie conformal algebra
and $M$ is an $A$-module endowed with a commutative, associative product
$\mu:\,M\otimes M\to M$,
such that $\partial^M:\,M\to M$, and $a_\lambda:\,M\to\mb C[\lambda]\otimes M$, 
satisfy the Leibniz rule.
Given an $h$-chain $x\in C_h$, we define the \emph{contraction operator}
$\iota_x:\, C^k\to C^{k-h},\,k\geq h$, 
in the same way as we defined, in Section \ref{sec:3.3},
the contraction operator associated to an element of $\tilde\Gamma_h$.
If $a_1\otimes\cdots\otimes a_{h}\otimes\phi
\in A^{\otimes h}\otimes\Hom(\mb F[\lambda_1,\dots,\lambda_{h-1}],M)$ 
is a representative of $x\in C_h$,
and $c\in C^k$, we let, for $h<k$,
\begin{equation}\label{eq:ago21_3}
\{{a_{h+1}}_{\lambda_{h+1}}\cdots{a_{k-1}}_{\lambda_{k-1}}a_k\}_{\iota_x c}
=
(\chi_h\phi)^\mu\big(
\{{a_1}_{\lambda_1}\cdots
{a_h}_{\lambda_{h}}{a_{h+1}}_{\lambda_{h+1}}
\cdots{a_{k-1}}_{\lambda_{k-1}}a_k\}_c
\big)\,,
\end{equation}
where, in the RHS, $\phi^\mu$ is defined by \eqref{eq:ago16_2}
and $\chi_h$ is given by \eqref{eq:ago20_1}.
For $h=k$ , equation \eqref{eq:ago21_3} has to be modified as follows:
\begin{equation}\label{eq:giulia}
\iota_x c
=
\int \phi^\mu\big(
\{{a_1}_{\lambda_1}\cdots
{a_{k-1}}_{\lambda_{k-1}}a_k\}_c
\big)\,\in M/\partial^MM = C^0\,.
\end{equation}
\begin{lemma}\label{lem:ago21}
\begin{enumerate}
\alphaparenlist
\item
For $c\in C^k$, the RHS of \eqref{eq:ago21_3} does not depend on the choice of the representative
for $x$ in $A^{\otimes h}\otimes\Hom(\mb F[\lambda_1,\dots,\lambda_{h-1}],M)$.
Hence the contraction operator $\iota_x$ is well defined for $x\in C_h$.
\item
For $c\in C^k$, the RHS of \eqref{eq:ago21_3} satisfies conditions B1., B2. and B3.
Hence $\iota_xc\in C^{k-h}$.
\end{enumerate}
\end{lemma}
\begin{proof}
If 
$x=a_1\otimes\cdots(\partial a_i)\cdots\otimes a_{h}\otimes\phi
+a_1\cdots\otimes a_{h}\otimes(\lambda_i^*\phi)$, 
for $1\leq i\leq h-1$, we have
\begin{eqnarray*}
&\displaystyle{
\{{a_{h+1}}_{\lambda_{h+1}}\cdots{a_{k-1}}_{\lambda_{k-1}}a_k\}_{\iota_xc}
\,=\,
(\chi_h\phi)^\mu\big(
\{{a_1}_{\lambda_1}\cdots(\partial a_i)_{\lambda_i}\cdots{a_{k-1}}_{\lambda_{k-1}}a_k\}_c
}\\
&\displaystyle{
+\lambda_i
\{{a_1}_{\lambda_1}\cdots{a_{k-1}}_{\lambda_{k-1}}a_k\}_c
\big)\,,
}
\end{eqnarray*}
and this is zero since, by assumption, $c$ satisfies condition B1.
Similarly, if 
$x= a_1\otimes\cdots\otimes a_{h-1}\otimes(\partial a_h)\otimes\phi
-a_1\cdots\otimes a_{h}\otimes((\lambda_1^*+\cdots+\lambda_{h-1}^*-\partial)\phi)$, 
we have
\begin{eqnarray}\label{eq:ago21_4}
&&\displaystyle{
\{{a_{h+1}}_{\lambda_{h+1}}\cdots{a_{k-1}}_{\lambda_{k-1}}a_k\}_{\iota_xc}
\,=\,
(\chi_h\phi)^\mu\big(
\{{a_1}_{\lambda_1}\cdots{(\partial a_h)}_{\lambda_h}\cdots{a_{k-1}}_{\lambda_{k-1}}a_k\}_c 
}\\
&&\displaystyle{
-(\lambda_1+\cdots+\lambda_{h-1})
\{{a_1}_{\lambda_1}\cdots{a_{k-1}}_{\lambda_{k-1}}a_k\}_c
\big)
+(\chi_h\partial\phi)^\mu\big(
\{{a_1}_{\lambda_1}\cdots{a_{k-1}}_{\lambda_{k-1}}a_k\}_c
\big)\,.
}\nonumber
\end{eqnarray}
Using the condition B1. for $c$,
we can rewrite the RHS of \eqref{eq:ago21_4} as
$$
(\chi_h\phi)^\mu\big(
-(\lambda_1+\cdots+\lambda_{h-1}+\lambda_h)
\{{a_1}_{\lambda_1}\cdots{a_{k-1}}_{\lambda_{k-1}}a_k\}_c 
\big)
+(\chi_h\partial\phi)^\mu\big(
\{{a_1}_{\lambda_1}\cdots{a_{k-1}}_{\lambda_{k-1}}a_k\}_c
\big)\,,
$$
which is zero thanks to Lemma \ref{lem:ago20_2}(c).
Furthermore, if 
$x=a_1\otimes\cdots\otimes a_{h}\otimes\phi
-\text{sign}(\sigma)a_{\sigma(1)}\otimes\cdots\otimes a_{\sigma(h)}\otimes(\sigma^* \phi)$,
for a permutation $\sigma\in S_h$, we have
\begin{eqnarray}\label{eq:ago21_5}
&&\displaystyle{
\{{a_{h+1}}_{\lambda_{h+1}}\cdots{a_{k-1}}_{\lambda_{k-1}}a_k\}_{\iota_xc}
\,=\,
(\chi_h\phi)^\mu\big(
\{{a_1}_{\lambda_1}\cdots{a_{k-1}}_{\lambda_{k-1}}a_k\}_c
\big)
}\nonumber\\
&&\displaystyle{
\,\,\,\,\,\,\,\,\,
-\text{sign}(\sigma)
(\chi_h\sigma^*\phi)^\mu\big(
\{{a_1}_{\lambda_1}\cdots{a_{k-1}}_{\lambda_{k-1}}a_k\}_c
\big)
\,=\,
(\chi_h\phi)^\mu\big(
\{{a_1}_{\lambda_1}\cdots{a_{k-1}}_{\lambda_{k-1}}a_k\}_c
}\\
&&\displaystyle{
\,\,\,\,\,\,\,\,\,\,\,\,\,\,\,\,\,\,
-\text{sign}(\sigma)
\{{a_{\sigma(1)}}_{\lambda_{\sigma(1)}}\cdots
{a_{\sigma(h)}}_{\lambda_{\sigma(h)}}
{a_{h+1}}_{\lambda_{h+1}}
\cdots
{a_{k-1}}_{\lambda_{k-1}}a_k\}_c
\big)\,,
}\nonumber
\end{eqnarray}
where, in the second equality, we used Lemma \eqref{lem:ago20_2}(d)
and the definition \eqref{eq:ago20_3} of $\sigma^*$ acting 
on $\Hom(\mb F[\lambda_1,\dots,\lambda_h],M)$.
Clearly, the RHS of \eqref{eq:ago21_5} is zero since,
by assumption, $c$ satisfies condition B3.
This proves part (a).
For part (b), condition B1. for $\iota_xc$
follows immediately from the same condition on $c$.
We have
\begin{eqnarray}\label{eq:sep21_1}
&\displaystyle{
\{{a_{h+1}}_{\lambda_{h+1}}\cdots{a_{k-1}}_{\lambda_{k-1}}(\partial a_k)\}_{\iota_xc}
\,=\,
(\chi_h\phi)^\mu\big(
\{{a_1}_{\lambda_1}\cdots{a_{k-1}}_{\lambda_{k-1}}(\partial a_k)\}_c
\big) 
}\\
&\displaystyle{
= (\chi_h\phi)^\mu\big(
(\lambda_1+\cdots+\lambda_{k-1}+\partial^M)
\{{a_1}_{\lambda_1}\cdots{a_{k-1}}_{\lambda_{k-1}}a_k\}_c
\big)\,.
}\nonumber
\end{eqnarray}
By Lemmas \ref{lem:ago16} and \ref{lem:ago20_2}(c),
the RHS of \eqref{eq:sep21_1} is the same as
\begin{eqnarray*}
&& (\lambda_{h+1}+\cdots+\lambda_{k-1}+\partial^M)
(\chi_h\phi)^\mu\big(
\{{a_1}_{\lambda_1}\cdots{a_{k-1}}_{\lambda_{k-1}}a_k\}_c
\big) \\
&& =
(\lambda_{h+1}+\cdots+\lambda_{k-1}+\partial^M)
\{{a_{h+1}}_{\lambda_{h+1}}\cdots{a_{k-1}}_{\lambda_{k-1}}a_k\}_{\iota_xc}\,,
\end{eqnarray*}
namely $\iota_xc$ satisfies condition B2.
Similarly, for condition B3.,
let $\sigma$ be a permutation of the set $\{h+1,\dots,k\}$.
We have
\begin{eqnarray}\label{eq:sep21_2}
&& \{{a_{\sigma(h+1)}}_{\lambda_{\sigma(h+1)}}\cdots{a_{\sigma(k-1)}}_{\lambda_{\sigma(k-1)}}
a_{\sigma(k)}\}_{\iota_xc} \\
&& =
(\chi_h\phi)^\mu\big(
\{{a_1}_{\lambda_1}\cdots{a_h}_{\lambda_h}
{a_{\sigma(h+1)}}_{\lambda_{\sigma(h+1)}}\cdots
{a_{\sigma(k-1)}}_{\lambda_{\sigma(k-1)}}a_{\sigma(k)}\}_c
\big)\,.\nonumber
\end{eqnarray}
We then observe that, replacing in the above equation
$\lambda_k$ by $-\sum_{j=h+1}^{k-1}\lambda_j-\partial^M$,
$\partial^M$ acting from the left,
is the same as replacing it,
inside the argument of $(\chi_h\phi)^\mu$ in the RHS,
by $-\sum_{j=1}^{k-1}\lambda_j-\partial^M$.
For this we use Lemmas \ref{lem:ago16} and \ref{lem:ago20_2}(c).
After this substitution,
the RHS of \eqref{eq:sep21_2} becomes,
using the condition B3. for $c$,
$$
\text{sign}(\sigma)
(\chi_h\phi)^\mu\big(
\{{a_1}_{\lambda_1}\cdots{a_{k-1}}_{\lambda_{k-1}}a_k\}_c
\big)
\,=\,
\{{a_{h+1}}_{\lambda_{h+1}}\cdots{a_{k-1}}_{\lambda_{k-1}}a_k\}_{\iota_xc}
\,.
$$
\end{proof}

\begin{proposition}\label{prop:dic17_2}
The contraction operators on the superspace $C^\bullet$ commute, i.e.
for $x\in C_h$ and $y\in C_j$ we have
$$
\iota_x\iota_y\,=\,(-1)^{hj}\iota_y\iota_x\,.
$$
\end{proposition}
\begin{proof}
Let $a_1\otimes \cdots\otimes a_h\otimes\phi
\in A^{\otimes h}\otimes\Hom(\mb F[\lambda_1,\cdots,\lambda_{h-1}],M)$ 
be a representative for $x\in C_h$,
$b_1\otimes \cdots\otimes b_j\otimes\psi
\in A^{\otimes j}\otimes\Hom(\mb F[\mu_1,\cdots,\mu_{j-1}],M)$ 
be a representative for $y\in C_j$,
and let $c\in C^k$.
For $k>h+j$, the proof is similar to that of Proposition \ref{prop:dic17_1}.
Thus we only have to consider the case $k=h+j$.
Recalling \eqref{eq:ago21_3} and \eqref{eq:giulia}, we have
$$
\iota_y(\iota_xc)
\,=\,
\tint\psi^\mu\big(
(\chi_h\phi)^\mu\big(
\{{a_1}_{\lambda_1}\cdots
{a_{h-1}}_{\lambda_{h-1}}
{a_h}_{\lambda_h}
{b_1}_{\mu_1}\cdots
{b_{j-1}}_{\mu_{j-1}}b_j\}_c
\big)\big)\,.
$$
Applying the skew-symmetry condition B3. for $c$
and using the definition \eqref{eq:ago20_1} of $\chi_h$,
we get, after integration by parts, that the RHS is
$$
(-1)^{hj}\tint(\chi_j\psi)^\mu\big(
\phi^\mu\big(
\{
{b_1}_{\mu_1}\cdots
{b_{j-1}}_{\mu_{j-1}}{b_j}_{\mu_j}
{a_1}_{\lambda_1}\cdots
{a_{h-1}}_{\lambda_{h-1}}
{a_h}
\}_c
\big)\big)\,,
$$
which is the same as $(-1)^{hj}\iota_x(\iota_yc)$.
\end{proof}

For example, for $m\in C_0=\{m'\in M\,|\,\partial m'=0\}$,
we have
$\{{a_1}_{\lambda_1}\cdots{a_{k-1}}_{\lambda_{k-1}}a_k\}_{\iota_mc}
=m\{{a_1}_{\lambda_1}\cdots{a_{k-1}}_{\lambda_{k-1}}a_k\}_c$.
Recall also that $C_1=A\otimes M/\partial(A\otimes M)$.
The contraction operators associated to 1-chains are given by the following formulas:
if $c\in C^1=\Hom_{\mb F[\partial]}(A,M)$, then
\begin{equation}\label{eq:sat_8_1}
\iota_{a\otimes m}c \,=\, \int mc(a)\,,
\end{equation}
while if $c\in C^k$, with $k\geq2$, then
\begin{equation}\label{eq:sat_8}
\{{a_2}_{\lambda_2}\cdots{a_{k-1}}_{\lambda_{k-1}}a_k\}_{\iota_{a_1\otimes m}c}
\,=\,
{\{{a_1}_{\partial^M}{a_2}_{\lambda_2}\cdots{a_{k-1}}_{\lambda_{k-1}}a_k\}_c}_\to m\,,
\end{equation}
where the arrow in the RHS means, as usual, that $\partial^M$ should be moved to the right.

Also we have the following formulas for the Lie derivative $L_x=[d,\iota_x]$
by a 1-chain $x\in C_1$ acting on $C^0=M/\partial^MM$ 
and $C^1=\Hom_{\mb F[\partial]}(A,M)$:
\begin{eqnarray}\label{eq:sat_10}
L_{a\otimes m}\tint n &=& \tint (a_{\partial^M}n)_\to m\,,\\
(L_{a\otimes m}c)(b) 
&=&
\big(a_{\partial^M}c(b)\big)_\to m
+_\leftarrow\big((b_{-\partial^M}m)c(a)\big)
-c\big([a_{\partial^M}b]\big)_\to m\,,\nonumber
\end{eqnarray}
where the left arrow in the RHS means, as usual, 
that $\partial^M$ should be moved to the left.

The definitions of the contraction operators associated to elements 
of $\Gamma_\bullet$ and $C_\bullet$ are ``compatible".
This is stated in the following:
\begin{theorem}\label{th:ago21}
For $x\in C_h$ and $\gamma\in\Gamma^k$, with $k\geq h$,
we have
$$
\iota_x(\psi^k(\gamma))\,=\,
\psi^{k-h}(\iota_{\chi_h (x)}(\gamma))\,,
$$
where $\psi^k:\,\Gamma^k\hookrightarrow C^k$, denotes
the injective linear map defined in Theorem \ref{th:red},
and $\chi_h:\,C_h\to\Gamma_h$, denotes
the linear map defined in Proposition \ref{prop:ago20}.
In other words, there is a commutative diagram of linear maps:
\begin{equation}\label{eq:ago16_7}
\UseTips
\xymatrix{
C^k &   \ar[r]^{\iota_x} & &  C^{k-h} \\
\ar@{^{(}->}[u]^{\psi^k} \Gamma^k & \ar[r]^{\iota_\xi}  
& & \Gamma^{k-h} \ar@{^{(}->}[u]^{\psi^{k-h}} 
}\,,
\end{equation}
provided that $\xi\in\Gamma_h$ and $x\in C_h$ are related by $\xi=\chi_h(x)$.
\end{theorem}
\begin{proof}
Let $\tilde\gamma\in\tilde\Gamma^k$ be a representative of $\gamma\in\Gamma^k$,
and let $a_1\otimes\cdots\otimes a_{h}\otimes\phi
\in A^{\otimes h}\otimes\Hom(\mb F[\lambda_1,\dots,\lambda_{h-1}],M)$
be a representative of $x\in C_h$.
Recalling the definition \eqref{eq:5} of $\psi^k$ 
and the definition \eqref{eq:ago21_3} of $\iota_x$, we have
\begin{equation}\label{eq:ago21_6}
\{{a_{h+1}}_{\lambda_{h+1}}\cdots{a_{k-1}}_{\lambda_{k-1}}a_k\}_{\iota_x\psi^k(\tilde\gamma)}
\,=\,
(\chi_h\phi)^\mu\big(
\tilde\gamma_{
\lambda_1,\cdots,\lambda_{k-1},\lambda_k^\dagger}
(a_1,\cdots,a_{k})\big)\,,
\end{equation}
where, in the RHS, $\lambda_k^\dagger$ stands for $-\sum_{j=1}^{k-1}\lambda_j-\partial^M$,
with $\partial^M$ acting on the argument of $(\chi_h\phi)^\mu$.
By Lemmas \ref{lem:ago16} and \eqref{lem:ago20_2}(c), 
we can replace $\lambda_k^\dagger$ by $-\sum_{j=h+1}^{k-1}\lambda_j-\partial^M$,
where now $\partial^M$ is moved to the left of $(\chi_h\phi)^\mu$.
Hence, the RHS of \eqref{eq:ago21_6} is the same as
$$
(\chi_h\phi)^\mu\big(\tilde\gamma_{\lambda_1,\lambda_2,\cdots,\lambda_{k}}
(a_1,\cdots,a_{k})\big)
\big|_{\lambda_k\mapsto\lambda_k^\dagger}
=
\{{a_{h+1}}_{\lambda_{h+1}}\cdots{a_{k-1}}_{\lambda_{k-1}}a_k
\}_{\psi^{k-h}(\iota_{\chi_h(x)}(\tilde\gamma))}
\,,
$$
thus completing the proof of the theorem.
\end{proof}

\vspace{3ex}
\subsection{Lie conformal algeroids.}~~
\label{sec:3.4.5.a}
A Lie conformal algebroid is an analogue of a Lie algebroid.
\begin{definition}\label{def:algebroid}
A \emph{Lie conformal algebroid} is a pair $(A,M)$,
where $A$ is a Lie conformal algebra,
$M$ is a commutative associative differential algebra with derivative $\partial^M$,
such that $A$ is a left $M$-module and $M$ is a left $A$-module,
satisfying the following compatibility conditions ($a,b\in A,\,m,n\in M$):
\begin{enumerate}
\item[L1.] $\partial(ma)=(\partial^Mm)a+m(\partial a)$,
\item[L2.] $a_\lambda (mn)=(a_\lambda m)n+m(a_\lambda n)$,
\item[L3.] $[a_\lambda mb]=(a_\lambda m)b+m[a_\lambda b]$.
\end{enumerate}
\end{definition}
It follows from condition L3. and skew-symmetry \eqref{eq:0.2} of the $\lambda$-bracket,
that
\begin{enumerate}
\item[L3'.] $[ma_\lambda b]
=
\big(e^{\partial^M\partial_\lambda}m\big)[a_\lambda b]+(a_{\lambda+\partial} m)_\to b$,
\end{enumerate}
where the first term in the RHS is
$\sum_{i=0}^\infty \frac1{i!}
\big((\lambda+\partial^M)^i m\big)(a_{(i)} b)$,
and in the second term the arrow means, as usual, 
that $\partial$ should be moved to the right, acting on $b$.

We next give two examples analogous to those in the Lie algebroid case.
Let $M$ be, as above, a commutative associative differential algebra.
Recall from Section \ref{sec:2} that
a conformal endomorphism on $M$
is an $\mb F$-linear map $\varphi(=\varphi_\lambda):\,M\to\mb F[\lambda]\otimes M$
satisfying 
$\varphi_\lambda(\partial^Mm)=(\partial^M+\lambda)\varphi_\lambda(m)$.
The space $\cend(M)$ of conformal endomorphism
is then a Lie conformal algebra
with the $\mb F[\partial]$-module structure given by 
$(\partial\varphi)_\lambda=-\lambda\varphi_\lambda$,
and the $\lambda$-bracket given by
$$
[\varphi_\lambda\psi]_\mu
=
\varphi_\lambda\circ\psi_{\mu-\lambda}
-\psi_{\mu-\lambda}\circ\varphi_\lambda\,.
$$
\begin{example}\label{ex:dic12_1}
Let $\cder(M)$ be the subalgebra of the Lie conformal algebra $\cend(M)$
consisting of all conformal derivations on $M$,
namely of the the conformal endomorphisms satisfying the Leibniz rule:
$\varphi_\lambda(mn)=\varphi_\lambda(m)n+m\varphi_\lambda(n)$.
Then the pair $(\cder(M),M)$ is a Lie conformal algebroid,
where $M$ carries the tautological $\cder(M)$-module structure,
and $\cder(M)$ carries the following $M$-module structure:
\begin{equation}\label{eq:dic12_1}
(m\varphi)_\lambda
=
\big(e^{\partial^M\partial_\lambda}m)\varphi_\lambda\,.
\end{equation}
This is indeed an $M$-module, since $e^{x\partial^M}(mn)=(e^{x\partial^M}m)(e^{x\partial^M}n)$.
Furthermore, condition L1. holds thanks to the obvious identity
$e^{\partial^M\partial_\lambda}\lambda=(\lambda+\partial^M)e^{\partial^M\partial_\lambda}$.
Condition L2. holds by definition.
Finally, for condition L3. we have
\begin{eqnarray*}
[\varphi_\lambda m\psi]_\mu(n)
&=&
\varphi_\lambda\big((m\psi)_{\mu-\lambda}(n)\big) - (m\psi)_{\mu-\lambda}\big(\varphi_\lambda(n)\big) \\
&=&
\varphi_\lambda\big(\big(e^{\partial^M\partial_\mu}m\big)\psi_{\mu-\lambda}(n)\big) 
- \big(e^{\partial^M\partial_\mu}m\big)\psi_{\mu-\lambda}\big(\varphi_\lambda(n)\big) \\
&=&
\big(e^{(\lambda+\partial^M)\partial_\mu}
\varphi_\lambda(m)\big)\psi_{\mu-\lambda}(n) 
+
\big(e^{\partial^M\partial_\mu}m\big) 
\Big(
\varphi_\lambda\big(\psi_{\mu-\lambda}(n)\big) 
- \psi_{\mu-\lambda}\big(\varphi_\lambda(n)\big) 
\Big) \\
&=&
\big(e^{\partial^M\partial_\mu}
\varphi_\lambda(m)\big)\psi_{\mu}(n) 
+
\big(e^{\partial^M\partial_\mu}m\big) 
[\varphi_\lambda\psi]_\mu(n) 
=
\big(
\varphi_\lambda(m)\psi
+
m[\varphi_\lambda\psi]
\big)_\mu(n)\,.
\end{eqnarray*}
\end{example}
\begin{example}\label{ex:dic12_2}
Assume, as in Section \ref{sec:3.3}, that $A$ is a Lie conformal algebra
and $M$ is an $A$-module endowed with a commutative, associative product,
such that $\partial^M:\,M\to M$, and $a_\lambda:\,M\to\mb C[\lambda]\otimes M$,
for $a\in A$, satisfy the Leibniz rule.
The space $M\otimes A$ has a natural structure of $\mb F[\partial]$-module,
where $\partial$ acts as 
\begin{equation}\label{eq:dic12_4}
\tilde\partial(m\otimes a)=(\partial^Mm)\otimes a+m\otimes(\partial a)\,.
\end{equation}
Clearly, $M\otimes A$ is a left $M$-module via multiplication on the first factor.
We define a left $\lambda$-action of $M\otimes A$ on  $M$ by
\begin{equation}\label{eq:dic12_2}
(m\otimes a)_\lambda n
=
\big(e^{\partial^M\partial_\lambda}m\big) (a_\lambda n)\,,
\end{equation}
and a $\lambda$-bracket on $M\otimes A$ by
\begin{equation}\label{eq:dic12_3}
\big[(m\otimes a) _\lambda (n\otimes b)\big]
=
\big(\big(e^{\partial^M\partial_\lambda}m\big)n\big)\otimes [a_\lambda b]
+\big((m\otimes a)_\lambda n\big) \otimes b
-e^{\tilde\partial\partial_\lambda}\big((n\otimes b)_{-\lambda} m \otimes a\big)\,.
\end{equation}
We claim that \eqref{eq:dic12_4} and \eqref{eq:dic12_3} make $M\otimes A$ a Lie conformal algebra,
\eqref{eq:dic12_2} makes $M$ an $M\otimes A$-module,
and the pair $(M\otimes A,M)$ is a Lie conformal algebroid. 
This will be proved in Proposition \ref{prop:dic6}, using Lemmas \ref{lem:dic12_1} 
and \ref{lem:dic6_2}.
\end{example}

\begin{lemma}\label{lem:dic12_1}
\begin{enumerate}
\alphaparenlist
\item
The following $\lambda$-bracket defines a Lie conformal algebra structure
on the $\mb C[\partial]$-module $M\otimes A$:
\begin{equation}\label{eq:dic6_2}
\big[(m\otimes a)\,_\lambda\,(n\otimes b)\big]_0
\,=\,
\big(\big(e^{\partial^M\partial_\lambda}m\big)n\big)\otimes [a _\lambda b]\,.
\end{equation}
\item
For $x,y\in M\otimes A$ and $m\in M$, we have
\begin{equation}\label{eq:sat_1}
[mx_\lambda y]_0 = \big(e^{\partial^M\partial_\lambda}m\big)[x_\lambda y]_0\,\,,\,\,\,\,
[x_\lambda my]_0 = m[x_\lambda y]_0\,.
\end{equation}
\end{enumerate}
\end{lemma}
\begin{proof}
For the first sesquilinearity condition, we have
\begin{eqnarray*}
\big[\tilde\partial(m\otimes a)\,_\lambda\,(n\otimes b)\big]_0
&=&
\big(\big(e^{\partial^M\partial_\lambda}\partial^Mm\big) n\big) \otimes  [a _\lambda b]
-\big(\big(e^{\partial^M\partial_\lambda}m\big)n\big)\otimes \lambda[a _\lambda b] \\
&=&
-\lambda\big[(m\otimes a)\,_\lambda\,(n\otimes b)\big]_0\,.
\end{eqnarray*}
The second sesquilinearity condition and skew-symmetry can be proved in a similar way,
and they are left to the reader.
Let us check the Jacobi identity.
We have,
$$
\big[(m\otimes a)\,_\lambda\,\big[(n\otimes b)\,_\mu\,(p\otimes c)\big]_0\big]_0
=
\big(e^{\partial^M\partial_\lambda}m\big)\big(e^{\partial^M\partial_\mu}n\big)p
\otimes [a_\lambda[b_\mu c]]\,.
$$
Exchanging $a\otimes m$ with $b\otimes n$ and $\lambda$ with $\mu$, we get
$$
\big[(n\otimes b)\,_\mu\,\big[(m\otimes a)\,_\lambda\,(p\otimes c)\big]_0\big]_0
=
\big(e^{\partial^M\partial_\lambda}m\big)\big(e^{\partial^M\partial_\mu}n\big)p
\otimes [b_\mu[a_\lambda c]]\,.
$$
Furthermore, we have
$$
[[m\otimes a\,_\lambda\, n\otimes b]_0\,_\nu\,p\otimes c]_0
=
\big(e^{\partial^M\partial_\nu}\big(e^{\partial^M\partial_\lambda}m\big)n\big)p
\otimes [[a_\lambda b]_\nu c]\,.
$$
Putting $\nu=\lambda+\mu$, the RHS becomes
$$
\big(e^{\partial^M\partial_\lambda}m\big)\big(e^{\partial^M\partial_\mu}n\big)p
\otimes [[a_\lambda b]_{\lambda+\mu} c]\,.
$$
Hence, the Jacobi identity for the $\lambda$-bracket \eqref{eq:dic6_2}
follows immediately from the Jacobi identity for the $\lambda$-bracket on $A$.
This proves part (a). Part (b) is immediate.
\end{proof}
We define another $\lambda$-product on $M\otimes A$:
\begin{equation}\label{eq:dic12_5}
(m\otimes a)_\lambda(n\otimes b)=\big((m\otimes a)_\lambda n\big)\otimes b\,.
\end{equation}
Notice that the $\lambda$-bracket \eqref{eq:dic12_3} can be nicely written in terms
of the $\lambda$-bracket \eqref{eq:dic6_2} 
and the $\lambda$-product \eqref{eq:dic12_5}:
\begin{equation}\label{eq:dic12_6}
[x_\lambda y]= {[x_\lambda y]}_0+x_\lambda y-y_{-\lambda-\tilde{\partial}}x\,.
\end{equation}
\begin{lemma}\label{lem:dic6_2}
\begin{enumerate}
\alphaparenlist
\item
The $\lambda$-product \eqref{eq:dic12_5} satisfies both sesquilinearity conditions 
($x,y\in M\otimes A$):
\begin{equation}\label{eq:dic6_4}
(\tilde\partial x)_\lambda y = -\lambda\, x_\lambda y\,\,,\,\,\,\,
x_\lambda(\tilde\partial y) = (\lambda+\tilde\partial) (x_\lambda y)\,.
\end{equation}
\item 
For $x\in M\otimes A,\,m\in M$ and $y$ either in $M\otimes A$ or in $M$, we have
\begin{equation}\label{eq:sat_2}
(mx)_\lambda y = \big(e^{\partial^M\partial_\lambda}m\big) x_\lambda y\,\,,\,\,\,\,
x_\lambda(my) = (x_\lambda m)y+m(x_\lambda y)\,.
\end{equation}
\item
We have the following identity for $x,y,z\in M\otimes A$:
\begin{equation}\label{eq:dic6_5}
x_\lambda[y_\mu z]_0
=
[(x_\lambda y)_{\lambda+\mu}z]_0+[y_\mu(x_\lambda z)]_0\,.
\end{equation}
\item
We have the following identity for $x,y\in M\otimes A$ and $z$ either in $M$ or in $M\otimes A$:
\begin{equation}\label{eq:dic6_6}
x_\lambda(y_\mu z)
-y_\mu(x_\lambda z)
=
[x_\lambda y]_{\lambda+\mu}z\,.
\end{equation}
\end{enumerate}
\end{lemma}
\begin{proof}
We have
$$
(\tilde\partial(m\otimes a))_\lambda(n\otimes b)
=
\big(e^{\partial^M\partial_\lambda}(\partial^M-\lambda)m\big) (a_\lambda n)\otimes b\,.
$$
The first sesquilinearity condition follows from the obvious identity
$e^{\partial^M\partial_\lambda}(\partial^M-\lambda)=-\lambda e^{\partial^M\partial_\lambda}$.
The second sesquilinearity condition can be proved in a similar way.
This proves part (a).
Part (b) is immediate.
For part (c) and (d), let $x=a\otimes m,\,y=b\otimes n,\,z=c\otimes p\,\in A\otimes M$.
We have
\begin{equation}\label{eq:dic6_8}
x_\lambda[y_\mu z]_0
=
\big(e^{\partial^M\partial_\lambda}m\big)\big(a_\lambda \big(e^{\partial^M\partial_\mu}n\big) p\big)
\otimes [b_\mu c] \,,
\end{equation}
Similarly,
\begin{equation}\label{eq:dic6_9}
[(x_\lambda y)_{\nu} z]_0
=
\Big(e^{\partial^M\partial_\nu}
\big(e^{\partial^M\partial_\lambda}m\big)
(a_\lambda n)\Big) p
\otimes [b_{\nu}c] \,.
\end{equation}
Hence, if we put $\nu=\lambda+\mu$, the RHS becomes
\begin{equation}\label{eq:dic6_9_1}
\big(e^{\partial^M\partial_\lambda}m\big)
\big(e^{\partial^M\partial_\mu}
(a_\lambda n)\big) p
\otimes [b_{\lambda+\mu}c] 
=
\big(e^{\partial^M\partial_\lambda}m\big)
\big(a_\lambda \big(e^{\partial^M\partial_\mu}n\big)\big) p
\otimes [b_{\mu}c] \,,
\end{equation}
where we used the sesquilinearity of the $\lambda$-bracket on $A$.
Furthermore, we have
\begin{equation}\label{eq:dic6_10}
[y_\mu(x_\lambda z)]_0
=
\big(e^{\partial^M\partial_\mu}n\big)
\big(e^{\partial^M\partial_\lambda}m\big)(a_\lambda p)
\otimes
[b_{\mu}c] \,.
\end{equation}
Combining equations \eqref{eq:dic6_8}, \eqref{eq:dic6_9_1} and \eqref{eq:dic6_10},
we immediately get \eqref{eq:dic6_5}, thanks to the assumption that the $\lambda$-action
of $A$ on $M$ is a derivation of the commutative associative product on $M$.
We are left to prove part (d).
We have
\begin{eqnarray}\label{eq:dic6_piove_1}
&& x_\lambda(y_\mu p)
=
\big(e^{\partial^M\partial_\lambda}m\big)
a_\lambda \Big(\big(e^{\partial^M\partial_\mu}n\big)(b_\mu p)\Big) 
\nonumber\\
&&\,\,\, =
\big(e^{\partial^M\partial_\lambda}m\big)\big(e^{\partial^M\partial_\mu}n\big)
a_\lambda (b_\mu p)
+
\big(e^{\partial^M\partial_\lambda}m\big)\big(e^{\partial^M\partial_\mu}(a_\lambda n)\big)
(b_{\lambda+\mu} p)\,.
\end{eqnarray}
For the second equality, we used the Leibniz rule and the sesquilinearity condition
for the $\lambda$-action of $A$ on $M$.
Exchanging $x$ with $y$ and $\lambda$ with $\mu$, we have
\begin{equation}\label{eq:dic6_piove_2}
y_\mu(x_\lambda p)
=
\big(e^{\partial^M\partial_\lambda}m\big)\big(e^{\partial^M\partial_\mu}n\big)
b_\mu (a_\lambda p)
+
\big(e^{\partial^M\partial_\mu}n\big)\big(e^{\partial^M\partial_\lambda}(b_\mu m)\big)
(a_{\lambda+\mu} p)\,.
\end{equation}
By similar computations, we get
\begin{equation}\label{eq:dic6_piove_3}
(x_\lambda y)_{\lambda+\mu} z
=
\big(e^{\partial^M\partial_\lambda}m\big)\big(e^{\partial^M\partial_\mu}(a_\lambda n)\big)
(b_{\lambda+\mu} p)\,,
\end{equation}
and
\begin{equation}\label{eq:dic6_piove_4}
(y_{-\lambda-\partial} x)_{\lambda+\mu} p
=
\big(e^{\partial^M\partial_\mu}n\big)\big(e^{\partial^M\partial_\lambda}(b_\mu m)\big)
(a_{\lambda+\mu} p)\,.
\end{equation}
Finally, it follows by a straightforward computation that
\begin{equation}\label{eq:dic6_piove_5}
{[x_\lambda y]_0}_{\lambda+\mu} z
=
\big(e^{\partial^M\partial_\lambda}m\big)\big(e^{\partial^M\partial_\mu}n\big)
[a_\lambda b]_{\lambda+\mu} p\,.
\end{equation}
Equation \eqref{eq:dic6_6} is obtained combining
equations \eqref{eq:dic6_piove_1}, \eqref{eq:dic6_piove_2}, \eqref{eq:dic6_piove_3}, 
\eqref{eq:dic6_piove_4} and \eqref{eq:dic6_piove_5}.
\end{proof}
\begin{proposition}\label{prop:dic6}
\begin{enumerate}
\alphaparenlist
\item
The $\lambda$-bracket \eqref{eq:dic12_3} defines a Lie conformal algebra structure
on the $\mb F[\partial]$-module $M\otimes A$.
\item
The $\lambda$-action \eqref{eq:dic12_2} defines a structure of a $M\otimes A$-module on $M$.
\item The pair $(M\otimes A,M)$ is a Lie conformal algebroid.
\item We have a homomorphism of Lie conformal algebroids 
$(M\otimes A,M)\to(\cder(M),M)$,
given by the identity map on $M$ and the following Lie conformal algebra homomorphism
from $M\otimes A$ to $\cder(M)$:
$$
m\otimes a\mapsto \big(e^{\partial^M\partial_\lambda}m\big)a_\lambda\,.
$$
\end{enumerate}
\end{proposition}
\begin{proof}
It immediately follows from Lemma \ref{lem:dic12_1} and Lemma \ref{lem:dic6_2}(a)
that the $\lambda$-bracket \eqref{eq:dic12_6}
satisfies sesquilinearity and skew-symmetry.
Furthermore, the Jacobi identity for the $\lambda$-bracket \eqref{eq:dic12_3}
follows from Lemma \ref{lem:dic12_1} and equations \eqref{eq:dic6_5} and \eqref{eq:dic6_6}.
This proves part (a).
Part (b) is Lemma \ref{eq:dic6_4}(c), in the case $z\in M$.
For part (c) we need to check conditions L1., L2. and L3.
The first two conditions are immediate. The last one follows from equations 
\eqref{eq:sat_1} and \eqref{eq:sat_2}.
Finally, part (d) is straightforward and is left to the reader.
\end{proof}

\vspace{3ex}
\subsection{The Lie algera structure on $\Pi C_1$
and the $\Pi C_1$-structure on the complex $(C^\bullet,d)$.}~~
\label{sec:3.4.5.b}
Recall that the space of 1-chains of the complex $(C^\bullet,d)$ is 
$C_1=(A\otimes M)/\partial(A\otimes M)$ with odd parity.
We want to define a Lie algebra structure on $\Pi C_1$, where, as usual, 
$\Pi$ denotes parity reversing, making $C^\bullet$ into a $\Pi C_1$-complex.
By Proposition \ref{prop:dic6}(a), we have a Lie conformal algebra structure on $M\otimes A$.
Hence, if we identify $M\otimes A$ with $A\otimes M$ by exchanging  the two factors,
we get a structure of a Lie algebra on the quotient space $(A\otimes M)/\partial(A\otimes M)$,
induced by the $\lambda$-bracket at $\lambda=0$ \cite{K}. 

Explicitly, we get the following well-defined Lie algebra bracket on
$\Pi C_1=(A\otimes M)/\partial(A\otimes M)$:
\begin{equation}\label{eq:sat_3}
[a\otimes m,b\otimes n]
\,=\,
[a_{\partial^M_1}b]_\to \otimes mn
+b\otimes \big(a_{\partial^M}n\big)_\to m
-a\otimes \big(b_{\partial^M}m\big)_\to n\,,
\end{equation}
where in the RHS, as usual, the right arrow means that $\partial^M$ should be moved to the right,
and in the first summand $\partial^M_1$ denotes $\partial^M$ acting only on
the first factor $m$.

Recall from Section \ref{sec:3.2.5} that
$\tilde\Gamma_1
= (A\otimes M[[x]])\big/(\partial\otimes1+1\otimes\partial_x)(A\otimes M[[x]])$,
and $\Gamma_1=\big\{\xi\in\tilde\Gamma_1\,|\,\partial\xi=0\big\}$,
where the action of $\partial$ on $\tilde\Gamma_1$ is given by \eqref{eq:nov22_1}.
Under this identification,
the map $\chi_1:\,C_1\to\Gamma_1$ defined by \eqref{eq:ago20_1} and \eqref{eq:sat_4}
is given by
\begin{equation}\label{eq:sat_5}
\chi_1(a\otimes m)\,=\,a\otimes e^{x\partial^M}m\,.
\end{equation}
\begin{proposition}\label{prop:sat}
The map $\chi_1:\,C_1\to\Gamma_1$
is a Lie algebra homomorphism,
which factors through a Lie algebra  isomorphism $\chi_1:\,\bar C_1\to\Gamma_1$,
provided that $A$ decomposes as in \eqref{eq:anniv_1}.
\end{proposition}
\begin{proof}
We have, by \eqref{eq:sat_3} and \eqref{eq:sat_5} that
\begin{eqnarray}\label{eq:sat_6}
&& \chi_1\big([a\otimes m,b\otimes n]\big)
\\
&& =
[a_{\partial^M_1}b]_\to\otimes\big(e^{x\partial^M}m\big)\big(e^{x\partial^M}n\big)
+b\otimes e^{x\partial^M} \big((a_{\partial^M}n)_\to m\big)
-a\otimes e^{x\partial^M} \big((b_{\partial^M}m)_\to n\big)\,.\nonumber
\end{eqnarray}
Recalling formula \eqref{eq:nov23_2} for the Lie bracket on $\tilde\Gamma_1$,
we have
\begin{eqnarray}\label{eq:sat_7}
&& [\chi_1(a\otimes m),\chi_1(b\otimes n)]\big)
\\
&& =
[a_{\partial_{x_1}}b]\otimes\big(e^{x_1\partial^M}m\big)\big(e^{x\partial^M}n\big)\,\Big|_{x_1=x}
+b\otimes \big\langle m(x_1),a_{\lambda_1}n(x)\big\rangle
-a\otimes \big\langle n(x_1),b_{\lambda_1}m(x)\big\rangle\,.\nonumber
\end{eqnarray}
Clearly, the first term in the RHS of \eqref{eq:sat_6}
is the same as the first term in the RHS of \eqref{eq:sat_7}.
Recalling the definition \eqref{eq:nov22_3} of the pairing $\langle\,,\,\rangle$,
and using the sesquilinearity of the $\lambda$-action of $A$ on $M$, 
we have that the second term in the RHS of \eqref{eq:sat_6}
is the same as the second term in the RHS of \eqref{eq:sat_7},
and similarly for the third terms.
The last statement follows from Proposition \ref{th:ago20}.
\end{proof}
\begin{proposition}\label{prop:sat_2}
The cohomoloy complex $(C^\bullet,d)$
has a $\Pi C_1$-structure $\varphi:\,\hat{\Pi C_1}\to\End C^\bullet$,
given by $\varphi(\partial_\eta)=d,\,\varphi(\eta x)=\iota_x,\,\varphi(x)=L_x=[d,\iota_x]$.
Moreover, $(\bar C^\bullet,d)$ is a $\Pi C_1$-subcomplex.
\end{proposition}
\begin{proof}
Due to Remark \ref{rem:vsep} and Proposition \ref{prop:dic17_2}, 
we only need to check that, for $x,y\in\Pi C_1$, we have
\begin{equation}\label{eq:sat_9}
[L_x,\iota_y]\,=\,\iota_{[x,y]}\,.
\end{equation}
This follows from a long but straightforward computation, 
using the explicit formulas \eqref{eq:d>} and \eqref{eq:sat_8}
for the differential and the contraction operators. 
It is left to the reader.

Notice though that, in the special case when $A$ decomposes as in \eqref{eq:anniv_1},
equation \eqref{eq:sat_9} is a corollary of Proposition \ref{prop:vsep30_1},
Theorem \ref{th:red} and Theorem \ref{th:ago21} for $h=1$.
Indeed, due to these results, it suffices to check that both sides of \eqref{eq:sat_9}
coincide when acting on $C^1=\Hom_{\mb F[\partial]}(A,M)$.
In the latter case,
using equations \eqref{eq:d0}, \eqref{eq:july29_1}, 
\eqref{eq:sat_8_1}, \eqref{eq:sat_8}, \eqref{eq:sat_10} and \eqref{eq:sat_3}, 
we have, for $c\in C^1$,
\begin{eqnarray*}
L_{a\otimes m}(\iota_{b\otimes n}c)
&=&
\tint c(b)\big(a_{\partial^M}n\big)_\to m
+\tint n\big(a_{\partial^M}c(b)\big)_\to m\,,\\
\iota_{b\otimes n}(L_{a\otimes m}c)
&=&
\tint n\big(a_{\partial^M}c(b)\big)_\to m
+\tint c(a)\big(b_{\partial^M}m\big)_\to n
-\tint nc\big([a_{\partial^M}b]\big)_\to m\,, \\
\iota_{[a\otimes m,b\otimes n]}c
&=&
\tint nc\big([a_{\partial^M}b]\big)_\to m
+\tint c(b)\big(a_{\partial^M}n\big)_\to m
-\tint c(a)\big(b_{\partial^M}m\big)_\to n\,.
\end{eqnarray*}
It follows that \eqref{eq:sat_9} holds when applied to elements of $C^1$.
\end{proof}
The above results imply the following
\begin{theorem}\label{th:sat}
The maps $\psi^\bullet:\,\Gamma^\bullet\to\bar C^\bullet\subset C^\bullet$
and $\chi_1:\,C_1\to\Gamma_1$ define a homomorphism of $\mf g$-complexes.
Provided that $A$ decomposes as in \eqref{eq:anniv_1},
we obtain an isomorphism of $\Pi C_1\simeq\Pi\Gamma_1$-complexes 
$\psi^\bullet:\,\Gamma^\bullet\stackrel{\sim}{\rightarrow}\bar C^\bullet$.
\end{theorem}
\begin{proof}
It follows from Theorem \ref{th:red},
Proposition \ref{th:ago20},
Theorem \ref{th:ago21}
and Proposition \ref{prop:sat}.
\end{proof}

\vspace{3ex}
\subsection{Pairings between 1-chains and 1-cochains.}~~
\label{sec:3.pair}
Recall that $\tilde\Gamma^0=M$.
Hence, the contraction operators of 1-chains, restricted to the space of 1-cochains,
define a natural pairing
$\tilde\Gamma_1\times\tilde\Gamma^1\to M$,
which, to $\xi\in\tilde\Gamma_1$ and $\tilde\gamma\in\tilde\Gamma^1$, associates
\begin{equation}\label{eq:ago21_9}
\iota_\xi\tilde\gamma\,=\,\phi^\mu(\tilde\gamma_\lambda(a))\,\in\, M\,,
\end{equation}
where $a\otimes\phi\in A\otimes\Hom(\mb F[\lambda],M)$
is a representative of $\xi$.

When we consider the reduced spaces, we have $\Gamma^0=M/\partial M$,
and the above map induces a natural pairing
$\Gamma_1\times\Gamma^1\to M/\partial M$,
which, to $\xi\in\tilde\Gamma_1$ and $\gamma\in\Gamma^1$, associates
\begin{equation}\label{eq:ago21_10}
\iota_\xi\gamma\,=\,\tint\phi^\mu(\tilde\gamma_\lambda(a))\,\in\, M/\partial M\,,
\end{equation}
where again $a\otimes\phi\in A\otimes\Hom(\mb F[\lambda],M)$
is a representative of $\xi$,
and $\tilde\gamma\in\tilde\Gamma^1$ is a representative of $\gamma$.

A similar pairing can be defined for 1-chains in $C_1$ and 1-cochains in $C^1$.
Recall that $C^0=M/\partial M$,
$C^1$ is the space of $\mb F[\partial]$-module homomorphisms $c:\,A\to M$,
and $C_1=A\otimes M/\partial(A\otimes M)$.
The corresponding pairing
$C_1\times C^1\to M/\partial M$, 
is obtained as follows.
To $x\in C_1$ and $c\in C^1$, we associate,
recalling \eqref{eq:giulia},
\begin{equation}\label{eq:ago21_11}
\iota_x(c)=\tint m\cdot c(a)\,\in\, M/\partial M\,,
\end{equation}
where $a\otimes m\in A\otimes M$ is a representative of $x$.

Recalling Theorems \ref{th:red} and \ref{th:ago21},
the above pairings \eqref{eq:ago21_10} and \eqref{eq:ago21_11}
are compatible in the sense that
$\iota_x(c)\,=\,\iota_\xi(\gamma)$,
provided that $\gamma\in\Gamma^1$ and $c\in C^1$ are related by $c=\psi^1(\gamma)$,
and $\xi\in\Gamma_1$ and $x\in C_1$ are related by $\xi=\chi_1(x)$.

\vspace{3ex}
\subsection{Contraction by a $1$-chain as an odd derivation of $\tilde\Gamma^\bullet$.}~~
\label{sec:3.6}
Recall that, if the $A$-module $M$ has a commutative associative product,
and $\partial^M$ and $a_\lambda^M$ are even derivations of it,
then the basic cohomology complex $\tilde\Gamma^\bullet$
is a $\mb Z$-graded commutative associative superalgebra
with respect to the exterior product \eqref{eq:sfor_2},
and the differential $\delta$ is an odd derivation of degree +1.
\begin{proposition}\label{prop:ago21}
The contraction operator $\iota_\xi$,
associated to a 1-chain $\xi\in\tilde\Gamma_1$, is an odd derivation 
of the superalgebra $\tilde\Gamma^\bullet$ of degree -1.
\end{proposition}
\begin{proof}
Let $a_1\otimes\phi$, with $a_1\in A$ and $\phi\in\Hom(\mb F[\lambda_1],M)$,
be a representative of $\xi\in\tilde\Gamma_1$.
By the definition \eqref{eq:sfor_2} of the exterior product, we have
\begin{eqnarray}\label{eq:ago21_7}
&\displaystyle{
(\iota_\xi(\tilde\alpha\wedge\tilde\beta))_{\lambda_2,\cdots,\lambda_{h+k}}(a_2,\cdots,a_{h+k}) 
=
\sum_{\sigma\in S_{h+k}} \frac{\text{sign}(\sigma)}{h!k!}
\phi^\mu\Big(
\tilde\alpha_{\lambda_{\sigma(1)},\cdots,\lambda_{\sigma(h)}}
(a_{\sigma(1)},\cdots,a_{\sigma(h)}) \times
}\nonumber\\
&\displaystyle{
\tilde\beta_{\lambda_{\sigma(h+1)},\cdots,\lambda_{\sigma(h+k)}}
(a_{\sigma(h+1)},\cdots,a_{\sigma(h+k)})
\Big)
\,.
}
\end{eqnarray}
By the skew-symmetry condition A2. for $\tilde\alpha$ and $\tilde\beta$, we can rewrite 
the RHS of \eqref{eq:ago21_7} as
\begin{eqnarray}\label{eq:ago21_8}
&& 
\sum_{i=1}^{h}\sum_{\sigma\,|\,\sigma(i)=1} \frac{\text{sign}(\sigma)}{h!k!}(-1)^{i+1}
\phi^\mu\Big(
\tilde\alpha_{\lambda_1,\lambda_{\sigma(1)},\stackrel{i}{\check{\cdots}},\lambda_{\sigma(h)}}
(a_1,a_{\sigma(1)},\stackrel{i}{\check{\cdots}},a_{\sigma(h)})
\Big) \times\nonumber\\
&&\,\,\,\,\,\,\,\,\,\,\,\,\,\,\,\,\,\,\,\,\,\,\,\,\,\,\,\,\,\,\,\,\,\,\,\,\times
\tilde\beta_{\lambda_{\sigma(h+1)},\cdots,\lambda_{\sigma(h+k)}}
(a_{\sigma(h+1)},\cdots,a_{\sigma(h+k)}) \\
&& +
\sum_{i=h+1}^{h+k}\sum_{\sigma\,|\,\sigma(i)=1} \frac{\text{sign}(\sigma)}{h!k!}(-i)^{i-h+1}
\tilde\alpha_{\lambda_{\sigma(1)},\cdots,\lambda_{\sigma(h)}}
(a_{\sigma(1)},\cdots,a_{\sigma(h)})
\times\nonumber\\
&&\,\,\,\,\,\,\,\,\,\,\,\,\,\,\,\,\,\,\,\,\,\,\,\,\,\,\,\,\,\,\,\,\,\,\,\,\times
\phi^\mu\Big(
\tilde\beta_{\lambda_1,\lambda_{\sigma(h+1)},\stackrel{i}{\check{\cdots}},\lambda_{\sigma(h+k)}}
(a_1,a_{\sigma(h+1)},\stackrel{i}{\check{\cdots}},a_{\sigma(h+k)})
\Big)
\,.\nonumber
\end{eqnarray}
By Lemma \ref{lem:ago15},
the set of all permutations $\sigma\in S_{h+k}$ such that $\sigma(i)=1$,
is naturally in bijection with the set of all permutations $\tau$ of $\{2,\dots,h+k\}$,
and the correspondence between the signs is $\text{sign}(\tau)=(-1)^{i+1}\text{sign}(\sigma)$.
Hence, \eqref{eq:ago21_8} can be rewritten as
\begin{eqnarray*}
&\displaystyle{
\sum_{\tau} \frac{\text{sign}(\tau)}{h!k!}
\Big(
h(\iota_\xi\tilde\alpha)_{\lambda_{\tau(2)},\cdots,\lambda_{\tau(h)}}(a_{\tau(2)},\cdots,a_{\tau(h)})
\tilde\beta_{\lambda_{\tau(h+1)},\cdots,\lambda_{\tau(h+k)}}
(a_{\tau(h+1)},\cdots,a_{\tau(h+k)})
}\\
&\displaystyle{
+k(-1)^{h}
\tilde\alpha_{\lambda_{\tau(2)},\cdots,\lambda_{\tau(h+1)}}
(a_{\tau(2)},\cdots,a_{\tau(h+1)})
(\iota_\xi\tilde\beta)_{\lambda_{\tau(h+2)},\cdots,\lambda_{\tau(h+k)}}
(a_{\tau(h+2)},\cdots,a_{\tau(h+k)})
\Big)
}\\
&\displaystyle{
=(\iota_\xi(\tilde\alpha)\wedge\tilde\beta)_{\lambda_2,\cdots,\lambda_{h+k}}(a_2,\cdots,a_{h+k})
+(-1)^h
(\tilde\alpha\wedge\iota_\xi(\tilde\beta))_{\lambda_2,\cdots,\lambda_{h+k}}(a_2,\cdots,a_{h+k})
\,.}
\end{eqnarray*}
\end{proof}

\begin{remark}\label{rem:reallylast}
One can show that the $\mf g$-structure of all our complexes
$\tilde\Gamma^\bullet,\,\Gamma^\bullet$ and $C^\bullet$
can be extended to a structure of a calculus algebra, 
as defined in \cite{DTT}.
Namely,
one can extend the Lie algebra bracket from the space of 1-chains 
to the whole space of chains (with reverse parity),
and define there a commutative superalgebra structure,
which extends our $\mf g$-structure
and satisfies all the properties of a calculus algebra.
\end{remark}

\section{The complex of variational calculus as a Lie conformal algebra cohomology complex}
\label{sec:4}

\subsection{Algebras of differentiable functions.}~~
\label{sec:5.1}
An \emph{algebra of differentiable functions} $\mc V$
in the variables $u_i$, indexed by a finite set $I=\{1,\dots,\ell\}$,
is, by definition, a differential algebra
(i.e. a unital commutative associative algebra with a derivation $\partial$),
endowed with commuting derivations
$\frac{\partial}{\partial u_i^{(n)}}\,:\,\,\mc V\to\mc V$, for all $i\in I$ and $n\in\mb Z_+$,
such that, given $f\in\mc V$,
$\frac{\partial}{\partial u_i^{(n)}}f=0$ for all but finitely many $i\in I$ and $n\in\mb Z_+$,
and the following commutation rules with $\partial$ hold:
\begin{equation}\label{eq:comm_frac}
\Big[\frac{\partial}{\partial u_i^{(n)}} , \partial\Big] = \frac{\partial}{\partial u_i^{(n-1)}}\,,
\end{equation}
where the RHS is considered to be zero if $n=0$.
As in the previous sections, 
we denote by $f\mapsto\tint f$ the canonical quotient map $\mc V\to\mc V/\partial\mc V$.

Denote by $\mc C\subset\mc V$ the subspace of constant functions, i.e. 
\begin{equation}\label{eq:4.2}
\mc C=\big\{f\in\mc V\,\big|\,\frac{\partial f}{\partial u_i^{(n)}}=0\,\,\forall i\in I,\,n\in\mb Z_+\big\}\,.
\end{equation}
It follows from \eqref{eq:comm_frac}
by downward induction that
\begin{equation}\label{eq:dic20_2}
\Ker(\partial)\subset\mc C\,.
\end{equation}
Also, clearly, $\partial\mc C\subset\mc C$.

Typical examples of algebras of differentiable functions are: 
the ring of polynomials 
\begin{equation}\label{eq:dic20_1}
R_\ell
\,=\,
\mb F[u_i^{(n)}\,|\,i\in I,n\in\mb Z_+]\,,
\end{equation}
where $\partial(u_i^{(n)})=u_i^{(n+1)}$,
any localization of it by some multiplicative subset $S\subset R$,
such as the whole field of fractions $Q=\mb F(u_i^{(n)}\,|\,i\in I,n\in\mb Z_+)$,
or any algebraic extension of the algebra $R$ or of the field $Q$ obtained by adding
a solution of certain polynomial equation.
In all these examples the action of $\partial:\, \mc V\to\mc V$
is given by
$\displaystyle{
\partial=\sum_{i\in I,n\in\mb Z_+} u_i^{(n+1)}\frac{\partial}{\partial u_i^{(n)}}
}$.
Another example of an algebra of differentiable functions
is the ring $R_\ell[x]=\mb F[x,u_i^{(n)}\,|\,i\in I,n\in\mb Z_+]$,
where 
$\displaystyle{
\partial=\frac{\partial}{\partial x}
+\sum_{i\in I,n\in\mb Z_+} u_i^{(n+1)}\frac{\partial}{\partial u_i^{(n)}}
}$.

The \emph{variational derivative}
$\frac\delta{\delta u}:\,\mc V\to\mc V^{\oplus \ell}$
is defined by 
\begin{equation}\label{eq:varder}
\frac{\delta f}{\delta u_i}\,:=\,\sum_{n\in\mb Z_+}(-\partial)^n\frac{\partial f}{\partial u_i^{(n)}}\,.
\end{equation}
It follows immediately from \eqref{eq:comm_frac} that
\begin{equation}\label{eq:mar11_2}
\frac{\delta}{\delta u_i}(\partial f) = 0\,,
\end{equation}
for every $i\in I$ and $f\in\mc V$,
namely, $\partial\mc V\subset\Ker \frac{\delta}{\delta u}$.

A \emph{vector field} is, by definition, a derivation of $\mc V$ of the form
\begin{equation}\label{2006_X}
X=\sum_{i\in I,n\in\mb{Z}_+} P_{i,n} \frac{\partial}{\partial u_i^{(n)}}\,\,,  \quad P_{i,n} \in \mc V\,.
\end{equation}
We let $\mf g$ be the Lie algebra of all vector fields.
The subalgebra of \emph{evolutionary vector fields} is $\mf g^\partial\subset\mf g$,
consisting of the vector fields commuting with $\partial$.
By \eqref{eq:comm_frac}, a vector field $X$ is evolutionary 
if and only if it has the form
\begin{equation}\label{2006_X2}
X_P=\sum_{i\in I,n\in\mb{Z}_+} (\partial^n P_i) \frac{\partial}{\partial u_i^{(n)}}
\,\,,\,\,\,\,
\text{ where } P=(P_i)_{i\in I}\in\mc V^\ell\,.
\end{equation}

\vspace{3ex}
\subsection{Normal algebras of differentiable functions.}~~
\label{sec:very_final}
Let $\mc V$ be an algebra of differentiable functions in the variables $u_i,\,i\in I=\{1,\dots,\ell\}$.
For $i\in I$ and $n\in\mb Z_+$ we let 
\begin{equation}\label{eq:july21_1}
\mc V_{n,i}\,:=\,\Big\{ f\in\mc V\,\Big|\,
\frac{\partial f}{\partial u_j^{(m)}}=0\,\,
\text{ if } (m,j)>(n,i) \text{ in lexicographic order }\Big\}\,.
\end{equation}
We also let $\mc V_{n,0}=\mc V_{n-1,\ell}$.

A natural assumption on $\mc V$ is to contain elements
$u_i^{(n)}$, for $i\in I,n\in\mb Z_+$,
such that 
\begin{equation}\label{eq:dic20_4}
\frac{\partial u_i^{(n)}}{\partial u_j^{(m)}}
\,=\,
\delta_{ij}\delta_{mn}\,.
\end{equation}
Clearly, such elements are uniquely defined up to adding constant functions.
Moreover, choosing these constants appropriately, we can assume
that $\partial u_i^{(n)}=u_i^{(n+1)}$.
Thus, under this assumption $\mc V$ is an algebra of differentiable functions 
extension of the algebra $R_\ell$ in \eqref{eq:dic20_1}.
\begin{lemma}\label{lem:fine}
Let $\mc V$ be an algebra of differentiable functions extension of the algebra $R_\ell$.
Then:
\begin{enumerate}
\alphaparenlist
\item We have $\partial=\partial_R+\partial^\prime$, where
\begin{equation}\label{eq:dic20_3}
\partial_R
\,=\,
\sum_{i\in I,n\in\mb Z_+}u_i^{(n+1)}\frac{\partial}{\partial u_i^{(n)}}\,,
\end{equation}
and $\partial^\prime$ is a derivation of $\mc V$ which 
commutes with all $\frac{\partial}{\partial u_i^{(n)}}$
and which vanishes on $R_\ell\subset\mc V$.
In particular, $\partial^\prime\mc V_{n,i}\subset\mc V_{n,i}$.
\item If $f\in\mc V_{n,i}\backslash\mc V_{n,i-1}$,
then $\partial f\in\mc V_{n+1,i}\backslash\mc V_{n+1,i-1}$, and it has the form
\begin{equation}\label{dec27_1_new}
\partial f\,=\, \sum_{j\leq i} h_ju_j^{(n+1)}+r\,,
\end{equation}
where $h_j\in\mc V_{n,i}$ for all $j\leq i,\, r\in\mc V_{n,i}$, 
and $h_i\neq0$.
\item For $f\in\mc V$, $\tint fg=0$ for every $g\in\mc V$ if and only if $f=0$.
\end{enumerate}
\end{lemma}
\begin{proof}
Part (a) is clear.
By part (a), we have that
$\partial f$ is as in \eqref{dec27_1_new},
where
$h_j=\frac{\partial f}{\partial u_j^{(n)}}\in\mc V_{n,i}$,
and $r=
\sum_{j\in I,m\leq n}u_j^{(m)}\frac{\partial f}{\partial u_j^{(m-1)}}+\partial^\prime f\in\mc V_{n,i}$.
We are left to prove part (c).
Suppose $f\neq0$ is such that $\tint fg=0$ for every $g\in\mc V$.
By taking $g=1$, we have that $f\in\partial\mc V$.
Hence $f$ has the form \eqref{dec27_1_new}
for some $i\in I$ and $n\in\mb Z_+$.
But then $u_i^{(n+1)}f$ does not have this form, so that $\tint u_i^{(n+1)}f\neq0$.
\end{proof}
\begin{definition}\label{def:normal}
The algebra of differentiable functions $\mc V$
is called \emph{normal} if 
$\frac\partial{\partial u_i^{(n)}}\big(\mc V_{n,i}\big)=\mc V_{n,i}$ 
for all $i\in I,n\in\mb Z_+$.
Given $f\in\mc V_{n,i}$, we denote by $\tint du_i^{(n)}f\in\mc V_{n,i}$
a preimage of $f$ under the map $\frac{\partial}{\partial u_i^{(n)}}$.
This integral is defined up to adding elements from $\mc V_{n,i-1}$.
\end{definition}
\begin{proposition}\label{prop:dic20}
Any normal algebra of differentiable functions $\mc V$
is an extension of $R_\ell$.
\end{proposition}
\begin{proof}
As pointed out above, we need to find elements $u_i^{(n)}\in\mc V$, for $i\in I,n\in\mb Z_+$,
such that \eqref{eq:dic20_4} holds.
By the normality assumption, 
there exists $v_i^n\in\mc V_{n,i}$ such that $\frac{\partial v_i^n}{\partial u_i^{(n)}}=1$.
Note that $ \frac{\partial}{\partial u_{i}^{(n)}} \frac{\partial v_i^n}{\partial u_{i-1}^{(n)}}=
\frac{\partial 1}{\partial u_{i-1}^{(n)}}=0$,
hence $\frac{\partial v_i^n}{\partial u_{i-1}^{(n)}}\in\mc V_{n,i-1}$.
If we then replace $v_i^n$ by
$w_i^n=v_i^n-\tint du_{i-1}^n\frac{\partial v_i^n}{\partial u_{i-1}^{(n)}}$,
we have that 
$\frac{\partial w_i^n}{\partial u_i^{(n)}}=1$
and $\frac{\partial w_i^n}{\partial u_{i-1}^{(n)}}=0$.
Proceeding by downward induction, we obtained the desired element $u_i^{(n)}$.
\end{proof}

Clearly, the algebra $R_\ell$ is normal.
Moreover, any extension $\mc V$ of $R_\ell$ can be further extended to a normal algebra,
by adding missing integrals.
For example, the localization of $R_1=\mb F[u^{(n)}\,|\,n\in\mb Z_+]$ by $u$
is not a normal algebra, since it doesn't contain $\tint\frac{du}{u}$.
Note that any differential algebra $(A,\partial)$
can be viewed as a trivial algebra of differentiable functions with $\frac{\partial}{\partial u_i^{(n)}}=0$.
Such an algebra does not contain $R_\ell$, hence it is not normal.

\vspace{3ex}
\subsection{The complex of variational calculus.}~~
\label{sec:5.2}
Let $\mc V$ be an algebra of differentiable functions.
The \emph{basic de Rham complex} $\tilde\Omega^\bullet=\tilde\Omega^\bullet(\mc V)$
is defined as the
free commutative superalgebra over $\mc V$
with odd generators $\delta u_i^{(n)},\,i\in I,n\in\mb{Z}_+$.
In other words $\tilde\Omega^\bullet$ consists of finite sums of the form
\begin{equation}\label{eq:apr24_1}
\tilde\omega=\sum_{i_r\in I, m_r\in\mb{Z}_+}
f^{m_1\cdots m_{k}}_{i_1\cdots i_{k}}\, \delta u_{i_1}^{(m_1)}\wedge\cdots\wedge
\delta u_{i_{k}}^{(m_{k})}
\,\,,  \quad f^{m_1\cdots m_{k}}_{i_1\cdots i_{k}} \in \mc V\,,
\end{equation}
and it has a (super)commutative product given by the wedge product $\wedge$.
We have a natural $\mb Z_+$-grading 
$\tilde\Omega^\bullet=\bigoplus_{k\in\mb Z_+}\tilde\Omega^k$
defined by saying that elements in $\mc V$ have degree 0,
while the generators $\delta u_i^{(n)}$ have degree 1.
Hence $\tilde\Omega^k$ is a free module over $\mc V$ with basis given by the elements
$\delta u_{i_1}^{(m_1)}\wedge\cdots\wedge\delta u_{i_{k}}^{(m_{k})}$,
with $(m_1,i_1)>\dots>(m_{k},i_{k})$ (with respect to the lexicographic order).
In particular $\tilde\Omega^0=\mc V$ 
and $\tilde\Omega^1=\bigoplus_{i\in I,n\in\mb Z_+}\mc V\delta u_i^{(n)}$.
Notice that there is a natural $\mc V$-linear pairing $\tilde\Omega^1\times\mf g\to\mc V$
defined on generators by 
$\big(\delta u_i^{(m)},\frac{\partial}{\partial u_j^{(n)}}\big)=\delta_{i,j}\delta_{m,n}$,
and extended to $\tilde\Omega^1\times\mf g$ by $\mc V$-bilinearity.

We let $\delta$ be an odd derivation of degree 1 of $\tilde\Omega^\bullet$,
such that $\delta f=\sum_{i\in I,\,n\in\mb Z_+}\frac{\partial f}{\partial u_i^{(n)}}\delta u_i^{(n)}$
for $f\in\mc V$, and $\delta(\delta u_i^{(n)})=0$.
It is immediate to check that $\delta^2=0$ and that,
for $\tilde\omega\in\tilde\Omega^k$ as in \eqref{eq:apr24_1}, we have
\begin{equation}\label{eq:apr24_2}
\delta(\tilde\omega)
\,=\,
\sum_{\substack{i_r\in I, m_r\in\mb{Z}_+\\ j\in I, n\in\mb Z_+}}
\frac{\partial f^{m_1\cdots m_k}_{i_1\cdots i_k}}{\partial u_{i_j}^{(n)}}\, 
\delta u_{j}^{(n)}\wedge \delta u_{i_1}^{(m_1)}\wedge\cdots\wedge\delta u_{i_k}^{(m_k)}\,.
\end{equation}

For $X\in\mf g$ we define the \emph{contraction operator}
$\iota_X:\,\tilde\Omega^\bullet\to\tilde\Omega^\bullet$,
as an odd derivation of $\tilde\Omega^\bullet$ of degree -1,
such that $\iota_X(f)=0$ for $f\in\mc V$,
and $\iota_X(\delta u_i^{(n)})=X(u_i^{(n)})$.
If $X\in\mf g$ is as in \eqref{2006_X}
and $\tilde\omega\in\tilde\Omega^k$ is as in \eqref{eq:apr24_1}, we have
\begin{equation}\label{eq:apr24_3}
\iota_X(\tilde\omega)
\,=\,
\sum_{i_r\in I, m_r\in\mb{Z}_+} \sum_{q=1}^{k} (-1)^{q+1}
f^{m_1\cdots m_{k}}_{i_1\cdots i_{k}} P_{i_q,m_q}\, 
\delta u_{i_1}^{(m_1)}\wedge\stackrel{q}{\check{\cdots}}\wedge\delta u_{i_{k}}^{(m_{k})}\,.
\end{equation}
In particular, for $f\in\mc V$ we have
\begin{equation}\label{eq:apr24_8}
\iota_X(\delta f)\,=\,X(f)\,.
\end{equation}
It is easy to check that the operators $\iota_X,\,X\in\mf g$, form an abelian
(purely odd) subalgebra of the Lie superalgebra $\Der\,\tilde\Omega^\bullet$, namely
\begin{equation}\label{eq:apr24_5}
[\iota_X,\iota_Y]=\iota_X\circ\iota_Y+\iota_Y\circ\iota_X=0\,.
\end{equation}

The \emph{Lie derivative} $L_X$ along $X\in\mf g$
is defined as a degree 0 derivation of the superalgebra $\tilde\Omega^\bullet$,
commuting with $\delta$,
and such that
\begin{equation}\label{eq:vsep5}
L_X(f) = X(f)\,\,\,\,\,\,
\text{ for }\,\,\,
f\in\tilde\Omega^0\,.
\end{equation}
One easily check (on generators) Cartan's formula (cf. \eqref{eq:3.1}):
\begin{equation}\label{eq:apr24_6}
L_X\,=\,[\delta,\iota_X] = \delta\circ\iota_X+\iota_X\circ\delta\,.
\end{equation}
We next prove the following:
\begin{equation}\label{eq:apr24_9}
[\iota_X,L_Y] = \iota_X\circ L_Y-L_Y\circ \iota_X = \iota_{[X,Y]}\,.
\end{equation}
It is clear by degree considerations that both sides of \eqref{eq:apr24_9}
act as zero on $\tilde\Omega^0=\mc V$.
Moreover, it follows by \eqref{eq:apr24_8} that
$[\iota_X,L_Y](\delta f)=\iota_X \delta \iota_Y \delta f -\iota_Y\delta \iota_X\delta f
=X(Y(f))-Y(X(f))=[X,Y](f)=\iota_{[X,Y]}(\delta f)$ for every $f\in\mc V$.
Equation \eqref{eq:apr24_9} then follows by the fact that both sides are even derivations
of the wedge product in $\tilde\Omega$.
Finally, as immediate consequence of equation \eqref{eq:apr24_9}, we get that
\begin{equation}\label{eq:apr24_10}
[L_X,L_Y] = L_X\circ L_Y-L_Y\circ L_X = L_{[X,Y]}\,.
\end{equation}
Thus, $\tilde\Omega^\bullet$ is a $\mf g$-complex,
$\hat{\mf g}$ acting on $\tilde\Omega^\bullet$ by derivations.

Note that the action of $\partial$ on $\mc V$ extends to a degree 0 derivation 
of $\tilde\Omega^\bullet$, such that
\begin{equation}\label{eq:vsep5_2}
\partial(\delta u_i^{(n)})\,=\,\delta u_i^{(n+1)}\,,\,\, i\in I,\,n\in\mb Z_+\,.
\end{equation}
This derivation commutes with $\delta$, hence we can
consider the corresponding 
\emph{reduced de Rham complex} $\Omega^\bullet=\Omega^\bullet(\mc V)$,
usually called the \emph{complex of variational calculus}:
$$
\Omega^\bullet=\bigoplus_{k\in\mb Z_+}\Omega^k
\,\,,\,\,\,\,
\Omega^k=\tilde\Omega^k/\partial\tilde\Omega^k\,,
$$
with the induced action of $\delta$.
%
With an abuse of notation, we denote by $\delta$ and, for $X\in\mf g^\partial$, by $\iota_X,\,L_X$,
the maps induced on the quotient space $\Omega^k$ by the corresponding maps 
on $\tilde\Omega^k$.
Obviously, $\Omega^\bullet$ is a $\mf g^\partial$-complex.

\vspace{3ex}
\subsection{Isomorphism of the cohomology $\mf g^\partial$-complexes $\Omega^\bullet$ and $\Gamma^\bullet$.}~~
\label{sec:5.4}
\begin{proposition}\label{prop:ago22}
Let $\mc V$ be an algebra of differentiable functions.
Consider the Lie conformal algebra $A=\oplus_{i\in I}\mb F[\partial]u_i$
with the zero $\lambda$-bracket.
Then $\mc V$ is a module over the Lie conformal algebra $A$,
with the $\lambda$-action given by 
\begin{equation}\label{eq:july19_2_new}
{u_i}_\lambda f=
\sum_{n \in \mb{Z}_+} \lambda^n \frac{\partial f}{\partial u_i ^{(n)}} \, .
\end{equation}
Moreover, the $\lambda$-action of $A$ on $\mc V$ is by derivations 
of the associative product in $\mc V$.
\end{proposition}
\begin{proof}
The fact that $\mc V$ is an $A$-module follows from the 
definition of an algebra of differentiable functions. 
The second statement is clear as well.
\end{proof}
Let $\tilde\Gamma^\bullet=\tilde\Gamma^\bullet(A,\mc V)$
and $\Gamma^\bullet=\Gamma^\bullet(A,\mc V)$
be the basic and reduced Lie conformal algebra cohomology complexes
for the $A$-module $\mc V$, defined in Proposition \ref{prop:ago22}.
Thus, to every algebra of differentiable functions $\mc V$
we can associate two apparently unrelated
types of cohomology complexes:
the basic and reduced de Rham cohomology complexes, 
$\tilde\Omega^\bullet(\mc V)$ and $\Omega^\bullet(\mc V)$,
defined in Section \ref{sec:5.2},
and the basic and reduced 
Lie coformal algebra cohomology complexes 
$\tilde\Gamma^\bullet(A,\mc V)$ and $\Gamma^\bullet(A,\mc V)$,
defined in Section \ref{sec:1.1},
for the Lie conformal algebra $A=\bigoplus_{i\in I}\mb F[\partial]u_i$, with the zero $\lambda$-bracket,
acting on $\mc V$, with the $\lambda$-action given by \eqref{eq:july19_2_new}.
We are going to prove that, in fact, these complexes are isomorphic,
and all the related structures (such as exterior products, contraction operators,
Lie derivatives,...) correspond via this isomorphism.

We denote, as in Section \ref{sec:3.2}, by
$\tilde\Gamma_\bullet=\tilde\Gamma_\bullet(A,\mc V)$
(resp. $\Gamma_\bullet=\Gamma_\bullet(A,\mc V)$)
the basic (resp. reduced) space of chains of $A$ with coefficients in $\mc V$.
Recall from Secton \ref{sec:3.2.5} that $\Pi\tilde\Gamma_1$ 
is identified with the space 
$(A\otimes \mc V[[x]])\big/(\partial\otimes1+1\otimes\partial_x)(A\otimes \mc V[[x]])$,
and it carries a Lie algebra structure
given by the Lie bracket \eqref{eq:nov23_2},
which in this case takes the form, for $i,j\in I$ 
and 
$P(x)=\sum_{m\in\mb Z_+}\frac1{m!}P_mx^m,\,Q(x)=\sum_{n\in\mb Z_+}\frac1{n!}Q_nx^n\,\in \mc V[[x]]$:
\begin{equation}\label{eq:sat_12}
[u_i\otimes P(x),u_j\otimes Q(x)]
=
-u_i\otimes \sum_{n\in\mb Z_+}Q_n\frac{\partial P(x)}{\partial u_j^{(n)}}
+u_j\otimes \sum_{m\in\mb Z_+}P_m\frac{\partial Q(x)}{\partial u_i^{(m)}}\,.
\end{equation}
Moreover, $\partial$ acts on $\tilde\Gamma_1$ by \eqref{eq:nov22_1}.
Its kernel $\Pi \Gamma_1$ 
consists of elements of the form
\begin{equation}\label{eq:sat_11}
\sum_{i\in I}u_i\otimes e^{x\partial}P_i
\,\,,\,\,\,\,
\text{ where } P_i\in\mc V\,,
\end{equation}
and it is a Lie subalgebra of $\Pi \tilde\Gamma_1$.
We also denote, as in Section \ref{sec:5.1}, 
by $\mf g$ the Lie algebra of all vector fields \eqref{2006_X}
acting on $\mc V$,
and by $\mf g^\partial\subset\mf g$ the Lie subalgebra  of evolutionary vector fields \eqref{2006_X2}.

\begin{proposition}\label{prop:voct3}
The map 
$\Phi_1:\,\Pi \tilde\Gamma_1\to\mf g$, which maps
\begin{equation}\label{eq:sept3_1}
\xi=\sum_{i\in I}u_i\otimes P_i(x)
=\sum_{i\in I,n\in\mb Z_+}\frac1{n!} u_i\otimes P_{i,n} x^n\,\in\,\tilde\Gamma_1\,,
\end{equation}
to
\begin{equation}\label{eq:ago22_1}
\Phi_1(\xi)\,=\,\sum_{i\in I,\,n\in\mb Z_+}P_{i,n}\frac{\partial}{\partial u_i^{(n)}}\,,
\end{equation}
is a Lie algebra isomorphism.
Moreover, the image of the space of reduced 1-chains via $\Phi_1$ 
is the space of evolutionary vector fields. Hence we have the induced
Lie algebra isomorphism 
$\Phi_1\,:\,\,\Pi \Gamma_1 \stackrel{\sim}{\to} \mf g^\partial$.
\end{proposition}
\begin{proof}
Clearly, $\Phi_1$ is a bijective map, and, by \eqref{eq:sat_11},
$\Phi_1(\Gamma_1)=\mf g^\partial$.
Hence we only need to check $\Phi_1$ is a Lie algebra homomorphism.
This is immediate from equation \eqref{eq:sat_12}.
\end{proof}
\begin{theorem}\label{th:july24}
The map $\Phi^\bullet:\,\tilde\Gamma^\bullet\to\tilde\Omega^\bullet$, 
such that $\Phi^0=\id|_{\mc V}$ and, for $k\geq1$,
$\Phi^k:\,\tilde\Gamma^k\to\tilde\Omega^k$
is given by
\begin{equation}\label{eq:july24_9}
\Phi^k(\tilde\gamma)\,=\,
\frac1{k!} \sum_{i_r\in I, m_r\in\mb{Z}_+}
f^{m_1\cdots m_{k}}_{i_1\cdots i_{k}}\, 
\delta u_{i_1}^{(m_1)}\wedge\cdots\wedge\delta u_{i_{k}}^{(m_{k})}\,,
\end{equation}
where $f^{m_1\cdots m_{k}}_{i_1\cdots i_{k}}\in\mc V$ is the coefficient 
of $\lambda_1^{m_1}\cdots\lambda_{k}^{m_{k}}$ 
in $\tilde\gamma_{\lambda_1,\dots,\lambda_{k}}(u_{i_1},\dots,u_{i_{k}})$,
is an isomorphism of superalgebras,
and an isomorphism of $\mf g$-complexes,
(once we identify the Lie algebras $\mf g$ and $\Pi\tilde\Gamma_1$ via $\Phi_1$,
as in Proposition \ref{prop:voct3}).

Moreover, $\Phi^\bullet$ commutes with the action of $\partial$,
hence it induces an isomorphism of the corresponding reduced $\mf g^\partial$-complexes:
$\Phi^\bullet:\,\Gamma^\bullet\stackrel{\sim}{\to}\Omega^\bullet$.
\end{theorem}
\begin{proof}
Notice that since, by assumption, $I$ is a finite index set,
the RHS of \eqref{eq:july24_9} is a finite sum,
so that $\Phi^k(\tilde\Gamma^k)\subset\tilde\Omega^k$.
By the sesquilinearity and skew-symmetry conditions A1. and A2. in Section \ref{sec:1.1},
elements $\tilde\gamma\in\tilde\Gamma^k$ are uniquely determined by
the collection of polynomials 
$\tilde\gamma_{\lambda_1,\cdots,\lambda_{k}}(u_{i_1},\cdots,u_{i_{k}})
=\sum_{m_r\in\mb Z_+}
f^{m_1\cdots m_{k}}_{i_1\cdots i_{k}}\lambda_1^{m_1}\cdots\lambda_{k}^{m_{k}}$,
which are skew-symmetric with respect to simultaneous permutation 
of the variables $\lambda_r$ and the indices $i_r$.
We want to check that $\Phi^k$ is a bijective linear map from $\tilde\Gamma^k$ to $\tilde\Omega^k$.
In fact,
denote by $\Psi^k:\,\tilde\Omega^k\to\tilde\Gamma^k$ the linear map which to $\tilde\omega$
as in \eqref{eq:apr24_1} associates the $k$-cochain $\Psi^k(\tilde\omega)$, such that
$$
\Psi^k(\tilde\omega)_{\lambda_1,\dots,\lambda_{k}}
(u_{i_1},\dots,u_{i_{k}})
= \sum_{m_r\in\mb{Z}_+}
\langle f\rangle^{m_1\cdots m_{k}}_{i_1\cdots i_{k}}\, 
\lambda_1^{m_1}\cdots\lambda_{k}^{m_{k}}\,,
$$
where $\langle f\rangle$ denotes the skew-symmetrization of $f$:
$$
\langle f\rangle^{m_1\cdots m_{k}}_{i_1\cdots i_{k}}
=\sum_{\sigma} \text{sign}(\sigma) 
f^{m_{\sigma(1)}\cdots m_{\sigma(k)}}_{i_{\sigma(1)}\cdots i_{\sigma(k)}}\,,
$$
and $\Psi^k(\tilde\omega)$ is extended to $A^{\otimes k}$ by the sesquilinearity condition A1.
It is straightforward to check that $\Psi^k(\tilde\omega)$ is indeed a $k$-cochain,
and that the maps $\Phi^k$ and $\Psi^k$ are inverse to each other.
This proves that $\Phi^\bullet$ is a bijective map.

Next, let us prove that $\Phi^\bullet$ is an associative superalgebra homomorphism.
Let $\tilde\alpha\in\tilde\Gamma^h,\,\tilde\beta\in\tilde\Gamma^k$ and
let $\alpha^{m_1,\cdots,m_{h}}_{i_1,\cdots,i_{h}}$
be the coefficient of $\lambda_1^{m_1}\cdots\lambda_{h}^{m_{h}}$
in $\tilde\alpha_{\lambda_1,\cdots,\lambda_{h}}(u_{i_1},\cdots,u_{i_{h}})$,
and let $\beta^{n_1,\cdots,n_{k}}_{j_1,\cdots,j_{k}}$
be the coefficient of $\lambda_1^{n_1}\cdots\lambda_{k}^{n_{k}}$
in $\tilde\beta_{\lambda_1,\cdots,\lambda_{k}}(u_{j_1},\cdots,u_{j_{k}})$.
By \eqref{eq:sfor_2},
the coefficient of $\lambda_1^{m_1}\cdots\lambda_{h+k}^{m_{h+k}}$ in 
$(\tilde\alpha\wedge\tilde\beta)_{\lambda_1,\cdots,\lambda_{h+k}}(u_{i_1},\cdots,u_{i_{h+k}})$
is
$$
\sum_{\sigma\in S_{h+k}} \frac{\text{sign}(\sigma)}{h!k!}
\alpha^{m_{\sigma(1)},\cdots,m_{\sigma(h)}}_{i_{\sigma(1)},\cdots,i_{\sigma(h)}}
\beta^{m_{\sigma(h+1)},\cdots,m_{\sigma(h+k)}}_{i_{\sigma(h+1)},\cdots,i_{\sigma(h+k)}}\,.
$$
The identity 
$\Phi^{h+k}(\tilde\alpha\wedge\tilde\beta)=\Phi^h(\tilde\alpha)\wedge\Phi^k(\tilde\beta)$
follows by the definition \eqref{eq:july24_9} of $\Phi^k$.

Let $\tilde\gamma\in\tilde\Gamma^k$,
and denote by $f^{m_1\cdots m_{k}}_{i_1\cdots i_{k}}\in\mc V$ 
the coefficient of $\lambda_1^{m_1}\cdots\lambda_{k}^{m_{k}}$ 
in $\tilde\gamma_{\lambda_1,\dots,\lambda_{k}}(u_{i_1},\dots,u_{i_{k}})$.
We want to prove that $\Phi^{k+1}(\delta\tilde\gamma)=\delta\Phi^k(\tilde\gamma)$.
By assumption, the $\lambda$-bracket on $A$ is zero,
and the $\lambda$-action of $A$ on $\mc V$ is given by 
\eqref{eq:july19_2_new}.
Hence, recalling \eqref{eq:july24_7},
the coefficient of $\lambda_1^{m_1}\cdots\lambda_{k+1}^{m_{k+1}}$ in 
$(\delta\tilde\gamma)_{\lambda_1,\cdots,\lambda_{k+1}}(u_{i_1},\cdots,u_{i_{k+1}})$
is
$$
\sum_{r=1}^{k+1}(-1)^{r+1} 
\frac{\partial 
f^{m_1\stackrel{r}{\check{\cdots}} m_{k+1}}_{i_1\stackrel{r}{\check{\cdots}} i_{k+1}}}
{\partial u_{i_r}^{(m_r)}}\,.
$$
It follows that
\begin{eqnarray*}
\Phi^{k+1}(\delta\tilde\gamma)
&=& \frac1{(k+1)!} \sum_{i_r\in I, m_r\in\mb{Z}_+}
\sum_{q=1}^{k+1}(-1)^{q+1} 
\frac{\partial 
f^{m_1\stackrel{q}{\check{\cdots}} m_{k+1}}_{i_1\stackrel{q}{\check{\cdots}} i_{k+1}}}
{\partial u_{i_q}^{(m_q)}}
\delta u_{i_1}^{(m_1)}\wedge\cdots\wedge\delta u_{i_{k+1}}^{(m_{k+1})} \\
&=& \frac1{k!} \sum_{i_r\in I, m_r\in\mb{Z}_+}
\frac{\partial 
f^{m_1\cdots m_{k}}_{i_1\cdots i_{k}}}
{\partial u_{i_0}^{(m_0)}}
\delta u_{i_0}^{(m_0)}\wedge\cdots\wedge\delta u_{i_{k}}^{(m_{k})}
=
\delta\Phi^k(\tilde\gamma)\,,
\end{eqnarray*}
thus proving the claim.

Similarly, \,\,\,\,
the coefficient \,\,\,\,
of $\lambda_1^{m_1}\cdots\lambda_{k}^{m_{k}}$ \,\,\,\,
in $(\partial\tilde\gamma)_{\lambda_1,\cdots,\lambda_{k}}(u_{i_1},\cdots,u_{i_{k}})$
\,\,\,\, is
$\partial^M
f^{m_1\cdots m_{k}}_{i_1\cdots i_{k}}$\\ 
$+\sum_{r=1}^{k}f^{m_1\cdots m_r-1\cdots m_{k}}_{i_1\cdots i_{k}}$,
so that
\begin{eqnarray*}
\Phi^k(\partial\tilde\gamma)
&=& \frac1{k!} \sum_{i_r\in I, m_r\in\mb{Z}_+}
\Big(
\partial^M f^{m_1\cdots m_{k}}_{i_1\cdots i_{k}}
\delta u_{i_1}^{(m_1)}\wedge\cdots\wedge\delta u_{i_{k}}^{(m_{k})} \\
&& +f^{m_1\cdots m_{k}}_{i_1\cdots i_{k}}
\sum_{q=1}^{k} 
\delta u_{i_1}^{(m_1)}\wedge\cdots\wedge
\delta u_{i_{q}}^{(m_{q}+1)}\wedge\cdots\wedge\delta u_{i_{k}}^{(m_{k})} \Big) 
=
\partial\Phi^k(\tilde\gamma)\,.
\end{eqnarray*}
This proves that $\Phi^\bullet$ is compatible with the action of $\partial$.

Finally, we prove that $\Phi^\bullet$ is compatible with the contraction operators.
Let $\tilde\gamma\in\tilde\Gamma^k$ be as in the statement of the theorem,
and let $\xi\in\tilde\Gamma_1$ be as in \eqref{eq:sept3_1}.
By equation \eqref{eq:nov22_2}, we have the following formula for 
the contraction operator $\iota_\xi$,
$$
(\iota_\xi\tilde\gamma)_{\lambda_2,\cdots,\lambda_{k}}(u_{i_2},\cdots,u_{i_{k}})
=
\sum_{i_1\in I}
\big\langle P_{i_1}(x_1),\tilde\gamma_{\lambda_1,\lambda_2,\cdots,\lambda_{k}}
(u_{i_1},u_{i_2},\cdots,u_{i_{k}})\big\rangle\,,
$$
where $\langle\,,\,\rangle$ denotes the contraction of $x_1$ with $\lambda_1$
defined in \eqref{eq:nov22_3}.
Hence, the coefficient 
of $\lambda_2^{m_2}\cdots\lambda_{k}^{m_{k}}$ in 
$(\iota_\xi\tilde\gamma)_{\lambda_2,\cdots,\lambda_{k}}(u_{i_2},\cdots,u_{i_{k}})$
is
$$
\sum_{i_1\in I,m_1\in\mb Z_+} 
P_{i_1,m_1} f^{m_1 m_2\cdots m_{k}}_{i_1 i_2\cdots i_{k}}\,.
$$
It follows that
$$
\Phi^{k-1}(\iota_\xi(\tilde\gamma))
= \frac1{(k-1)!} \sum_{i_r\in I, m_r\in\mb{Z}_+}
P_{i_1,m_1} f^{m_1 m_2\cdots m_{k}}_{i_1 i_2\cdots i_{k}}
\delta u_{i_2}^{(m_2)}\wedge\cdots\wedge\delta u_{i_{k}}^{(m_{k})}\,,
$$
which, recalling \eqref{eq:apr24_3} and \eqref{eq:ago22_1}, 
is the same as $\iota_{\Phi_1(\xi)}(\Phi^{k}(\tilde\gamma))$.
This completes the proof of the theorem.
\end{proof}

\vspace{3ex}
\subsection{An explicit construction of the $\mf g^\partial$-complex of variational calculus.}~~
\label{sec:v5.4}
Let $\mc V$ be an algebra of differentiable functions in the variables $\{u_i\}_{i\in I}$,
let $A=\bigoplus_{i\in I}\mb F[\partial]u_i$ be the free $\mb F[\partial]$-module
of rank $\ell$, considered as a Lie conformal algebra with the zero $\lambda$-bracket,
and consider the $A$-module structure on $\mc V$,
with the $\lambda$-action given by \eqref{eq:july19_2_new}.
By Theorem \ref{th:july24},
the $\mf g^\partial$-complex of variational calculus $\Omega^\bullet(\mc V)$
is isomorphic to the $\Pi\Gamma_1$-complex $\Gamma^\bullet(A,\mc V)$.
Furthermore, due to Theorems \ref{th:red} and \ref{th:ago21},
the $\Pi\Gamma_1$-complex $\Gamma(A,\mc V)$
is isomorphic to the $\Pi C_1$-complex
$C^\bullet(A,\mc V)=\bigoplus_{k\in\mb Z_+}C^k$,
which is explicitly described in Sections \ref{sec:1.3_b} and \ref{sec:3.5}.

In this section we use this isomorphism to describe explicitly 
the $\Pi C_1\simeq\mf g^\partial$-complex
of variational calculus $C^\bullet(A,\mc V)\simeq\Omega^\bullet(\mc V)$,
both in terms of ``poly-symbols'',
and in terms of skew-symmetric ``poly-differential operators''.
We shall identify these two complexes via this isomorphism.

We start by describing all vector spaces $\Omega^k$ and the maps 
$d:\,\Omega^k\to \Omega^{k+1},\,k\in\mb Z_+$.
First, we have
\begin{equation}\label{eq:vsept28_1}
\Omega^0\,=\,\mc V/\partial\mc V\,.
\end{equation}
Next, $\Omega^1=\Hom_{\mb F[\partial]}(A,\mc V)$,
hence we have a canonical identification
\begin{equation}\label{eq:vsept28_2}
\Omega^1\,=\,\mc V^{\oplus \ell}\,.
\end{equation}
Comparing \eqref{eq:d0} and \eqref{eq:july19_2_new},
we see that $d:\, \Omega^0\to \Omega^1$ is given by the variational derivative:
\begin{equation}\label{eq:vsept28_3}
d\tint f\,=\, \frac{\delta f}{\delta u}\,.
\end{equation}

For arbitrary $k\geq1$, the space $\Omega^k$ can be identified with the space
of $k$-\emph{symbols} in $u_i,\,i\in I$.
By definition, a $k$-symbol is a collection of expressions of the form
\begin{equation}\label{eq:vsept28_4}
\big\{{u_{i_1}}_{\lambda_1} {u_{i_2}}_{\lambda_2} \cdots {u_{i_{k-1}}}_{\lambda_{k-1}} u_{i_k}\big\}
\,\in\,\mb F[\lambda_1,\dots,\lambda_{k-1}]\otimes\mc V\,,
\end{equation}
where $i_1,\dots,i_k\in I$,
satisfying the following skew-symmetry property:
\begin{equation}\label{eq:vsept28_4_2}
\big\{{u_{i_1}}_{\lambda_1} {u_{i_2}}_{\lambda_2} \cdots {u_{i_{k-1}}}_{\lambda_{k-1}} u_{i_k}\big\}
=
\text{sign}(\sigma)
\big\{{u_{i_{\sigma(1)}}}_{\lambda_{\sigma(1)}} 
\cdots {u_{i_{\sigma(k-1)}}}_{\lambda_{\sigma(k-1)}} u_{i_{\sigma(k)}}\big\}
\,,
\end{equation}
for every permutation $\sigma\in S_k$,
where $\lambda_k$ is replaced, if it occurs in the RHS, 
by $\lambda_k^\dagger=-\sum_{j=1}^{k-1}\lambda_j-\partial$,
with $\partial$ acting from the left.
Clearly, by sesquilinearity, for $k\geq1$, the space $\Omega^k=C^k$ of $k$-$\lambda$-brackets
is one-to-one correspondence with the space of $k$-symbols.

For example, the space of 1-symbols is the same as $\mc V^{\oplus \ell}$.
A 2-symbol is a collection of elements $\big\{{u_i}_\lambda u_j\big\}\in\mb F[\lambda]\otimes\mc V$, 
for $i,j\in I$, such that
$$
\big\{{u_i}_\lambda u_j\big\}
=
-\big\{{u_j}_{-\lambda-\partial} u_i\big\}\,.
$$
A 3-symbol is a collection of elements 
$\big\{{u_i}_\lambda {u_j}_\mu u_k\big\}\in\mb F[\lambda,\mu]\otimes\mc V$, 
for $i,j,k\in I$, such that
$$
\big\{{u_i}_\lambda {u_j}_\mu u_k\big\}
=
-\big\{{u_j}_\mu {u_i}_\lambda u_k\big\}
=
-\big\{{u_i}_\lambda {u_k}_{-\lambda-\mu-\partial} u_j\big\}\,,
$$
and similarly for $k>3$.

Comparing \eqref{eq:d>} and \eqref{eq:july19_2_new}
we see that, if $F\in\mc V^{\oplus \ell}$, its differential $dF$ corresponds 
to the following 2-symbol:
\begin{equation}\label{eq:vsept28_5}
\big\{{u_i}_{\lambda} u_j\big\}
=
\sum_{n\in\mb Z_+}
\Big(
\lambda^n\frac{\partial F_j}{\partial u_i^{(n)}} -(-\lambda-\partial)^n\frac{\partial F_i}{\partial u_j^{(n)}}
\Big)
=
(D_F)_{ji}(\lambda)-(D_F^*)_{ji}(\lambda)\,,
\end{equation}
where $D_F$ is the Frechet derivative defined by \eqref{eq:dic15_2}.
More generally, the differential of a $k$-symbol for $k\geq1$
is given by the following formula:
\begin{eqnarray}\label{eq:dic15_5}
&\displaystyle{
d\Big(
\{{u_{i_1}}_{\lambda_1}\cdots {u_{i_{k-1}}}_{\lambda_{k-1}} u_{i_{k}}\}
\Big)_{i_1,\dots,i_k\in I}
\,=\,
\Big(
\sum_{n\in\mb Z_+} \sum_{s=1}^k (-1)^{s+1}  \lambda_s^n \frac{\partial}{\partial u_{i_s}^{(n)}}
\big\{{u_{i_1}}_{\lambda_1}\stackrel{s}{\check{\cdots}}{u_{i_k}}_{\lambda_{k}} u_{i_{k+1}}\big\} 
}\nonumber\\
&\displaystyle{
+(-1)^k \sum_{n\in\mb Z_+} \Big(-\sum_{j=1}^k\lambda_j-\partial\Big)^n \frac{\partial}{\partial u_{i_{k+1}}^{(n)}}
\big\{{u_{i_1}}_{\lambda_1}\cdots {u_{i_{k-1}}}_{\lambda_{k-1}} {u_{i_k}}\big\} 
\Big)_{i_1,\dots,i_{k+1}\in I} \,.
}
\end{eqnarray}

Provided that $\mc V$ is an algebra of differentiable functions extension of $R_\ell$,
an equivalent language is that of skew-symmetric poly-differential operators.
By definition, a $k$-\emph{differential operator}
is an $\mb F$-linear map $S:\,(\mc V^{\ell})^{k}\to\mc V/\partial\mc V$, of the form
\begin{equation}\label{eq:dic15_1}
S(P^1,\cdots,P^k)
\,=\, 
\int \sum_{\substack{n_1,\cdots,n_{k}\in\mb Z_+ \\ i_1,\cdots,i_{k}\in I}}
f_{i_1,\cdots,i_k}^{n_1,\cdots,n_{k}}
(\partial^{n_1}P^1_{i_1})\cdots(\partial^{n_{k}}P^{k}_{i_{k}})\,.
\end{equation}
The operator $S$ is called skew-symmetric if
$$
\int S(P^1,\cdots,P^k)=\text{sign}(\sigma)\int S(P^{\sigma(1)},\cdots,P^{\sigma(k)})\,,
$$
for every $P^1,\cdots,P^k\in\mc V^\ell$ and every permutation $\sigma\in S_k$.
Given a $k$-symbol
\begin{equation}\label{eq:dic15_8}
\big\{{u_{i_1}}_{\lambda_1} \cdots {u_{i_{k-1}}}_{\lambda_{k-1}} u_{i_k}\big\}
=\sum_{n_1,\dots,n_{k-1}\in\mb Z_+}
f_{i_1,\cdots,i_{k-1},i_k}^{n_1,\cdots,n_{k-1}}\lambda_1^{n_1}\cdots\lambda_{k-1}^{n_{k-1}}
\,\,,\,\,\,\, i_1,\dots,i_k\in I\,,
\end{equation}
where $f_{i_1,\cdots,i_k}^{n_1,\cdots,n_{k-1}}\in\mc V$,
we associate to it the following poly-differential operator:
$S:\,(\mc V^{\ell})^{k}\to\mc V/\partial\mc V$, is 
\begin{equation}\label{eq:vsept28_6}
S(P^1,\cdots,P^k)
\,=\, 
\int \sum_{\substack{n_1,\cdots,n_{k-1}\in\mb Z_+ \\ i_1,\cdots,i_{k}\in I}}
f_{i_1,\cdots,i_{k-1},i_k}^{n_1,\cdots,n_{k-1}}
(\partial^{n_1}P^1_{i_1})\cdots(\partial^{n_{k-1}}P^{k-1}_{i_{k-1}})P^k_{i_k}\,.
\end{equation}
Clearly, the skew-symmetry property of the $k$-symbol
is translated to the skew-symmetry of the poly-differential operator.
Conversely, integrating by parts, any $k$-differential operator
can be written in the form \eqref{eq:vsept28_6}.
Thus we have a surjective map $\Xi$ form the space of $k$-symbols
to the space of skew-symmetric $k$-differential operators.
Provided that $\mc V$ is an algebra of differentiable functions extension of $R_\ell$,
by Lemma \ref{lem:fine}(c), the $k$-differential operator $S$ can be written uniquely
in the form \eqref{eq:vsept28_6}.
Hence, the map $\Xi$ is an isomorphism.

Note that the space of 1-differential operators $S:\,\mc V^\ell\to\mc V/\partial\mc V$
can be canonically identified with the space $\Omega^1=\mc V^{\oplus\ell}$.
Explicitly, to the 1-differential operator $S(P)=\int\sum_{i\in I,n\in\mb Z_+}f^n_i\partial^nP_i$,
we associate:
\begin{equation}\label{eq:dic15_15}
\Big(\sum_{n\in\mb Z_+}(-\partial)^nf^n_i\Big)_{i\in I}\,\in\,\mc V^{\oplus\ell}\,.
\end{equation}

We can write down the expression of the differential $d:\,\Omega^k\to \Omega^{k+1}$ in terms of 
poly-differential operators.
First, if $F\in\Omega^1=\mc V^{\oplus\ell}$, the 2-differential operator corresponding 
to $dF\in\Omega^2$ is obtained by looking at equation \eqref{eq:vsept28_5}:
\begin{equation}\label{eq:dic15_4}
dF(P,Q)
\,=\,
\int \sum_{i\in I}\big(
Q_i X_P(F_i)-P_iX_Q(F_i)
\big)
\,=\,
\int \sum_{i,j\in I}\big(
Q_i D_F(\partial)_{ij}P_j - P_i D_F(\partial)_{ij}Q_j 
\big)\,,
\end{equation}
where $X_P$ denotes the evolutionary vector field associated to $P\in\mc V^\ell$,
defined in \eqref{2006_X2},
and $D_F(\partial)$ is the Frechet derivative \eqref{eq:dic15_2}.
Next, if $S:\,(\mc V^\ell)^k\to\mc V/\partial\mc V$ is a skew-symmetric $k$-differential operator,
its differential $dS$, obtained by looking at \eqref{eq:dic15_5},
is the following $k+1$-differential operator:
\begin{equation}\label{eq:dic15_6}
dS(P^1,\cdots,P^{k+1})
\,=\,
\sum_{s=1}^{k+1}(-1)^{s+1}
\big(X_{P^s}S\big)(P^1,\stackrel{s}{\check{\cdots}},P^{k+1})\,.
\end{equation}
In the above formula, if $S$ is as in \eqref{eq:dic15_1},
$X_PS$ denotes the $k$-differential operator obtained from $S$
by replacing the coefficients $f_{i_1,\cdots,i_k}^{n_1,\cdots,n_{k}}$
by $X_P(f_{i_1,\cdots,i_k}^{n_1,\cdots,n_{k}})$.

\begin{remark}
For $k\geq2$, a $k$-differential operator can also be understood as a map
$S:\,(\mc V^\ell)^{k-1}\to \mc V^{\oplus\ell}$ of the following form:
\begin{equation}\label{eq:dic15_7}
S(P^1,\cdots,P^{k-1})_{i_k}
\,=\, 
\sum_{\substack{n_1,\cdots,n_{k-1}\in\mb Z_+ \\ i_1,\cdots,i_{k-1}\in I}}
f_{i_1,\cdots,i_{k-1},i_k}^{n_1,\cdots,n_{k-1}}
(\partial^{n_1}P^1_{i_1})\cdots(\partial^{n_{k-1}}P^{k-1}_{i_{k-1}})\,.
\end{equation}
This corresponds to the $k$-symbol \eqref{eq:dic15_8} in the obvious way.
With this notation, the differential $dS$ is the following map $(\mc V^\ell)^k\to\mc V^{\oplus\ell}$:
\begin{eqnarray}\label{eq:dic15_14}
&\displaystyle{
dS(P^1,\cdots,P^{k})_{i}
= 
\sum_{s=1}^k(-1)^{s+1}(X_{P^s}S)(P^1,\stackrel{s}{\check{\cdots}},P^{k})_{i} 
}\\
&\displaystyle{
+ (-1)^k\sum_{j\in I,n\in\mb Z_+}(-\partial)^n
\Big(
P^k_j\frac{\partial S}{\partial u_{i}^{(n)}}(P^1,\cdots,P^{k-1})_j
\Big)\,.
}\nonumber
\end{eqnarray}
\end{remark}

Recall that the Lie algebra $\mf g^\partial\simeq\Pi C_1$
is identified with the space $\mc V^\ell$ via the map $P\mapsto X_P$,
defined in \eqref{2006_X2}.
Given $P\in\mc V^\ell$, we want to describe explicitly the action 
of the corresponding contraction operator $\iota_P$
and the Lie derivative $L_P=[d,\iota_P]$.
First, for $F\in\mc V^{\oplus\ell}=\Omega^1$, we have (cf. \eqref{eq:sat_8_1}):
\begin{equation}\label{eq:dic15_9}
\iota_P(F)
\,=\,
\tint\sum_{i\in I}P_iF_i\,\in\mc V/\partial\mc V=\Omega^0\,.
\end{equation}
Next, the contraction of a $k$-symbol for $k\geq2$ is given by the following formula
(cf. \eqref{eq:sat_8}):
\begin{equation}\label{eq:dic15_10}
\iota_P
\Big(
\big\{{u_{i_1}}_{\lambda_1} \cdots {u_{i_{k-1}}}_{\lambda_{k-1}} u_{i_k}\big\}
\Big)_{i_1,\dots,i_k\in I}
\,=\,
\Big(
\sum_{i_1\in I}
\big\{{u_{i_1}}_{\partial}{u_{i_2}}_{\lambda_2} \cdots {u_{i_{k-1}}}_{\lambda_{k-1}} u_{i_k}\big\}_\to 
P_{i_1}
\Big)_{i_2,\dots,i_k\in I}\,,
\end{equation}
where, as usual, the arrow in the RHS means that $\partial$ is moved to the right.
For $k=2$, the above formula becomes
\begin{equation}\label{eq:dic15_11}
\iota_P
\Big(
\big\{{u_{i}}_{\lambda} u_{j}\big\}
\Big)_{i,j\in I}
\,=\,
\Big(
\sum_{j\in I}
\big\{{u_{j}}_{\partial}{u_{i}}\big\}_\to P_{j}
\Big)_{i\in I}\,\in\mc V^{\oplus\ell}=\Omega^1\,.
\end{equation}
We can write the above formulas in the language of poly-differential operators.
For a $k$-differential operator $S$, we have
\begin{equation}\label{eq:dic15_12}
(\iota_{P^1}S)(P^2,\cdots,P^k)=S(P^1,P^2,\cdots,P^k)\,.
\end{equation}
For $k=2$ $\iota_{P^1}S$ is a 1-differential operator
which, by \eqref{eq:dic15_15}, is the same as an element of $\mc V^{\oplus\ell}=\Omega^1$.

\begin{remark}
In the interpretation \eqref{eq:dic15_7} of a $k$-differential operator,
the action of the contraction operator is given by
$$
(\iota_{P^1}S)(P^2,\cdots,P^{k-1})_{i_k}=S(P^1,P^2,\cdots,P^{k-1})_{i_k}\,.
$$
\end{remark}

Next, we write the formula for the Lie derivative $L_Q:\,\Omega^k\to\Omega^k$,
associated to $Q\in\mc V^\ell\simeq\mf g^\partial$, using Cartan's formula $L_Q=[\iota_Q,d]$.
Recalling \eqref{eq:vsept28_3} and \eqref{eq:dic15_9}, after integration by parts
we obtain, for $\tint f\in\Omega^0=\mc V/\partial\mc V$:
\begin{equation}\label{eq:dic16_1}
L_Q\big(\tint f\big)
\,=\,
\tint X_Q(f)\,,
\end{equation}
where $X_Q$ is the evolutionary vector field corresponding to $Q$ (cf. \eqref{2006_X2}).
Similarly, recalling \eqref{eq:vsept28_5} and \eqref{eq:dic15_11}, we obtain,
for $F\in\Omega^1=\mc V^{\oplus\ell}$:
\begin{eqnarray*}
d\iota_Q(F)
&=&
D_F(\partial)^*Q+D_Q(\partial)^*F\,,\\
\iota_Qd(F)
&=&
D_F(\partial)Q-D_F(\partial)^*Q\,,
\end{eqnarray*}
where $D_F(\partial)$ denotes the Frechet derivative \eqref{eq:dic15_2},
and $D_F(\partial)^*$ is the adjoint differential operator.
Putting the above formulas together, we get:
\begin{equation}\label{eq:dic16_2}
L_QF
\,=\,
D_F(\partial)Q+D_Q(\partial)^*F\,.
\end{equation}
For $k\geq2$, $L_Q$ acts on a $k$-symbol in $\Omega^k$ by the following formula,
which can be derived from \eqref{eq:dic15_5} and \eqref{eq:dic15_10}:
\begin{eqnarray*}
&\displaystyle{
L_Q
\{{u_{i_1}}_{\lambda_1}\cdots {u_{i_{k-1}}}_{\lambda_{k-1}} u_{i_{k}}\}
\,=\,
X_Q \{{u_{i_1}}_{\lambda_1}\cdots {u_{i_{k-1}}}_{\lambda_{k-1}} u_{i_{k}}\}
}\\
&\displaystyle{
+
\sum_{s=1}^{k-1}(-1)^{s+1}\sum_{j\in I}
\{{u_{j}}_{\lambda_s+\partial}{u_{i_1}}_{\lambda_1}
\stackrel{s}{\check{\cdots}} {u_{i_{k-1}}}_{\lambda_{k-1}} u_{i_{k}}\}_\to {D_Q(\lambda_s)}_{ji_s}
}\\
&\displaystyle{
+
(-1)^{k+1}\sum_{j\in I}
\{{u_{j}}_{\lambda_k^\dagger+\partial}{u_{i_1}}_{\lambda_1}
\cdots {u_{i_{k-2}}}_{\lambda_{k-2}} u_{i_{k-1}}\}_\to {D_Q(\lambda_k^\dagger)}_{ji_k}\,.
}
\end{eqnarray*}
In the RHS the evolutionary vector field $X_Q$ is applied to the coefficients of the $k$-symbol,
in the last two terms the arrow means, as usual, that we move $\partial$ to the right,
$D_Q(\lambda)$ denotes the Frechet derivative \eqref{eq:dic15_2}
considered as a polynomial in $\lambda$,
and, in the last term, $\lambda_k^\dagger=-\lambda_1-\cdots-\lambda_{k-1}-\partial$,
where $\partial$ is moved to the left.
This formula takes a much nicer form in the language of $k$-differential operators.
Namely we have:
\begin{equation}\label{eq:dic16_last}
(L_QS)(P^1,\cdots,P^k)
\,=\,
(X_QS)(P^1,\cdots,P^k)
+\sum_{s=1}^k S(P^1,\cdots,X_QP^s,\cdots,P^k)\,.
\end{equation}
Here $X_QS$ has the same meaning as in equation \eqref{eq:dic15_6}.
This formula can be obtained from the previous one by integration by parts.

\vspace{3ex}
\subsection{An application to the classification of symplectic differential operators.}~~
\label{sec:final}
Recall that $\mc C\subset\mc V$ denotes the subspace \eqref{eq:4.2} of constant functions.
In \cite{BDK} we prove the following:
\begin{theorem}\label{th:aliaa}
If $\mc V$ is normal, then $H^k(\Omega^\bullet,d)=\delta_{k,0}\mc C/(\mc C\cap\partial\mc V)$.
\end{theorem}

Recall  that a \emph{symplectic differential operator} (cf. \cite{D} and \cite{BDK})
is a skew-adjoint differential operator 
$S(\partial)=\big(S_{i,j}(\partial)\big)_{i,j\in I}:\,\mc V^{\ell}\to\mc V^{\oplus\ell}$,
which is closed, namely the following condition holds (cf. \eqref{eq:dic15_14}):
\begin{equation}\label{eq:dic16_symp}
{u_i}_\lambda S_{kj}(\mu)
-{u_j}_\mu S_{ki}(\lambda)
-{u_k}_{{}_{-\lambda-\mu-\partial}} S_{ji}(\lambda)\,=\,0\,,
\end{equation}
where the $\lambda$-action of $u_i$ on $\mc V$ is defined by \eqref{eq:july19_2_new}.
We have the following corollary of Theorem \ref{th:aliaa}.
\begin{corollary}\label{cor:dic16}
If $\mc V$ is a normal algebra of differentiable functions,
then any symplectic differential operator is of the form:
$S_F(\partial)=D_F(\partial)-D_F(\partial)^*$, for some $F\in\mc V^{\oplus\ell}$.
Moreover, $S_F=S_G$ if and only if $F-G=\frac{\delta f}{\delta u}$ for some $f\in\mc V$.
\end{corollary}

A skew-symmetric $k$-differential operator $S:\,(\mc V^\ell)^k\to\mc V/\partial\mc V$
is called \emph{symplectic} if it is closed, i.e. 
$$
\sum_{s=1}^{k+1}(-1)^{s+1}
\big(X_{P^s}S\big)(P^1,\stackrel{s}{\check{\cdots}},P^{k+1})
\,=\, 0\,.
$$
The following corollary of Theorem \ref{th:aliaa} is a generalization of Corollary \ref{cor:dic16}
and uses Proposition \ref{prop:dic20}
\begin{corollary}\label{cor:dic18}
If $\mc V$ is a normal algebra of differentiable functions,
then any symplectic $k$-differential operator, for $k\geq1$, is of the form:
$$
S(P^1,\cdots,P^{k})
\,=\,
\sum_{s=1}^{k}(-1)^{s+1}
\big(X_{P^s}T\big)(P^1,\stackrel{s}{\check{\cdots}},P^{k})\,,
$$
for some skew-symmetric $k-1$-differential operator $T$.
Moreover, $T$ is defined up to adding a symplectic $k-1$-differential operator.
\end{corollary}

\begin{remark}\label{rem:lastissimo}
It follows from the proof of Theorem \ref{th:aliaa}
that,  Corollaries \ref{cor:dic16} and \ref{cor:dic18}
hold in any algebra of differentiable functions $\mc V$,
provided that we are allowed to take $F$ and $T$ respectively
in an extension of $\mc V$, obtained by adding finitely many integrals
of elements of $\mc V$
(an integral of an element $f\in\mc V_{n,i}$ is a preimage $\tint du_i^{(n)} f$
of $\frac{\partial}{\partial u_i^{(n)}}$
independent on $u_j^{(m)}$ with $(m,j)>(n,i)$).
\end{remark}

\begin{remark}\label{rem:nomore}
The map $\Xi$ defined in Section \ref{sec:v5.4} may have a non-zero kernel
if $\mc V$ is not an extension of the algebra $R_\ell$,
but, of course, for any $\mc V$ the image of $\Xi$ is a $\mf g^\partial$-complex.
The 0-th term of this complex is $\mc V/\partial\mc V$
and the $k$-th term, for $k\geq1$, is the space of skew-symmetric $k$-differential operators
$S:\,(\mc V^\ell)^k\to\mc V/\partial\mc V$.
\end{remark}

\begin{remark}\label{rem:last}
Throughout this section we assumed that the number $\ell$ of variables $u_i$ is finite,
but this assumption is not essential,
and our  arguments go through with minor modifications.
This is the reason for distinguishing $\mc V^{\ell}$ from $\mc V^{\oplus\ell}$,
in order to accommodate the case $\ell=\infty$.
\end{remark}


\end{document}